\documentclass{article}
\usepackage{ifthen}
\usepackage{mdwlist}
\usepackage{amsmath,amssymb,amsfonts,amsthm}
\usepackage{bm}
\usepackage{hyperref}
\usepackage{enumitem}
\usepackage{graphicx}
\usepackage{xspace}
\usepackage{verbatim}
\usepackage{algorithm}
\usepackage{algpseudocode}
\usepackage[margin=1in]{geometry}
\usepackage{color}
\usepackage{thmtools}
\usepackage{thm-restate}
\usepackage{mathtools}
\usepackage{multicol}
\usepackage{multirow}
\usepackage[T1]{fontenc}
\usepackage{latexsym}
\usepackage{epsfig}
\usepackage{epstopdf}
\usepackage{subcaption}
\usepackage{bbm}
\graphicspath{{plots/}}

\def\a{\alpha}
\def\b{\beta}

\renewcommand{\epsilon}{\ve}
\def\ve{\varepsilon}

\def\th{\theta}

\def\rad{r}

\def\dh{d_H}
\def\tmix{t_{\text{mix}}}

\newcommand{\E}{\mbox{\bf E}}

\newcommand{\Var}{\mbox{\bf Var}}

\newcommand{\pr}[2][]{\mbox{Pr}\ifthenelse{\not\equal{}{#1}}{_{#1}}{}\!\left[#2\right]}

\newcommand{\reals}{\mathbb{R}}

\newcommand{\inftynorm}[1]{\left\lVert #1 \right\rVert_{\infty}}

\newcommand{\dtv}{d_{\mathrm {TV}}}

\newcommand{\abss}[1]{\left\lvert {#1} \right\rvert}

\newtheorem{theorem}{Theorem}

\newtheorem{proposition}{Proposition}

\newtheorem{remark}{Remark}

\newtheorem{lemma}{Lemma}
\newtheorem{claim}{Claim}
\newtheorem{corollary}{Corollary}

\newtheorem{definition}{Definition}

\newcommand{\ignore}[1]{}
\providecommand{\poly}{\operatorname*{poly}}

\newcommand{\bg}[1]{\medskip\noindent{\bf #1}}
\definecolor{Red}{rgb}{1,0,0}

\newcommand{\oldbound}[1]{{}}

\newcommand{\dm}{d_{\max}}

\title{Concentration of Multilinear Functions of the Ising Model with Applications to Network Data}

\author {
Constantinos Daskalakis\thanks{Supported by NSF CCF-1617730, CCF-1650733, and ONR N00014-12-1-0999.}\\
EECS \& CSAIL, MIT\\
\tt{costis@csail.mit.edu}
\and
Nishanth Dikkala\thanks{Supported by NSF CCF-1617730, CCF-1650733, and ONR N00014-12-1-0999.} \\
EECS \& CSAIL, MIT\\
\tt{nishanthd@csail.mit.edu}
\and
Gautam Kamath\thanks{Supported by NSF CCF-1617730, CCF-1650733, and ONR N00014-12-1-0999. Part of this work was done while the author was an intern at Microsoft Research New England.} \\
EECS \& CSAIL, MIT\\
\tt{g@csail.mit.edu}
}

\begin{document}
\maketitle

  \begin{abstract}
We prove near-tight concentration of measure for polynomial functions of the Ising model under high temperature. For any degree $d$, we show that a degree-$d$ polynomial of a $n$-spin Ising model exhibits exponential tails that scale as $\exp(-r^{2/d})$ at radius $r=\tilde{\Omega}_d(n^{d/2})$. Our concentration radius is optimal up to logarithmic factors for constant $d$, improving known results by polynomial factors in the number of spins. We demonstrate the efficacy of polynomial functions as statistics for testing the strength of interactions in social networks in both synthetic and real world data.
\end{abstract}

  \section{Introduction} \label{sec:intro}

The {\em Ising model} is a fundamental probability distribution defined in terms of a graph $G=(V,E)$ whose nodes and edges are associated with scalar parameters $(\theta_{v})_{v\in V}$ and $(\theta_{u,v})_{\{u,v\}\in E}$ respectively. The distribution samples a vector $x \in \{\pm 1\}^V$ with probability:
\begin{align}
p(x) = {\rm exp} \left(\sum_{v\in V} \theta_v x_v + \sum_{(u,v) \in E} \theta_{u,v} x_u x_v - \Phi\left(\vec \th\right) \right), \label{eq:ising model}
\end{align}  
where $\Phi\left(\vec \th\right)$ serves to provide normalization. Roughly speaking, there is a random variable $X_v$ at every node of $G$, and this
variable may be in one of two states, or spins: up ($+1$) or down ($-1$). The scalar parameter $\theta_v$ models a local
field at node $v$. The sign of $\theta_v$ represents whether this local field favors $X_v$ taking the value $+1$, i.e. the
up spin, when $\theta_v> 0$, or the value $-1$, i.e. the down spin, when $\theta_v < 0$, and its magnitude represents the
strength of the local field. Similarly, $\theta_{u,v}$ represents the direct interaction between nodes $u$ and $v$. Its sign
represents whether it favors equal spins, when $\theta_{u,v} > 0$, or opposite spins, when $\theta_{u,v} < 0$, and its magnitude
corresponds to the strength of the direct interaction. Of course, depending on the structure of $G$ and the node and edge parameters, there may be indirect interactions between nodes, which may overwhelm
local fields or direct interactions.

Many popular models, for example, the usual ferromagnetic Ising model~\cite{Ising25,Onsager44}, the Sherrington-Kirkpatrick mean field model~\cite{SherringtonK75} of spin glasses, and the Hopfield model~\cite{Hopfield82} of neural networks, the Curie-Weiss model~\cite{DeserCG68} all belong to the above family of distributions, with various special structures on $G$, the $\theta_{u,v}$'s and the $\theta_v$'s. Since its introduction in Statistical Physics, the Ising model has found a myriad of applications in diverse research disciplines, including probability theory, Markov chain Monte Carlo, computer vision, theoretical computer science, social network analysis, game theory, computational biology, and neuroscience; see~e.g.~\cite{LevinPW09, Chatterjee05, Felsenstein04, DaskalakisMR11, GemanG86, Ellison93, MontanariS10} and their references. The ubiquity
of these applications motivate the problem of inferring Ising models from samples, or inferring statistical properties of Ising models from samples. This type of problem has enjoyed much study in statistics, machine learning, and information theory; see, e.g.,~\cite{ChowL68,AbbeelKN06,CsiszarT06,Chatterjee07,RavikumarWL10,JalaliJR11,SanthanamW12,BreslerGS14,Bresler15,VuffrayMLC16,BreslerK16,Bhattacharya16,BhattacharyaM16,MartindelCampoCU16,KlivansM17,HamiltonKM17,DaskalakisDK18}.

Despite the wealth of theoretical study and practical applications of this model, outlined above, there are still aspects of it that are poorly understood. In this work, we focus on the important topic of concentration of measure. We are interested in studying the concentration properties of polynomial functions $f\left({X}\right)$ of the Ising model. That is, for a random vector $X$ sampled from $p$ as above and a polynomial $f$, we are interested in the concentration of $f(X)$ around its expectation $\E[f(X)]$. Since the coordinates of $X$ take values in $\{\pm 1\}$, we can without loss of generality focus our attention to multi-linear functions $f$ (Definition \ref{def:multilinear}).

While the theory of concentration inequalities for functions of independent random variables has reached a high level of sophistication, proving concentration of measure for functions of dependent random variables is significantly harder, the main tools being martingale methods, logarithmic Sobolev inequalities and transportation cost inequalities. One shortcoming of the latter methods is that explicit constants are very hard or almost impossible to get. For the Ising model, in particular, the log-Sobolev inequalities of Stroock and Zegarlinski~\cite{StroockZ92}, known under high temperature,\footnote{High temperature is a widely studied regime of the Ising model where it enjoys a number of useful properties such as decay of correlations and fast mixing of the Glauber dynamics. Throughout this paper we will take ``high temperature'' to mean that Dobrushin's conditions of weak dependence are satisfied. See Definition~\ref{def:dobrushin}.} do not give explicit constants, and it is also not clear whether they extend to systems beyond the lattice. 

An alternative approach, proposed recently by Chatterjee~\cite{Chatterjee05}, is an adaptation to the Ising model of Stein's method of exchangeable pairs.
This powerful method is well-known in probability theory, and has been used to derive concentration inequalities with explicit constants for functions of dependent random variables (see \cite{MackeyJCFT14, DaskalakisDK18} for some recent works). 
Chatterjee uses this technique to establish concentration inequalities for Lipschitz functions of the Ising model under high temperature. While these inequalities are tight (and provide Gaussian tails) for linear functions of the Ising model, they are unfortunately not tight for higher degree polynomials, in that the concentration radius is off by factors that depend on the dimension $n=|V|$. For example, consider the function $f_c(X)=\sum_{i \neq j}c_{ij} X_i X_j$ of an Ising model without external fields, where the $c_{ij}$'s are signs. Chatterjee's results imply that this function concentrates at radius $\pm O(n^{1.5})$, but as we show this is suboptimal by a factor of $\tilde{\Omega}(\sqrt{n})$. 

In particular, our main technical contribution is to obtain near-tight concentration inequalities for polynomial functions of the Ising model, whose concentration radii are tight up to logarithmic factors. %\footnote{Motivated by the prior work of~\cite{DaskalakisDK18}, Gheissari, Lubetzky, and Peres also independently obtained concentration inequalities for polynomial functions of the Ising model by different methods. (This work was submitted to NIPS 2017 on May 19, 2017 and will appear in the proceedings of the conference, their preprint appeared on Arxiv on May 31st, 2017.)} 
 A corollary of our main result (Theorem~\ref{thm:multilinear}) is as follows:
\begin{theorem}
\label{thm:informal}
Consider any degree-$d$ multilinear function $f$ with coefficients in $[-1, 1]$, defined on an Ising model $p$ without external field in the high-temperature regime.
Then there exists a constant $C = C(d) > 0$ (depending only on $d$) such that for any $r = \tilde \Omega_d(n^{d/2})$, we have
$$\Pr_{X \sim p} [|f(X) - \E[f(X)]| > r] \leq \exp\left(-C\cdot \frac{r^{2/d}}{n \log n}\right).$$
The concentration radius is tight up to logarithmic factors, and the tail bound is tight up to a $O_d(1/\log n)$ factor in the exponent of the tail bound.
\end{theorem}
Our formal theorem statements for bilinear and higher degree multilinear functions appear as Theorems~\ref{thm:bilinear} and~\ref{thm:multilinear} of Sections~\ref{sec:bilinear} and~\ref{sec:multilinear}, respectively. Some further discussion of our results is in order:
\begin{itemize}
\item Under existence of external fields, it is easy to see that the above concentration does not hold, even for bilinear functions, as observed in Section~\ref{sec:bilinear-concentration-external-field}. Motivated by our applications in Section~\ref{sec:experiments} we extend the above concentration of measure result to centered bilinear functions (where each variable $X_i$ appears as $X_i-\E[X_i]$ in the function) that also holds under arbitrary external fields; see Theorem~\ref{thm:bilinear-concentration-external}. We leave extensions of this result to higher degree multinear functions to the next version of this paper.
%See Theorems~\ref{thm:bilinear} and~\ref{thm:bilinear-concentration-external} for the statement of our theorem for bilinear functions of Ising models without external fields and with external fields respectively, Section~\ref{sec:bilinear-tightness} for the tightness of our bound. 
%In Section~\ref{sec:multilinear}, we describe how to generalize this result to prove concentration for higher degrees. 

\item Moreover, notice that the tails for degree-$2$ functions are exponential and not Gaussian, and this is unavoidable, and that as the degree grows the tails become heavier exponentials, and this is also unavoidable. In particular, the tightness of our bound is justified in Section~\ref{sec:bilinear-tightness} and Remark~\ref{rem:multilinear-tight}. 

\item Lastly, like Chatterjee and Stroock and Zegarlinski, we prove our results under high temperature. On the other hand, it is easy to construct low temperature Ising models where no non-trivial concentration holds.\footnote{Consider an Ising model with no external fields, comprising  two disjoint cliques of half the vertices with infinitely strong bonds; i.e.~$\theta_v=0$ for all $v$, and $\theta_{u,v}=\infty$ if $u$ and $v$ belong to the same clique. Now consider the multilinear function $f(X)=\sum_{u \not\sim v}X_u X_v$, wher $u \not\sim v$ denotes that $u$ and $v$ are not neighbors (i.e.~belong to different cliques). It is easy to see that the maximum absolute value of $f(X)$ is $\Omega(n^2)$ and that there is no concentration at radius better than some $\Omega(n^2)$.}
\end{itemize}

With our theoretical understanding in hand, we proceed with an experimental evaluation of the efficacy of multilinear functions applied to hypothesis testing.
Specifically, given a binary vector, we attempt to determine whether or not it was generated by an Ising model.
Our focus is on testing whether choices in social networks can be approximated as an Ising model, a common and classical assumption in the social sciences~\cite{Ellison93, MontanariS10}.
We apply our method to both synthetic and real-world data.
On synthetic data, we investigate when our statistics are successful in detecting departures from the Ising model.
For our real-world data study, we analyze the Last.fm dataset from HetRec'11~\cite{CantadorBK11}.
Interestingly, when considering musical preferences on a social network, we find that the Ising model may be more or less appropriate depending on the genre of music.

\subsection{Related Work}
As mentioned before, Chatterjee previously used the method of exchangeable pairs to prove variance and concentration bounds for linear statistics of the Ising model~\cite{Chatterjee05}.
In~\cite{DaskalakisDK18}, the authors apply and extend this method to prove variance bounds for bilinear statistics.
The present work improves upon this by proving concentration rather than bounding the variance, as well as considering general degrees $d$ rather than just $d=2$.
In simultaneous work\footnote{The present work was under submission to NIPS 2017 from May 19 to September 4, 2017, and thus not made public until after being accepted.}, Gheissari, Lubetzky, and Peres proved concentration bounds which are qualitatively similar to ours, though the techniques are somewhat different\cite{GheissariLP17}.

\subsection{Organization}
In Section~\ref{sec:preliminaries}, we define the notation we use in this paper.
We describe and prove our results for concentration of bilinear functions in Section~\ref{sec:bilinear}.
This method serves as a blueprint for our main result, concentration of higher-order multilinear functions, which we show in Section~\ref{sec:multilinear}.
In Section \ref{sec:experiments}, we describe our experimental investigation and results.

  \section{Preliminaries}
\label{sec:preliminaries}
%We consider the Ising model on a graph $G = (V,E)$ with $n$ nodes.
%This is a distribution over $\{\pm 1\}^n$, with a parameter vector $\vec \th \in \mathbb{R}^{|V| + |E|}$.
%$\vec \th$ has a parameter corresponding to each edge $e \in E$ and each node $v \in V$.
%The probability mass function assigned to a string $x$ is
%$$P(x) = \exp\left(\sum_{v \in V} \th_v x_v +  \sum_{e = (u,v) \in E} \th_{e} x_u x_v - \Phi(\vec \th) \right), $$
%where $\Phi(\vec \th)$ is the log-partition function for the distribution. 

We will abuse notation, referring to both the probability distribution $p$ and the random vector $X$ that it samples in $\{\pm 1\}^V$ as the Ising model. 
That is, $X \sim p$.
We will subscript $X$ as follows.
At times, we will consider a sequence of $X$'s at various ``time steps'' -- we will use $X_t$ or $X_i$ to denote random vectors in this sequence.
Other times, we will need to consider the value of the vector $X$ at a particular node -- we will use $X_u$ or $X_v$ to indicate random variables in this sequence.
Whether we index based on time step versus node should be apparent from the choice of subscript variable, and otherwise clear from context.
Occationally, we will use both: $X_{t,u}$ denotes the variable corresponding to node $u$ in the Ising model $X$ at some time step $t$.
Throughout the paper we will refer to the set $\Omega = \{\pm 1\}^V$.

\begin{definition}
	\label{def:multilinear}
	A \emph{degree-$d$ multilinear function} defined on $n$ variables $x_1, \dots, x_n$ is a polynomial such that
	$$\sum_{S \subseteq [n] : |S| \leq d} a_{S} \prod_{i \in S} x_i,$$
  where $a : 2^{[n]} \rightarrow \mathbb{R}$ is a coefficient vector.
\end{definition}
When the degree $d=1$, we will refer to the function as a linear function, and when the degree $d=2$ we will call it a bilinear function. 
Note that since $X_u \in \{\pm 1\}$, any polynomial function of an Ising model is a multilinear function.
We will use $a$ to denote the coefficient vector of such a multilinear function.
Note that we will use permutations of the subscripts to refer to the same coefficient, i.e., $a_{uv}$ is the same as $a_{vu}$. Also we will use the term $d$-linear function to refer to a multilinear function of degree $d$.

We say an Ising model has \emph{no external field} if $\th_v = 0$ for all $v \in V$.
An Ising model is \emph{ferromagnetic} if $\th_e \geq 0$ for all $e \in E$.

We now give a formal definition of the high-temperature regime, also known as Dobrushin's uniqueness condition -- in this paper, we will use the terms interchangeably.
\begin{definition}[Dobrushin's Uniqueness Condition]
\label{def:dobrushin}
Consider an Ising model $p$ defined on a graph $G=(V,E)$ with $|V|=n$ and parameter vector $\vec{\th}$.
Suppose $\max_{v \in V} \sum_{u \ne v} \tanh\left(\abss{\th_{uv}}\right) \le  1-\eta$ for some $\eta > 0$. Then $p$ is said to satisfy Dobrushin's uniqueness condition, or be in the high temperature regime. 
In this paper, we use the notation that an Ising model is $\eta$-high temperature to parameterize the extent to which it is inside the high temperature regime.
Note that since $\tanh(|x|) \le |x|$ for all $x$, the above condition follows from more simplified conditions which avoid having to deal with hyperbolic functions. 
For instance, either of the following two conditions:
\begin{align*}
\max_{v \in V} \sum_{u \ne v} \abss{\th_{uv}} &\le  1-\eta \text{ or}\\
\b \dm &\le  1-\eta
\end{align*}
are sufficient to imply Dobrushin's condition (where $\b = \max_{u,v} \abss{\th_{uv}}$ and $\dm$ is the maximum degree of $G$).
%Let $A = (a_{uv})$ be a $n \times n$ matrix with non-negative entries such that $a_{uu} = 0$ for all $u \in V$, and for any $u \in V$ and $x,y \in \{ \pm 1\}^V, \dtv(\mu_u(.|\bar{x}^u), \mu_u(.|\bar{y}^u)) \le \sum_{v \in V} a_{uv} \mathbbm{1}_{x_v \neq y_v}$ where $\mu_u(.|\bar{x}^u)$ represents the distribution of $X_u$ conditioned on the values of the remaining nodes being $\bar{x}^u$.
%In particular, one feasible choice of $A$ has $a_{uv} = 4|\theta_{uv}|$.
%If such a matrix $A$ exists and satisfies $\max_{u \in V} \sum_{v \in V} a_{uv} \le 1-\eta$, where $\eta>0$ is a constant,
%then the Ising model is said to satisfy \emph{Dobrushin's condition}.
\end{definition}

In some situations, we may use the parameter $\eta$ implicitly and simply say the Ising model is in the high temperature regime.
%Definition \ref{def:dobrushin} is hard to parse and gain intuition about. We present here a weaker definition which implies Dobrushin's condition but is easier to understand intuitively.
%\begin{definition}
%	\label{def:high-temp-weak}
%	For all $e \in E$, $\theta_e \leq {\eta \over 4\dm}$, where $\eta < 1$ is any constant.
%\end{definition}
In general, when one refers to the \emph{temperature} of an Ising model, a high temperature corresponds to small $\theta_{uv}$ values, and a low temperature corresponds to large $\theta_{uv}$ values.

We will use the following lemma which shows concentration of measure for Lipschitz functions on the Ising model in high temperature. 
It is a well-known result and can be found for instance as Theorem 4.3 of~\cite{Chatterjee05}.
\begin{lemma}[Lipschitz Concentration Lemma]
	\label{lem:lipschitz-lemma}
	Suppose that $f(X_1,\ldots,X_n)$ is a function of an Ising model in the high-temperature regime. 
	Suppose the Lipschitz constants of $f$ are $l_1, l_2, \ldots, l_n$ respectively. That is,
	$$\abss{f(X_1,\ldots,X_i,\ldots,X_n) - f(X_1,\ldots,X_i',\ldots,X_n)} \le l_i$$
	for all values of $X_1,\ldots,X_{i-1},X_{i+1},\ldots,X_n$ and for any $X_i$ and $X_i'$. Then,
	$$\Pr\left[ \abss{f(X)-\E[f(X)]} > t\right] \le 2\exp\left(-\frac{\eta t^2}{2\sum_{i=1}^n l_i^2} \right).$$
\end{lemma}
Note that this immediately implies sharp concentration bounds for linear functions on the Ising model.
% Can show that we can avoid factor of 4 by arguing that $\dtv$ between the conditional distributions of a node is bounded by $\sum_{u \ne v} \abss{\th_{uv}}$ and not $\sum_{u \ne v} 4\abss{\th_{uv}}$. This total variation is precisely what appears in the calculations showing contraction of Hamming.

We will refer to elements in $\Omega$ as both states and configurations of the Ising model. The name states will be more natural when considering Markov chains such as the Glauber dynamics. Glauber dynamics is the canonical Markov chain for sampling from an Ising model. 
%The dynamics are a Markov chain defined on the set $\Omega$. They proceed as follows: In each time step $t$, pick a node $u$ uniformly at random and re-sample $u$ conditioned on the value of the rest of the nodes.
% A complete description of the Glauber dynamics is provided in Section 1 of the supplementary material.
Glauber dynamics define a reversible, ergodic Markov chain whose stationary distribution is identical to the corresponding Ising model. In many relevant settings, including the high-temperature regime, the dynamics are rapidly mixing (i.e., in $O(n\log n)$ steps) and hence offer an efficient way to sample from Ising models. We consider the basic variant known as single-site Glauber dynamics. The dynamics are a Markov chain defined on the set $\Omega$. 
They proceed as follows:
\begin{enumerate}
	\item Let $X_t$ denote the state of the dynamics at time $t$. Start at any state $X_0 \in \Omega$. 
	\item Let $N(u)$ be the set of neighbors of node $u$. Pick a node $u$ uniformly at random and update $X_u$ as follows
	\begin{align*}
	&~~ X_{t+1,u} = 1 \quad \text{w.p. } \quad \frac{\exp\left(\th_u + \sum_{v \in N(u)} \th_{uv}X_{t,v} \right)}{\exp\left(\th_u + \sum_{v \in N(u)} \th_{uv}X_{t,v} \right) + \exp\left(-\th_u-\sum_{v \in N(u)} \th_{uv}X_{t,v} \right)} \\
	&~~ X_{t+1,u} = -1 \quad \text{w.p. } \quad \frac{\exp\left(-\th_u-\sum_{v \in N(u)} \th_{uv}X_{t,v} \right)}{\exp\left(\th_u+ \sum_{v \in N(u)} \th_{uv}X_{t,v} \right) + \exp\left(-\th_u-\sum_{v \in N(u)} \th_{uv}X_{t,v} \right)}
	\end{align*}
\end{enumerate}

Glauber dynamics for an Ising model in the high temperature regime are fast mixing. In particular, they mix in $O(n\log n)$ steps. To be more concrete, for an Ising model $p$ in $\eta$-high temperature, we define
\begin{align}
\tmix = \frac{n\log n}{\eta}, \label{eq:tmix}
\end{align}
The dynamics for an Ising model in high temperature also display the cutoff phenomenon. Due to this, we have Lemma \ref{lem:close-to-stationary}.
\begin{lemma}
	\label{lem:close-to-stationary}
	Let $x_0$ be any starting state for the Glauber dynamics and let $t^* = (\zeta+2)\tmix$ for some $0 \le d \le n$. If $X_{t^*,x_0}$ is the state reached after $t^*$ steps of the dynamics, then
	$$\dtv(X_{t^*,x_0}, p) \le \exp\left( - (\zeta+1)n\log n \right)$$ 
	for all $x_0$.
\end{lemma}
\begin{proof}
	This follows in a straightforward manner from the cutoff phenomenon observed with respect to the mixing of the Glauber dynamics in this setting.
	The bound on the mixing time of Glauber dynamics for high temperature Ising models (Theorem 15.1 of \cite{LevinPW09})\footnote{Note that Theorem 15.1 of \cite{LevinPW09} uses a definition of high temperature which is less general than the one we present here. But it can also be shown via very similar calculations to hold for our more general version of the high temperature regime.} gives us that to achieve $\dtv(X_t,p) \le \ve$, we must run the dynamics for $t = \frac{n\log n + \log(1/\ve)}{\eta}$ steps. This implies, that after $t^*$ steps, the total variation distance $\ve$ achieved is
	\begin{align*}
	\ve &\le \exp(-t^* \eta + n\log n) \\
	&= \exp(-(\zeta+1)n\log n).
	\end{align*}
\end{proof}

\begin{definition}
\label{def:hamming}
The \emph{Hamming distance} between $x, y \in \{\pm 1\}^n$ is defined as $\dh(x, y) = \sum_{i \in [n]} \mathbbm{1}_{\{x_i \neq y_i\}}$.
\end{definition}

\begin{definition}[The greedy coupling]
	\label{def:greedy-coupling}
	Consider two instances of Glauber dynamics associated with the same Ising model $p$: $X_0^{(1)}, X_1^{(1)},\ldots$ and $X_0^{(2)},X_1^{(2)},\ldots$. The following coupling procedure is known as the \emph{greedy coupling}. Start chain 1 at $X_0^{(1)}$ and chain 2 at $X_0^{(2)}$ and in each time step $t$, choose a node $v \in V$ uniformly at random to update in both the runs. Let $p^{(1)}$ denote the probability that the first chain sets $X_{t,v}^{(1)} = 1$ and let $p^{(2)}$ be the probability that the second chain sets $X_{t,v}^{(2)} = 1$. Let $p_1 \le p_2$ be a rearrangement of the $p^{(i)}$ values in increasing order. Also let $p_0 = 0$ and $p_{3} = 1$. Draw a number $x$ uniformly at random from $[0,1]$ and couple the updates according to the following rule:\\
	
	\noindent If $x \in [p_l,p_{l+1}]$ for some $0 \le l \le 2$, set $X_{t,v}^{(i)} = -1$ for all $1 \le i \le l$ and $X_{t,v}^{(i)} = 1$ for all $l < i \le 2$.
\end{definition}

We summarize some properties of this coupling in the following lemma, which appear in Chapter 15 of~\cite{LevinPW09}.
\begin{lemma}
	\label{lem:greedy-properties}
	The greedy coupling (Definition \ref{def:greedy-coupling}) satisfies the following properties.
	\begin{enumerate}
		\item It is a valid coupling.
		\item If $p$ is an Ising model in $\eta$-high temperature, then 
		$$\E\left[ \dh(X_t^{(1)}, X_t^{(2)}) \middle| (X_0^{(1)}, X_0^{(2)}) \right] \le \left(1 - \frac{\eta}{n} \right)^t \dh(X_0^{(1)}, X_0^{(2)}).$$
		\item The distribution of $X_t^{(1)}$, for any $t \ge 0$, conditioned on $X_0^{(1)}$ is independent of $X_0^{(2)}$.
	\end{enumerate}
\end{lemma}

%In addition, we will need the following property of the Glauber dynamics on the Ising model which states that there exists a way to couple two instances of the dynamics on the same graph such that the expected Hamming distance between the strings representing the state decreases.
%
%\begin{lemma}[Contraction of Hamming Distance under the Glauber Dynamics]
%	\label{lem:glauber-hamming-contraction}
%	Let $X_0,\ldots,X_t$ and $X_0',\ldots,X_t'$ be two runs of the Glauber dynamics on a high temperature Ising model. There exists a coupling of the two runs such that 
%	$$\E\left[\dh(X_t,X_t') | X_0,X_0'\right] \le \left(1-\frac{c(\b)}{n} \right)^t\dh(X_0,X_0')$$
%	under the coupling.
%\end{lemma}

\subsection{Martingales}
We briefly review some definitions from the theory of martingales in this section.
\begin{definition}
 A \emph{probability space} is defined by a triple $(O, \mathcal{F}, P)$ where $O$ is the possible set of outcomes of the probability space. $\mathcal{F}$ is a $\sigma$-field which is a set of all measurable events of the space and $P$ is a function which maps events in $\mathcal{F}$ to probability values.
\end{definition}

\begin{definition}
 A sequence of random variables $X_0,X_1,\ldots,X_i,\ldots$ on the probability space $(O,\mathcal{F},P)$ is a \emph{martingale sequence} if for all $i \ge 0$, $\E\left[X_{i+1} | \mathcal{F}_{i} \right] = X_i$.
\end{definition}

\begin{definition}
A \emph{stopping time} with respect to a martingale sequence defined on $(O,\mathcal{F}, P)$ is a function $\tau: O \rightarrow \{1,2,\ldots \}$ such that $\{ \tau = n\} \in \mathcal{F}_n$ for all $n$. Also, $P[\tau = \infty] > 0$ is allowed.
\end{definition}

\begin{definition}
Let $X_1,X_2,\ldots,X_n$ be a set of possibly dependent random variables. Consider any function $f(X_1,X_2,\ldots,X_n)$ on them. Then the sequence $\{B_i\}_{i \ge 1}$ where
\begin{align}
B_i = \E\left[ f(X_1,X_2,\ldots,X_n) \middle| X_1,X_2,\ldots,X_i \right]
\end{align}
is a martingale sequence and is known as the \emph{Doob martingale} of the function $f(.)$.
\end{definition}

A popular set of tools which have been used for showing concentration results such as McDiarmid's inequality come from the theory of martingales. 
In our proof, the following two martingale inequalities will be useful. 
The first is the well-known Azuma's inequality.
\begin{lemma}[Azuma's Inequality]
\label{lem:azuma}
Let $(\Omega,\mathcal{F},P)$ be a probability space. Let $\mathcal{F}_0 \subset \mathcal{F}_1 \subset \mathcal{F}_2 \ldots $ be an increasing sequence of sub-$\sigma$-fields of $\mathcal{F}$. 
Let $X_0,X_1,\ldots,X_t$ be random variables on $(\Omega,\mathcal{F},P)$ such that $X_i$ is $\mathcal{F}_i$-measurable. Suppose they represent a sequence of martingale increments. That is, $\E[X_i | \mathcal{F}_{i-1}] = 0$ or $S_i = \sum_{j=0}^i X_j$ forms a martingale sequence defined on the space $(\Omega, \mathcal{F}, P)$. Let $K\ge 0$ be such that $\Pr[\abss{X_i} \le K] = 1$ for all $i$. Then for all $\rad \ge 0$,
$$\Pr\left[ \abss{S_t} \ge \rad \right] \le 2\exp\left(-\frac{\rad^2}{tK^2} \right) $$
\end{lemma}

The second inequality due to Freedman is a generalization of Azuma's inequality. It applies when a bound on the martingale increments $\abss{X_i}$ only holds until some stopping time, unlike Azuma's, which requires a bound on the martingale increments for all times.
\begin{lemma}[Freedman's Inequality (Proposition 2.1 in \cite{Freedman75})]
\label{lem:freedman}
Let $(\Omega,\mathcal{F},P)$ be a probability space. Let $\mathcal{F}_0 \subset \mathcal{F}_1 \subset \mathcal{F}_2 \ldots $ be an increasing sequence of sub-$\sigma$-fields of $\mathcal{F}$. 
Let $X_0,X_1,\ldots,X_t$ be random variables on $(\Omega,\mathcal{F},P)$ such that $X_i$ is $\mathcal{F}_i$-measurable. Suppose they represent a sequence of martingale increments. That is, $S_i = \sum_{j=0}^i X_j$ forms a martingale sequence defined on the space $(\Omega, \mathcal{F}, P)$. Let $\tau$ be a stopping time defined on $\Omega$ and $K\ge 0$ be such that $\Pr[\abss{X_i} \le K] = 1$ for $i\le \tau$. Let $v_i = \Var[X_i | \mathcal{F}_{i-1}]$ and $V_t = \sum_{i=0}^t v_i$. Then,
\begin{eqnarray}
&\Pr[\abss{S_t} \ge \rad \text{ and } V_t \le b \text{ for some } t \le \tau ] \le 2\exp\left(-\frac{\rad^2}{2(\rad K + b)} \right)  \label{eq:freedman}\\
\equiv &\Pr[\exists t \le \tau \text{ s.t } \abss{S_t} \ge \rad \text{ and } V_t \le b ] \le 2\exp\left(-\frac{\rad^2}{2(\rad K + b)} \right)
\end{eqnarray}
\end{lemma}

  \section{Concentration of Measure for Bilinear Functions}
\label{sec:bilinear}
In this section, we prove our main concentration result for bilinear functions of the Ising model.
This is not as technically involved as the result for general-degree multilinear functions, but exposes many of the main conceptual ideas.
The theorem statement is as follows:

\begin{theorem}
\label{thm:bilinear}
Consider any bilinear function $f_a(x) = \sum_{u,v} a_{uv} x_u x_v$ on an Ising model $p$ (defined on a graph $G=(V,E)$ such that $\abss{V} = n$) in $\eta$-high-temperature regime with no external field.
Let $\|a\|_\infty = \max_{u,v} a_{uv}$.
If $X \sim p$, then for any $r \geq 300\|a\|_\infty n \log^2 n / \eta + 2$, we have
$$\Pr\left[ \abss{f_a(X) - \E\left[f_a(X)\right]} \ge \rad \right] \le 5\exp\left(-\frac{\eta\rad}{1735\|a\|_\infty n\log n} \right).$$
\end{theorem}
Note that, for the sake of convenience in our proof, this theorem is stated for bilinear functions where all terms are of degree $2$.
One can immediately obtain concentration for all bilinear functions by combining this result with concentration bounds for linear functions (see, i.e. Lemma~\ref{lem:lipschitz-lemma}).
Since linear functions concentrate in a much tighter radius ($O(\sqrt{n})$, rather than $\tilde O(n)$), this comes at a minimal additional cost.

\subsection{Overview of the Technique}
A well known approach to proving concentration inequalities for functions of dependent random variables is the via martingale tail bounds. For instance, Azuma's inequality yields such bounds without requiring any form of independence among the random variables it considers. It gives useful tail bounds whenever one can bound the martingale increments (i.e., the differences between consecutive terms of the martingale sequence) of the underlying martingale in absolute value. Such an approach is fruitful in showing concentration of linear functions on the Ising model in high temperature. The Glauber dynamics associated with Ising models in high temperature are fast mixing and offer a natural way to define a martingale sequence. In particular, consider the Doob martingale corresponding to any linear function $f$ for which we wish to show concentration, defined on the state of the dynamics at some time step $t^*$, i.e. $f(X_{t^*})$. If we choose $t^*$ larger than $O(n\log n)$ then $f(X_{t^*})$ would be very close to a sample from $p$ irrespective of the starting state. We set the first term of the martingale sequence as $\E[f(X_{t^*})|X_0]$ and the last term is simply $f(X_{t^*})$. By bounding the martingale increments we can show that $\abss{f(X_{t^*})-\E[f(X_{t^*})|X_0]}$ concentrates at the right radius with high probability. By making $t^*$ large enough we can argue that $\E[f(X_{t^*})|X_0] \approx \E[f(X)]$. Also, crucially, $t^*$ need not be too large since the dynamics are fast mixing. Hence we don't incur too big a hit when applying Azuma's inequality, and one can argue that linear functions are concentrated with a radius of $\tilde{O}(\sqrt{n})$. Crucial to this argument is the fact that linear functions are $O(1)$-Lipschitz (when the entries of $a$ are constant), bounding the Doob martingale differences to be $O(1)$.

The challenge with bilinear functions is that they are $O(n)$-Lipschitz -- a naive application of the same approach gives a radius of concentration of $\tilde{O}(n^{3/2})$, which albeit better than the trivial radius of $O(n^2)$ is not optimal. To show stronger concentration for bilinear functions, at a high level, the idea is to bootstrap the known fact that linear functions of the Ising model concentrate well at high temperature.

The key insight is that, when we have a $d$-linear function, its Lipschitz constants are bounds on the absolute values of certain $d-1$-linear functions. In particular, this implies that the Lipschitz constants of a bilinear function are bounds on the absolute values of certain associated linear functions. And although a worst case bound on the absolute value of linear functions with bounded coefficients would be $O(n)$, the fact that linear functions are concentrated within a radius of $\tilde{O}(\sqrt{n})$, means that bilinear functions are $\tilde{O}(\sqrt{n})$-Lipschitz in spirit.
In order to exploit this intuition, we turn to more sophisticated concentration inequalities, namely Freedman's inequality (Lemma \ref{lem:freedman}).
This is a generalization of Azuma's inequality, which handles the case when the martingale differences are only bounded until some stopping time (very roughly, the first time we reach a state where the expectation of the linear function after mixing is large). To apply Freedman's inequality, we would need to define a stopping time which has two properties:
\begin{enumerate}
	\item The stopping time is larger than $t^*$ with high probability. Hence, with a good probability the process doesn't stop too early. The harm if the process stops too early (at $t < t^*$) is that we will not be able to effectively decouple $\E\left[f_a(X_{t}) | X_0\right]$ from the choice of $X_0$. $t^*$ is chosen to be larger than the mixing time of the Glauber dynamics precisely because it allows us to argue that $\E\left[f_a(X_{t^*}) | X_0\right] \approx \E\left[f_a(X_{t^*})\right] = \E[f_a(X)]$.
	\item For all times $i+1$ less than the stopping time, the martingale increments are bounded, i.e. $\abss{B_{i+1}-B_i} = O(\sqrt{n})$ where $\{B_i\}_{i \ge 0}$ is the martingale sequence.
\end{enumerate}
We observe that the martingale increments corresponding to a martingale defined on a bilinear function have the flavor of the conditional expectations of certain linear functions which can be shown to concentrate at a radius $\tilde{O}(\sqrt{n})$ when the process starts at its stationary distribution. This provides us with a nice way of defining the stopping time to be the first time when one of these conditional expectations deviates by more than $\Omega(\sqrt{n}\poly \log n)$ from the origin. The stopping time we use is a bit more involved but its formulation is completely guided by the criteria listed above. Once, we have defined the stopping time, the next thing to show before we can apply Freedman's inequality is a bound on the conditional variance of the martingale increments which we do so again using the property that the martingale increments are bounded up until stopping time. Finally we proceed to apply Freedman's inequality to bound the desired quantity.

It is worth noting that the martingale approach described above closely relates to the technique of exchangeable pairs exposited by Chatterjee~\cite{Chatterjee05}. When we look at differences for the martingale sequence defined using the Glauber dynamics, we end up analyzing an exchangeable pair of the following form: sample $X \sim p$ from the Ising model. Take a step along the Glauber dynamics starting from $X$ to reach $X'$. $(X,X')$ forms an exchangeable pair. This is precisely how Chatterjee's application of exchangeable pairs  is set up. Chatterjee then goes on to study a function of $X$ and $X'$ which serves as a proxy for the variance of $f(X)$ and obtains concentration results by bounding the absolute value of this function. The definition of the function involves considering two greedily coupled runs of the Glauber dynamics just as we do in our martingale based approach. \\

\noindent To summarize, our proof of bilinear concentration involves showing various concentration properties for linear functions via Azuma's inequality (Section~\ref{sec:linear}), showing that the martingale has $\tilde O(\sqrt{n})$-bounded differences before our stopping time (Section~\ref{sec:martingale-diff}), proving that the stopping time is larger than the mixing time with high probability (Lemma~\ref{lem:bilinear-stopping-time-large}), and combining these ingredients using Freedman's inequality (Section~\ref{sec:freedman-complete}).\\

%One potential way to prove that linear functions of the Ising model concentrate is by analyzing the evolution of the Glauber dynamics Markov chain.

\noindent The organization of this section is as follows.
We will first focus on proving concentration for bilinear statistics with no external field.
In Section~\ref{sec:setup}, we state some additional preliminaries, and describe the martingale sequence and stopping time we will consider.
In Section~\ref{sec:linear}, we prove certain concentration properties of linear functions of the Ising model -- in particular, these will be useful in showing that the stopping time is large.
In Section~\ref{sec:martingale-diff}, we show that our martingale sequence has bounded differences before the stopping time.
In Section~\ref{sec:freedman-complete}, we put the pieces together and prove bilinear concentration.
In Section~\ref{sec:bilinear-concentration-external-field}, we discuss how to prove concentration for bilinear statistics under an external field. 
Note that under an external field, not all bilinear functions of the Ising model concentrate, and thus our statistics require appropriate recentering.
In Section~\ref{sec:bilinear-tightness}, we briefly argue that the exponential behavior of the tail is inherent -- for example, it could not be improved to a Gaussian tail.

\subsection{Setup}
\label{sec:setup}
We will consider functions where $\|a\|_\infty \leq 1$, Theorem~\ref{thm:bilinear} follows by a scaling argument.
Let $a \in [-1,1]^{V \choose 2}$ and define $f_a: \{\pm 1\}^V \rightarrow \reals$ as follows: 
$$f_a(x)=\sum_{u, v} a_{uv} x_u x_v.$$
The quantity of interest which we would like to bound is $\Pr[\abss{f_a(X)-\E[f_a(X)]} > \rad]$ where $X \sim p$ is a sample from the Ising model $p$.
For the time being, we will focus on the setting with \emph{no external field} for ease of exposition\footnote{Concentration under an external field (with appropriate re-centering) is discussed in Section~\ref{sec:bilinear-concentration-external-field}.}.

A crucial quantity to the whole discussion will be $\abss{f(X)-f(X')}$ where $X'$ is obtained by taking a single step of the Glauber dynamics associated with a high temperature Ising model $p$ starting from $X$.
Define $f_a^u(X) = \sum_{u \neq v} a_{uv}X_u$.
These $n$ linear functions $f_a^1(X),\ldots,f_a^n(X)$ will arise as a result of looking at $\abss{f_a(X)-f_a(X')}$, as shown in the following claim:
\begin{claim}
	\label{clm:glauber-step-characterization}
	If $X'$ is obtained by taking a step of the Glauber dynamics starting from $X$, then
	\begin{align*}
	\abss{f_a(X) - f_a(X')} =   \begin{cases}
	0 \quad \text{w.p. } p_0 \\
	2\abss{f_a^1(X)} \quad \text{w.p. } p_1 \\
	... \\
	2\abss{f_a^n(X)} \quad \text{w.p. } p_n \\	\end{cases}
	\end{align*}
	where $p_0 + \ldots + p_n = 1$.
\end{claim}
\begin{proof}
	In each step of the Glauber dynamics, a node $v$ is chosen uniformly at random and updated according to the distribution of $v$ conditioned on its neighbors under the Ising model. 
	With some probability $p_0^v$, the dynamics leave node $v$ unchanged (i.e. update it to its current value $X_v$). In this scenario, $f_a(X)-f_a(X') = 0$. If, on the other hand, the dynamics flip the sign of node $v$, then $f_a(X)-f_a(X') = (\sum_{u \neq v} a_{uv}X_u)(2X_v)$. Since $X_v \in \{ \pm 1\}$, $\abss{(\sum_{u \neq v} a_{uv}X_u)(2X_v)} = 2\abss{\sum_{u \neq v} a_{uv}X_u} = 2\abss{f_a^v(X)}$. 
\end{proof}

Next, we define a martingale sequence associated with any bilinear function $f_a$ of the Ising model. A sufficiently strong tail inequality on the difference between the first and last terms of the martingale will get us very close to the desired concentration result.
\begin{definition}
\label{def:bilinear-doob-martingale}
Let $t^* = 3\tmix = 3n\log n/\eta$. Let $X_0 \sim p$ be a sample from the Ising model $p$. Consider a walk of the Glauber dynamics starting at $X_0$ and running for $t^*$ steps: $X_0,X_1,\ldots,X_{t^*}$. $X_{t^*}$ can be viewed as a function of all the random choices made by the dynamics up to that point. That is, $X_{t^*} = h(X_0,R_1,\ldots,R_{t^*})$ where $R_i$ is a random variable representing the random choices made by the dynamics in step $i$. 
More precisely, $R_i$ represents the realization of the random choice of which node to (attempt to) update and a $Uniform([0,1])$ random variable (based upon which we decide whether or not to update the node's variable).
Hence $f_a(X_{t^*}) = \tilde{f}_{a}(X_0,R_1,\ldots,R_{t^*})$ where $\tilde{f}_a = f_a \circ h$.
Consider the Doob martingale associated with $\tilde{f}_{a}$ defined on the probability space $(O,2^{O},P)$ where $O$ is the set of all possible values of the variables $X_0,X_1,X_2,\ldots,X_{t^*}$ under the above described stochastic process and $P$ is the function which assigns probability to events in $2^O$ according to the underlying stochastic process. Also consider the increasing sequence of sub-$\sigma$-fields $\mathcal{F}_0 = 2^{O_0} \subset \mathcal{F}_1 = 2^{O_1} \subset \mathcal{F}_2 = 2^{O_2} \subset \ldots \mathcal{F}_{t^*} = 2^{O_{t^*}} = 2^O$ where $O_i$ is the set of all possible values of the variables $X_0,X_1,X_2,\ldots,X_i$. The terms in the martingale sequence are as follows. 
\begin{align}
B_0 &= \E[\tilde{f}_{a}(X_0,R_1,\ldots,R_{t^*}) | X_0] = \E\left[ f_a(X_{t^*}) | X_0 \right]\notag \\
&\cdots \notag \\
B_i &= \E[\tilde{f}_{a}(X_0,R_1,\ldots,R_{t^*}) | X_0,R_1,\ldots,R_i] = \E\left[ f_a(X_{t^*}) | X_0,X_1,\ldots,X_i \right] \label{eq:bilinear-doob}\\
&\cdots \notag \\
B_{t^*} &= \tilde{f}_{a}(X_0,R_1,\ldots,R_{t^*}) = \E\left[ f_a(X_{t^*}) | X_0,X_1,\ldots,X_{t^*} \right] = f_a(X_{t^*}) \notag
\end{align}

Since the dynamics are Markovian, we can also write $B_i$ as follows:
\begin{align*}
B_i = \E[f_a(X_{t^*}) | X_i] \quad \forall \: 0 \le i \le t^*.
\end{align*}
\end{definition}
Note that we deliberately choose to skip the term $\E[\tilde{f}_{a}(R_1,\ldots,R_{t^*})]$ and start the martingale sequence at $\E[\tilde{f}_{a}(X_0,R_1,\ldots,R_{t^*}) | X_0]$ instead. This is crucial because it enables us to obtain strong bounds on the martingale increments. We have a good understanding over the behavior of the difference in values of $f(X_{t^*})$ conditioned on $X_i$ versus $X_{i+1}$ but apriori we can't bound $\abss{\E\left[f_a(X_{t^*}) | X_0\right] - \E\left[f_a(X_{t^*})\right]}$.\\

At this point, we could try and apply Azuma's inequality by bounding the martingale increments $\abss{B_{i+1}-B_i}$. However, these increments can be $\Omega(n)$ in magnitude which would yield a radius of concentration of $\approx n^{1.5}$ from Azuma's inequality. As was remarked earlier, this is weak and we will see how we can show a radius of concentration $\approx n$ by harnessing the fact that the martingale increments are rarely, if ever, of the order $\Omega(n)$. This is because of concentration of linear functions on the Ising model. To harness this fact, we appeal to Freedman's inequality (Lemma \ref{lem:freedman}) and the first order of business in applying Freedman's inequality effectively is to define a stopping time on the martingale sequence such that two things hold:
\begin{enumerate}
	\item The stopping time is larger than $t^*$ with high probability. Hence, with a good probability the process doesn't stop too early. The harm if the process stops too early (at $t < t^*$) is that we will not be able to effectively decouple $\E\left[f_a(X_{t}) | X_0\right]$ from the choice of $X_0$. $t^*$ was chosen to be larger than the mixing time of the Glauber dynamics precisely because it allows us to argue that $\E\left[f_a(X_{t^*}) | X_0\right] \approx \E\left[f_a(X_{t^*})\right] = \E[f_a(X)]$.
	\item For all $i+1$ less than the stopping time, $\abss{B_{i+1}-B_i} = O(\sqrt{n})$.
\end{enumerate}

With the above criterion in mind, we define a stopping time $T_K$ on the martingale sequence.
\begin{definition}
	\label{def:stopping-time}
	Consider the martingale sequence defined in Definition~\ref{def:bilinear-doob-martingale}. Define the set $G_K^{a}(t)$ to be the following set of configurations:
	\begin{align}
	&G_K^a(t) = \left\{x_t \in \Omega \: | \:  \abss{\E[f_a^v(X_{t^*})|X_t=x_t]} \le K \text{ and } \abss{\E[f_a^v(X_{t^*-1})|X_t=x_t]} \le K \: \forall \: v \in V \right\} \notag\\
	&~~~~~~~~~~ \bigcap \left\{x_t \in \Omega \: |  \: \Pr\left[ \abss{f_a^v(X_{t^*}) - \E\left[ f_a^v(X_{t^*}) | X_t\right]} > K \middle| X_t=x_t  \right] \le 2\exp\left(-\frac{K^2}{16t^*} \right) \: \forall \: v \in V \right\} \label{eq:bilinear-Gt}\\
	&~~~~~~~~~~ \bigcap \left\{x_t \in \Omega \: | \: \Pr\left[ \abss{f_a^v(X_{t^*-1}) - \E\left[ f_a^v(X_{t^*-1}) | X_t \right]} > K \middle| X_t=x_t  \right] \le 2\exp\left(-\frac{K^2}{16t^*} \right) \: \forall \: v \in V \right\} \notag
	\end{align}
	where $\E[f_a^v(X_{t}) | X_{t_0}]$, for all $v \in V$, is defined as 0 for $t_0 > t$.
	Let $T_K : O \rightarrow \{0\} \bigcup \mathbb{N}$ be a stopping time defined as follows:
	\begin{align*}
	T_K = \min\{t^*+1, \min_{t \ge 0} \left\{ t \:\: \middle| t \notin G_K^a(t) \right\}\},
	\end{align*}
	Note that the event $\{T_K = t\}$ lies in the $\sigma$-field $2^{O_t}$ and hence the above definition is a valid stopping time.
\end{definition}

\subsection{Properties of Linear Functions of the Ising Model}
\label{sec:linear}
In this section, we prove the following lemma, concerned primarily with a particular type of concentration of linear functions on the Ising model.
\begin{lemma}
\label{lem:linear}
Let $X_0$ be a sample from an Ising model $p$ at $\eta$-high temperature with \emph{no external field}, and $X_t$ be obtained by taking $t$ steps along the Glauber dynamics corresponding to $p$ with the condition that the dynamics start at $X_0$. For any linear function $f(x) := \sum_{v \in V} a_v x_v$ such that $\abss{a_v} \le 1$, define $g^t(X_0) = \E[f(X_t)|X_0]$. Then the following hold for any $t \ge 0$,
\begin{eqnarray}
\E[f(X_t)] = 0, \\
\Pr\left[\abss{g^t(X_0)}  > \rad \right] \le 2\exp\left( -\frac{\eta \rad^2}{8n} \right),\\
%\Pr\left[\abss{f(X_t) - \E[f(X_t)|X_0]} > K \middle| X_0=x_0 \right] \le 2\exp\left(-\frac{K^2}{4t} \right) \quad \forall x_0.
\Pr\left[\abss{f(X_t) - \E[f(X_t)|X_0]} > K \right] \le 2\exp\left(-\frac{K^2}{4t} \right).
\end{eqnarray}
\end{lemma}
\begin{proof}
First, if $X_0 \sim p$, then since $p$ is the stationary distribution of the associated Glauber chain, $X_t \sim p$ as well. Hence, $\E[f(X_t)] = \E[f(X_0)] = \sum_{v \in V} a_v \E[X_{0,v}] = 0$ for all $t \ge 0$.\\

\noindent For showing the second property, we will first bound the Lipschitz constants of the function $g^t(.)$. We denote by $\vec{l}$ the vector of Lipschitz constants of $f(x)$. Since $f$ is a linear function, $\vec{l} = [2\abss{a_1},\ldots,2\abss{a_v},\ldots,2\abss{a_n}]$. We have for any $x,x'$ such that $\dh(x,x') = 1$,
\begin{align}
&\abss{g^t(x = x_1,\ldots,x_i,\ldots,x_n) - g^t(x' = x_1,\ldots,x_i',\ldots,x_n) } = \abss{\E[f(X_t)|X_0=x] - \E[f(X_t') | X_0'=x']} \notag\\
&~~= \abss{\E[f(X_t) - f(X_t') | X_0=x,X_0'=x']} \label{eq:l1}\\
&~~\le \abss{\E[(X_t-X_t') \cdot \vec{l}/2 | X_0=x,X_0'=x']} \notag\\
&~~\le \E\left[ 2\dh(X_t,X_t') | X_0=x,X_0'=x' \right] \label{eq:l2}\\
&~~\le 2. \label{eq:l3}
\end{align}
where (\ref{eq:l1}) holds for any valid coupling of the two chains starting at $X_0$ and $X_0'$ respectively, in particular, we use the greedy coupling (Definition~\ref{def:greedy-coupling}) here. (\ref{eq:l2}) follows because $|a_i| \le 1\ \forall i$, and (\ref{eq:l3}) follows because the expected Hamming distance between $X_t$ and $X_t'$, due to the contracting nature of the Glauber dynamics under the greedy coupling (Lemma~\ref{lem:greedy-properties}), is smaller than $\dh(X_0,X_0')$ which is equal to 1. Also note that $\E[g^t(X_0)] = \E[f(X_t)] = 0$.
Hence, applying Lemma \ref{lem:lipschitz-lemma} to $g^t(x)$, we get
$$\Pr\left[\abss{g^t(X_0)}  > \rad \right] \le 2\exp\left( -\frac{\eta \rad^2}{8n} \right).$$
Note that we could apply Lemma \ref{lem:lipschitz-lemma} to $g^t(.)$ because $X_0$ was drawn from the stationary distribution of the Glauber dynamics.\\

To show the third property, we will define a martingale similar to the one defined in Definition~\ref{def:bilinear-doob-martingale} and apply Azuma's inequality to it. 
Consider a run of the Glauber dynamics starting at $X_0 \sim p$ and running for $t$ steps. We will view $X_t$ as a function of all the random choices made by the dynamics up to step $t^*$. That is, $X_t = h(X_0,R_1,\ldots,R_t)$ where $R_i$ denotes the random choices made by the dynamics during step $i$. 
More precisely, $R_i$ represents the realization of the random choice of which node to (attempt to) update and a $Uniform([0,1])$ random variable (based upon which we decide whether or not to update the node's variable).
Hence $f(X_t) = \tilde{f}(X_0,R_1,\ldots,R_t)$ where $\tilde{f} = f \circ h$.
Consider the Doob martingale defined on $\tilde{f}$:
\begin{align}
D_0 &= \E[\tilde{f}(X_0,R_1,\ldots,R_t) | X_0]\notag \\
&\cdots \notag \\
D_i &= \E[\tilde{f}(X_0,R_1,\ldots,R_t) | X_0,R_1,\ldots,R_i] \label{eq:linear-doob}\\
&\cdots \notag \\
D_{t} &= \tilde{f}(X_0,R_1,\ldots,R_t) \notag
\end{align}

Since the dynamics are Markovian, we can also write $D_i$ as follows:
\begin{align*}
D_i = \E[f(X_t) | X_i].
\end{align*}
Next we will bound the increments of the above martingale and apply Azuma's inequality to get the desired tail bound.
In the following calculation, we will use the notation $x \rightarrow y$ where $x,y \in \{\pm 1\}^n$, to denote that $y$ is a possible transition according to a single step of the dynamics starting from $x$.
For any $0 \le i \le t-1$,
	\begin{align}
	&\abss{D_{i+1}-D_i} = \abss{\E\left[f(X_t) | X_{i+1}\right] - \E\left[f(X_t) | X_{i}\right]} \\ 
	&~~\le \max_{x,y: \dh(x,y)=1} \abss{\E\left[f(X_t) | X_{i+1}=x\right] - \E\left[f(X_t') | X_{i}'=y\right]} \label{eq:l4}\\
	&~~= \max_{x,y: \dh(x,y)=1} \abss{ \E\left[f(X_t) | X_{i+1}=x\right] - \sum_{y' : y \rightarrow y'} \Pr[y \rightarrow y']\E\left[f(X_t') | X_{i+1}'=y'\right]} \label{eq:l5}\\
	&~~\le \max_{x,y': \dh(x,y')\le 2} \abss{ \E\left[f(X_t) | X_{i+1}=x\right] - \E\left[f(X_t') | X_{i+1}'=y'\right]} \label{eq:l6}\\
	&~~= \max_{x,y' : \dh(x,y') \le 2} \abss{\E\left[f(X_t) - f(X_t') | X_{i+1} = x, X_{i+1}'=y'\right]} \label{eq:l7}\\
	&~~= \max_{x,y' : \dh(x,y') \le 2} \abss{\E\left[\sum_v a_{v}(X_{t,v}-X_{t,v}') | X_{i+1} = x, X_{i+1}'=y'\right]} \notag\\
	&~~\le \max_{x,y' : \dh(x,y') \le 2} \E\left[2\dh(X_t,X_t') | X_{i+1} = x, X_{i+1}'=y'\right] \le 2. \label{eq:l8}
	\end{align}
	where in (\ref{eq:l4}) we relabeled the variables in the second expectation to avoid notational confusion in the later steps of our bounding, maintaining the understanding that the sequence $\{X_i',Y_i'\}_i$ has the same distribution as $\{X_i,Y_i\}_i$., (\ref{eq:l7}) holds for any valid coupling of the $X_i$ and $X_i'$ chains, in particular, it holds for the greedy coupling between the runs (Definition~\ref{def:greedy-coupling}). (\ref{eq:l8}) follows from the condition $\abss{a_{v}} \le 1$ and the contracting nature of the Glauber dynamics (Lemma \ref{lem:greedy-properties}) under the greedy coupling.
	
	Hence, for all $0 \le i \le t-1$, $\abss{D_{i+1}-D_i} \le 2$. Azuma's inequality applied on the martingale sequence $\{D_i\}_{i \ge 0}$ yields
	\begin{align*}
	&\Pr\left[ \abss{D_t - D_0} > K \right] \le 2\exp\left( -\frac{K^2}{4t} \right) \\
	\implies &\Pr\left[ \abss{f(X_t)-\E[f(X_t)|X_0]} > K \right] \le 2\exp\left( -\frac{K^2}{4t} \right)
	\end{align*}
	
\end{proof}

\subsection{Showing that the Martingale Process Doesn't Stop Too Early}
As a consequence of Lemma \ref{lem:linear}, we can show that with sufficiently large probability the stopping time $T_K$ defined above is larger than $t^*$, and thus Freedman's inequality gives guarantees for all $t$ up to $t^*$ with high probability. The main lemma will be the following one, which shows that for any $t$, $X_t \in G_K^a(t)$ with high probability. 
\begin{lemma}
	\label{lem:goodset-largeprob}
	For any $t \ge 0$, for $t^* = 3\tmix$,
	$$\Pr\left[X_t \notin G_K^a(t)\right]\le 8n\exp\left( -\frac{K^2}{8t^*} \right).$$
\end{lemma}
\begin{proof}
	Since $X_0$ is a sample from the stationary distribution $p$ of the dynamics, it follows from the property of stationary distributions that $X_t$ is also a sample from $p$. Hence we have, from Lemma \ref{lem:linear}, and a union bound, that
	\begin{align}
	\Pr\left[ \exists \: v \in V \: \text{s.t. } \max\left\{\abss{\E\left[ f_a^v(X_{t^*}) | X_t \right]}, \abss{\E\left[ f_a^v(X_{t^*-1}) | X_t \right]}\right\} > K \right] \le 4n\exp\left(  -\frac{\eta K^2}{8n}\right). \label{eq:bilinear-gt1}
	\end{align}
	Let $D_K^a(t)$ be the event defined as 
	$$D_K^a(t) =\{ \exists \: v \in V \text{ s.t. }  \max\left\{\abss{f_a^v(X_{t^*}) - \E\left[ f_a^v(X_{t^*}) | X_t \right]}, \abss{f_a^v(X_{t^*-1}) - \E\left[ f_a^v(X_{t^*-1}) | X_t \right]} \right\} > K \}.$$
	From Lemma \ref{lem:linear}, and a union bound, we have,
	\begin{align}
	&\Pr\left[ D_K^a(t) \right] = \E\left[ \Pr\left[ D_K^a(t) | X_t \right] \right] \le 4n\exp\left(-\frac{K^2}{4t^*}\right) \notag \\
	\implies & \Pr\left[ \Pr\left[ D_K^a(t) | X_t \right] > \exp\left(-\frac{K^2}{8t^*}\right) \right] \le 4n\exp\left(-\frac{K^2}{8t^*}\right) \label{eq:bilinear-gt2}
	\end{align}
	where (\ref{eq:bilinear-gt2}) follows from a simple application of Markov's inequality.
	Hence,
	\begin{align}
	&\Pr\left[ \exists \: v \in V \text{ s.t. } \max\left\{\Pr\left[\abss{f_a^v(X_{t^*}) - \E\left[ f_a^v(X_{t^*}) | X_t \right]} > K\right], \Pr\left[\abss{f_a^v(X_{t^*-1}) - \E\left[ f_a^v(X_{t^*-1}) | X_t \right]} > K\right] \right\} > \exp\left(-\frac{K^2}{8t^*}\right) \right] \notag\\
	& \le \Pr\left[ \Pr\left[ D_K^a(t) | X_t \right] > \exp\left(-\frac{K^2}{8t^*}\right) \right] \le 4n\exp\left(-\frac{K^2}{8t^*}\right) \label{eq:bilinear-gt3}.
	\end{align}
	From (\ref{eq:bilinear-gt1}) and (\ref{eq:bilinear-gt3}), we have,
	\begin{align*}
	& \Pr\left[ X_t \notin G_K^a(t) \right] \le 4n\exp\left(-\frac{K^2}{8t^*} \right) +  4n\exp\left(-\frac{\eta K^2}{8n} \right) \le 8n\exp\left( -\frac{K^2}{8t^*} \right).
	\end{align*}
\end{proof}

Given Lemma~\ref{lem:goodset-largeprob}, the proof of the stopping time being large with high probability follows by a simple application of the union bound.
\begin{lemma}
	\label{lem:bilinear-stopping-time-large}
	For $t^* = 3\tmix$,
	$$\Pr\left[ t^* \ge T_K \right] \le 8nt^*\exp\left(-\frac{K^2}{8t^*} \right).$$
\end{lemma}
\begin{proof}
	From Lemma \ref{lem:goodset-largeprob}, we have
	\begin{align*}
	& \Pr\left[ X_t \notin G_K^a(t) \right] \le 8n\exp\left( -\frac{K^2}{8t^*} \right) \\
	\implies & \Pr[t^* \ge T_K] = \Pr\left[ \bigcup_{t=0}^{t^*} X_t \notin G_K^a(t) \right] \\
	& \le \sum_{t=0}^{t^*}\Pr\left[X_t \notin G_K^a(t) \right] \le 8t^* n\exp\left(-\frac{K^2}{8t^*} \right) \\
	&= 8nt^*\exp\left(-\frac{K^2}{8t^*} \right).
	\end{align*}
\end{proof}

\subsection{Bounding the Martingale Differences}
\label{sec:martingale-diff}

In this section, we prove a bound on the increments $B_{i+1}-B_i$ for the Doob martingale defined in Definition~\ref{def:bilinear-doob-martingale} which holds with probability 1 for all $i+1 < T_K$. We begin by showing a bound which holds pointwise when $X_i \in G_K^a(i)$ and $X_{i+1} \in G_K^a(i+1)$ in the form of Lemma \ref{lem:bilinear-martingale-increment-bound}.
\begin{lemma}
	\label{lem:bilinear-martingale-increment-bound}
	Consider the Doob martingale defined in Definition~\ref{def:bilinear-doob-martingale}. Suppose $X_{i} \in G_K^a(i)$ and $X_{i+1} \in G_K^a(i+1)$. Then
	\begin{align*}
	\abss{B_{i+1} - B_i} \le 16K + 16n^2\exp\left(-\frac{K^2}{16t^*} \right).
	\end{align*}
\end{lemma}
\begin{proof}
	For ease of exposition, we will refer to $G_K^a(i)$ as simply $G_i$ in the following proof.
	\begin{align}
	&\abss{B_{i+1}-B_i} = \abss{\E\left[f_a(X_{t^*}) \middle| X_{i+1}\right] - \E\left[f_a(X_{t^*}) \middle| X_{i}\right]} \\
	&~~ = \abss{\E\left[f_a(X_{t^*}) \middle| X_{i+1}\right] - \E\left[f_a(X_{t^*}') \middle| X_{i}'\right]} \label{eq:md1}\\
	&~~\le \abss{\E\left[f_a(X_{t^*}) \middle| X_{i+1} \right] - \E\left[f_a(X_{t^*-1}') \middle| X_{i}'\right]} + \abss{\E\left[f_a(X_{t^*-1}') \middle| X_{i}'\right] - \E\left[f_a(X_{t^*}') \middle| X_{i}'\right]}\label{eq:md2}\\
	&~~\le \abss{\E\left[f_a(X_{t^*}) - f_a(X_{t^*-1}') \middle| X_{i+1}, X_{i}'\right]} + \abss{\E\left[ f_a(X_{t^*}') - f_a(X_{t^*-1}') \middle| X_i' \right]}  \label{eq:md3}\\
	%&~~\le \abss{\E\left[f_a(X_{t^*}) - f_a(X_{t^*-1}') \middle| X_{i+1}, X_{i}'\right]} + 2K  \label{eq:md4}\\
	&~~= \left\lvert \E\left[\sum_v (X_{t^*,v}-X_{t^*-1,v}') \left(\sum_{u} a_{uv}X_{t^*,u}\right) +  \sum_u (X_{t^*,u}-X_{t^*-1,u}')\left(\sum_{v} a_{uv} X_{t^*-1,v}'\right) \middle| X_{i+1}, X_{i}'\right] \right\rvert \notag \\
	&~~~~~+ \abss{\E\left[ f_a(X_{t^*}') - f_a(X_{t^*-1}') \middle| X_i' \right]} \\
	&~~\le \E\left[\sum_v \abss{X_{t^*,v}-X_{t^*-1,v}'}\abss{\sum_{u} a_{uv} X_{t^*,u}} + \sum_u \abss{X_{t^*,u}-X_{t^*-1,u}'}\abss{\sum_{v} a_{uv} X_{t^*-1,v}'} \:\: \middle| X_{i+1}, X_{i}'\right] \notag \\
	&~~~~~+ \abss{\E\left[ f_a(X_{t^*}') - f_a(X_{t^*-1}') \middle| X_i' \right]} \label{eq:md5}
	\end{align}
	where in (\ref{eq:md1}) we relabeled the variables in the second expectation to avoid notational confusion in the later steps of our bounding, maintaining the understanding that the sequence $\{X_i',Y_i'\}_i$ has the same distribution as $\{X_i,Y_i\}_i$, in (\ref{eq:md2}) we added and subtracted the term $\E[f_a(X_{t^*-1}') | X_i']$, (\ref{eq:md3}) holds for any valid coupling of the two chains, one starting at $X_{i+1}$ and the other starting at $X_i'$, and both running for $t^*-1-i$ steps. In particular, we use the greedy coupling between these two runs (Definition~\ref{def:greedy-coupling}). 
	%Since $X_i' \in G_i$, and since $\abss{f_a(X_{t^*}') - f_a(X_{t^*-1}')} \le \max_{v \in V} 2\abss{f_a^v(X_{t^*-1}')}$, we have $\abss{\E\left[ f_a(X_{t^*}') - f_a(X_{t^*-1}') \middle| X_i' \right]} \le 2K$ by definition of $G_i$ and hence (\ref{eq:md4}) follows.
	Consider the first term in (\ref{eq:md5}).
	Since $X_{i+1} \in G_{i+1}$, we have
	\begin{align}
	&\abss{\E\left[\sum_u a_{uv}X_{t^*,u} \middle| X_{i+1} \right]} \le K \text{ and} \label{eq:md6} \\
	&\Pr\left[ \abss{\sum_u a_{uv}X_{t^*,u} - \E\left[\sum_u a_{uv}X_{t^*,u} \middle| X_{i+1} \right]} > K \middle| X_{i+1} \right] \le 2\exp\left(-\frac{K^2}{16t^*} \right). \label{eq:md7}\\
	\implies &\Pr\left[ \abss{\sum_u a_{uv}X_{t^*,u}} > 2K \middle| X_{i+1}\right] \le 2\exp\left(-\frac{K^2}{16t^*} \right). \label{eq:md8}
	\end{align}
	where (\ref{eq:md6}) and (\ref{eq:md7}) together imply (\ref{eq:md8}).
	Similarly, we also get that 
	\begin{align}
	\Pr\left[ \abss{\sum_v a_{uv}X_{t^*-1,v}'} > 2K \middle| X_i'\right] \le 2\exp\left(-\frac{K^2}{16t^*} \right). \label{eq:md11}
	\end{align}
	
	Since $\sum_v \abss{X_{t^*,v}-X_{t^*-1,v}'}\abss{\sum_{u} a_{uv} X_{t^*,u}} \le 2n^2$ for all $X_{t^*},X_{t^*-1}'$, and since $\sum_v \abss{X_{t^*,v}-X_{t^*-1,v}'} = 2\dh(X_{t^*},X_{t^*-1}')$, (\ref{eq:md8}) and (\ref{eq:md11}) imply that
	\begin{align}
	&\E\left[\sum_v \abss{X_{t^*,v}-X_{t^*-1,v}'}\abss{\sum_{u} a_{uv} X_{t^*,u}} + \sum_u \abss{X_{t^*,u}-X_{t^*-1,u}'}\abss{\sum_{v} a_{uv} X_{t^*-1,v}'} \:\: \middle| X_{i+1}, X_{i}'\right] \\
	&\le \E\left[4\dh(X_{t^*},X_{t^*-1}')K +  4\dh(X_{t^*},X_{t^*-1}')K \middle| X_{i+1}, X_{i}'\right] + 8n^2\exp\left(-\frac{K^2}{16t^*} \right)\label{eq:md9} \\
	&\le 8K + 8n^2\exp\left(-\frac{K^2}{16t^*} \right) ,\label{eq:md10}
	\end{align}
	where (\ref{eq:md10}) follows because of the contracting nature of Hamming distance under the greedy coupling of the Glauber dynamics (Lemma~\ref{lem:greedy-properties}).
	
	The same bound can be proven by following the same steps for the second term in (\ref{eq:md5}), completing the proof.
\end{proof}

As a consequence of Lemma \ref{lem:bilinear-martingale-increment-bound}, we get the following two useful corollaries.
\begin{corollary}
	\label{cor:martingale-increment-prob-bound}
	Consider the martingale sequence defined in Definition~\ref{def:bilinear-doob-martingale}. 
	$$\Pr\left[\forall \: 0 < i+1 < T_K, \: \abss{B_{i+1}-B_i} \le 16K+16n^2\exp\left(-\frac{K^2}{16t^*}\right) \right] = 1.$$
\end{corollary}
\begin{proof}
	Let $\kappa = 16K + 16n^2\exp\left(-\frac{K^2}{16t^*}\right)$.
	\begin{align}
	&~\Pr\left[\forall \: 0 < i+1 < T_K, \: \abss{B_{i+1}-B_i} \le \kappa \right] \notag\\
	&=1 - \Pr\left[\exists \: 0 < i+1 < T_K, \: \abss{B_{i+1}-B_i} > \kappa \right] \notag\\
	&=1 - \Pr\left[\exists \: 0 < i+1 < T_K, \: \left(X_i \in G_K^a(i), X_{i+1}\in G_K^a(i+1) \text{ and } \abss{B_{i+1}-B_i} > \kappa\right) \right. \notag\\
	&~~~\left. \text{ or }  \left(\left(X_i \notin G_K^a(i) \text{ or } X_{i+1}\notin G_K^a(i+1)\right) \text{ and } \abss{B_{i+1}-B_i} > \kappa\right)\right] \notag\\
	&= 1- \Pr\left[\exists \: 0 < i+1 < T_K, \: \left(X_i \in G_K^a(i), X_{i+1}\in G_K^a(i+1) \text{ and } \abss{B_{i+1}-B_i} > \kappa\right) \right] \label{eq:pr-one3}\\
	& = 1-0 \label{eq:pr-one4}
	\end{align}
	where (\ref{eq:pr-one3}) follows because by the definition of $T_K$, $\forall \: 0 < i+1 < T_K, \: \left(X_i \in G_K^a(i) \text{ and } X_{i+1}\notin G_K^a(i+1)\right)$, and (\ref{eq:pr-one4}) follows because $X_i \in G_K^a(i), X_{i+1}\in G_K^a(i+1) \implies \abss{B_{i+1}-B_i} \le \kappa$ (Lemma \ref{lem:bilinear-martingale-increment-bound}).
\end{proof}
Corollary \ref{cor:martingale-increment-prob-bound}, will give us one of the required conditions to apply Freedman's inequality.

As a corollary of Lemma \ref{lem:bilinear-martingale-increment-bound}, we get a bound on the variance of the martingale differences which holds with high probability.
To show it we first show Claim~\ref{clm:nice-sets} which states that, informally, for any time step $i$, with a large probability we hit an $X_i$ such that the probability of transitioning from $X_i$ to an $X_{i+1} \in G_K^a(i+1)$ is large.
\begin{claim}
	\label{clm:nice-sets}
	Denote by $N_K^a(i)$ the following set of configurations:
	\begin{align}
	N_K^a(i) = \left\{ x_i \in \Omega \middle| \Pr\left[X_{i+1} \notin G_K^a(i+1) \middle| X_i = x_i \right] \le \exp\left( -\frac{K^2}{16t^*} \right) \right\}.
	\end{align}
	Then,
	$$\Pr\left[X_i \notin N_K^a(i)\right] \le 8n\exp\left(-\frac{K^2}{16t^*}\right).$$
\end{claim}
\begin{proof}
	We have from Lemma \ref{lem:goodset-largeprob}, that 
	\begin{align}
	&\Pr\left[X_{i} \notin G_K^a(i) \right] \le 8n\exp\left( -\frac{K^2}{8t^*} \right) \text{ and} \label{eq:nice1}\\
	&\Pr\left[X_{i+1} \notin G_K^a(i+1) \right] \le 8n\exp\left( -\frac{K^2}{8t^*} \right). \label{eq:nice2}
	\end{align}
	From the definition of the set $N_K^a(i)$ we have,
	\begin{align}
	\Pr\left[X_{i+1} \in G_K^a(i) | X_i \notin N_K^a(i)\right] \le 1-\exp\left( -\frac{K^2}{16t^*} \right).
	\end{align}
	Then we have,
	\begin{align}
	&1-8n\exp\left( -\frac{K^2}{8t^*} \right) \le \Pr\left[X_{i+1} \in G_K^a(i)\right]\\
	&= \Pr\left[X_{i+1} \in G_K^a(i) | X_i \in N_K^a(i)\right] \Pr\left[X_i \in N_K^a(i)\right] + \Pr\left[X_{i+1} \in G_K^a(i) | X_i \notin N_K^a(i)\right]\Pr\left[X_i \notin N_K^a(i)\right] \\
	&\le \Pr\left[X_i \in N_K^a(i)\right] + \left(1-\exp\left( -\frac{K^2}{16t^*} \right)\right)\Pr\left[X_i \notin N_K^a(i)\right] \\
	&=   \left(1-\exp\left( -\frac{K^2}{16t^*} \right)\right) + \exp\left( -\frac{K^2}{16t^*} \right)\Pr\left[X_i \in N_K^a(i)\right]. \label{eq:nice3}
	\end{align}
	(\ref{eq:nice3}) implies,
	\begin{align}
	&\Pr\left[X_i \in N_K^a(i)\right] \ge  \frac{\exp\left( -\frac{K^2}{16t^*} \right) - 8n\exp\left( -\frac{K^2}{8t^*} \right)}{\exp\left( -\frac{K^2}{16t^*} \right)}\\
	&=1-8n\exp\left(-\frac{K^2}{16t^*}\right).
	\end{align}
\end{proof}

\begin{lemma}
	\label{lem:martingale-increment-variance}
	Consider the martingale sequence defined in Definition~\ref{def:bilinear-doob-martingale}. Let $b=\left(16K + 16n^2\exp\left(-\frac{K^2}{16t^*} \right)\right)^2 + n^4\exp\left(-\frac{K^2}{16t^*}\right)$. Denote by $N_K^a(i)$ the following set of configurations (as was defined in Claim \ref{clm:nice-sets}):
	\begin{align}
	N_K^a(i) = \left\{ x_i \in \Omega \middle| \Pr\left[X_{i+1} \notin G_K^a(i+1) \middle| X_i = x_i \right] \le \exp\left( -\frac{K^2}{16t^*} \right) \right\}.
	\end{align}
	Then,
	\begin{align*}
	\Pr\left[\Var[B_{i+1} - B_i | \mathcal{F}_i] > b \middle| X_i \in G_K^a(i) \cap N_K^a(i) \right] = 0.
	\end{align*}
	where $\mathcal{F}_i = 2^{O_i}$.
\end{lemma}
\begin{proof}
%	We have, by definition, $E[X | \mathcal{F}]$ is any $\mathcal{F}$-measurable function which satisfies
%	\begin{align}
%	\int_F \left[E[X | \mathcal{F}]\right] dP = \int_F X dP
%	\end{align}
%	for any $F \in \mathcal{F}$. Since any event $F_i \in \mathcal{F}_i = 2^{O_i}$ is a set of assignments of values to the variables $X_0,X_1,\ldots,X_i$, we have,
	Since, the random variables $X_0,\ldots,X_i$ together characterize every event in $\mathcal{F}_i$, we have,
	\begin{align}
	\Var[B_{i+1}-B_i | \mathcal{F}_i] = \Var[B_{i+1}-B_i | X_0,X_1,\ldots,X_i] = \Var[B_{i+1}-B_i | X_i] \label{eq:desigmafy} 
	\end{align}
	where the last equality follows from the Markov property of the Glauber dynamics.
	By the definition of $N_K^a(i)$, we have that
%	\begin{align}
%	\Pr\left[X_{i+1} \notin G_K^a(i+1) \middle| X_i \in N_K^a(i)\right] \le \exp\left( -\frac{K^2}{16t^*} \right).
%	\end{align}
%	This implies that,
%	\begin{align}
%	&\Pr\left[X_{i+1} \in G_K^a(i+1) \text{ and } X_i \in G_K^a(i) \middle| X_i \in G_K^a(i) \cap N_K^a(i)\right] \ge 1 - \exp\left( -\frac{K^2}{16t^*} \right)\\
%	\implies &\Pr\left[\abss{B_{i+1}-B_i} < 16K + 16n^2\exp\left(-\frac{K^2}{16t^*} \right) \middle| X_i \in G_K^a(i) \cap N_K^a(i)\right] \ge 1 - \exp\left( -\frac{K^2}{16t^*} \right) \label{eq:var6}
%	\end{align}
%	where (\ref{eq:var6}) follows from Lemma~\ref{lem:bilinear-martingale-increment-bound}. Now,
	\begin{align}
	&\Var\left[B_{i+1}-B_i \middle| X_i \in G_K^a(i) \cap N_K^a(i)\right] = \E\left[\left(B_{i+1}-B_i - \E\left[B_{i+1}-B_i | X_i \in G_K^a(i) \cap N_K^a(i)\right]\right)^2 \middle| X_i \in G_K^a(i) \cap N_K^a(i)\right] \\
	&= \E\left[\left(B_{i+1}-B_i\right)^2 \middle| X_i \in G_K^a(i) \cap N_K^a(i)\right] \label{eq:var7}\\
	&\le \E\left[\left(B_{i+1}-B_i\right)^2 \middle| X_i \in G_K^a(i) \cap N_K^a(i) \cap X_{i+1} \in G_K^a(i+1)\right] \notag\\
	&~~+ \E\left[\left(B_{i+1}-B_i\right)^2 \middle| X_i \in G_K^a(i) \cap N_K^a(i) \cap X_{i+1} \notin G_K^a(i+1)\right]\Pr\left[X_{i+1} \notin G_K^a(i+1) \middle| X_i \in G_K^a(i) \cap N_K^a(i)\right] \notag\\
	&\le \left(16K + 16n^2\exp\left(-\frac{K^2}{16t^*} \right)\right)^2 + 4n^4\Pr\left[X_{i+1} \notin G_K^a(i+1) \middle| X_i \in G_K^a(i) \cap N_K^a(i)\right] \label{eq:var8}\\
	&\le \left(16K + 16n^2\exp\left(-\frac{K^2}{16t^*} \right)\right)^2 + 4n^4\exp\left( -\frac{K^2}{16t^*} \right), \label{eq:var9}
	\end{align}
	where (\ref{eq:var7}) holds because $\E\left[B_{i+1}-B_i | X_i=x_i\right] = 0$ for all $x_i$ since $\{B_i\}_{i \ge 0}$ is a martingale, (\ref{eq:var8}) holds because $\Pr[\abss{B_{i+1}-B_i} \le 16K + 16n^2\exp\left(-\frac{K^2}{16t^*} \right) | X_i \in G_K^a(i) \cap X_{i+1} \in G_K^a(i+1)] = 1$ and the maximum $\E[(B_{i+1}-B_i)^2]$ can be is at most $2n^2$, (\ref{eq:var9}) follows from Claim \ref{clm:nice-sets}.
	 The last inequality implies the statement of the lemma.
\end{proof}

\subsection{Applying Freedman's Inequality and Completing the Proof}
\label{sec:freedman-complete}

With Lemma~\ref{lem:bilinear-martingale-increment-bound} and Lemma~\ref{lem:martingale-increment-variance} to bound the martingale increments, and Lemma~\ref{lem:bilinear-stopping-time-large} to show that the stopping time is large, we are ready to apply Freedman's inequality on the martingale defined in Definition~\ref{def:bilinear-doob-martingale}.

\begin{lemma}
\label{lem:apply-freedman}
For all $\rad \ge 300n\log^2 n/\eta$, 
$$\Pr\left[ \abss{f_a(X_{t^*}) - \E\left[f_a(X_{t^*})|X_0\right]} \ge \rad  \right] \le 5\exp\left(-\frac{\eta\rad}{1734n\log n} \right).$$
\end{lemma}
\begin{proof}
From Freedman's inequality (Lemma \ref{lem:freedman}) applied on the martingale sequence (Definition~\ref{def:bilinear-doob-martingale}), we get
\begin{align}
\Pr\left[\exists t < T_K \text{ s.t. } \abss{B_t-B_0} \ge \rad \text{ and } V_{t} \le B \right] \le 2\exp\left(-\frac{\rad^2}{2(\rad K_1 + B)} \right) \label{eq:fr1}
\end{align}
where $K_1 = 16K + 16n^2\exp\left( -\frac{K^2}{16t^*} \right)$ (Lemma \ref{lem:bilinear-martingale-increment-bound}) and $V_t$ is defined as follows:
\begin{align}
V_t = \sum_{i=0}^{t-1} \Var\left[ B_{i+1} - B_{i} | \mathcal{F}_i \right]. \label{eq:bilinear-vt-def}
\end{align}
%Since $V_{t} \le t\left(5K + 4n^2\exp\left(-\frac{K^2}{16t^*} \right)\right)^2$ with probability  1 for all $t < T_K$, we have that for\\
% $b = t^*\left(5K + 4n^2\exp\left(-\frac{K^2}{16t^*} \right)\right)^2$
%\begin{align}
%\Pr\left[\exists t < T_K \text{ s.t. } \abss{B_t - B_0}> \rad \right] \le 2\exp\left(-\frac{\rad^2}{\rad K_1 + b} \right) \label{eq:fr1}
%\end{align}
%because the bound on $V_t$ attains its maximum at $t=t^*$ as the martingale dies after $t^*$ (alternately, it can be viewed such that each term after $t^*$ is deterministically equal to the value at $t^*$ and hence has variance 0).
Set $B = t^*\left(16K + 16n^2\exp\left(-\frac{K^2}{16t^*} \right)\right)^2 + 4t^*n^4\exp\left( -\frac{K^2}{16t^*} \right)$ and $K = \sqrt{\rad}$.
Next, we note that if $r > 2n^2$, the statement of the theorem holds vacuously. From now on we handle the case when $r \le 2n^2$.
\begin{proposition}
If $2n^2 \ge r > 64n\log^2 n/\eta$, then $rK_1 + B \leq 867 r n \log n /\eta$.
\end{proposition}
\begin{proof}
We note that $n \leq r \leq n^2$: the former is in the condition of the proposition statement, and the latter is since the lemma is trivial for $r > n^2$.
Since $K = \sqrt{r}$, this implies $\sqrt{n} \leq K \leq n$.

We first focus on $rK_1$.
Since $r \geq 24 n \log^2 n /\eta$, we have that $16 n^2 \exp\left(-\frac{K^2}{16t^*}\right) \leq 16\sqrt{n}$, and therefore $K_1 \leq 16K + 16\sqrt{n} \leq 32K$, and thus $rK_1 \leq 32rK \leq 32rt^*$, where the latter inequality follows since $r \leq 2n^2$ while $t^* \geq n \log n$.

By a similar calculation, we have that $B \leq t^*\left(8K + 8\sqrt{n}\right)^2 + t^* \leq 257K^2t^* = 257rt^*$.

Adding the two, we have that $rK_1 + B \leq 289rt^* = 867r n \log n / \eta$, as desired.
\end{proof}

Hence, (\ref{eq:fr1}) becomes
\begin{align}
&\Pr\left[\exists t < T_K \text{ s.t. } \abss{B_t - B_0} \ge \rad \text{ and } V_{t} \le B \right] \le 2\exp\left(-\frac{\rad^2}{2(rK_1 + B)} \right) \\
&~~~\le 2\exp\left(-\frac{\eta\rad}{1734 n\log n} \right). \label{eq:fr2}
\end{align}

Next we will bound, $\Pr\left[V_{t^*} > B\right]$ which will be useful for obtaining the desired concentration bound from (\ref{eq:fr2}). 
\begin{align}
&\Pr\left[V_{t^*} > B\right] \le \Pr\left[V_{t^*} > B \middle| \forall \: 0 \le t \le t^* \: X_t \in G_K^a(t) \cap N_K^a(t) \right] + \Pr\left[\exists \: 0 \le t \le t^* \: X_t \notin G_K^a(t) \cap N_K^a(t) \right] \notag\\
&\le \Pr\left[\exists \: 0 \le t < t^* \: \text{ s.t. } \Var\left[ B_{t+1}-B_t | X_t \right] > \frac{B}{t^*} \middle| \forall \: 0 \le t \le t^* \: X_t \in G_K^a(t) \cap N_K^a(t) \right] \label{eq:var-bound2}\\
&~~+ \sum_{t=0}^{t^*}\left(\Pr\left[X_t \notin G_K^a(t) \right] + \Pr\left[X_t \notin N_K^a(t) \right]\right) \label{eq:var-bound3}\\
&\le 0 + t^*\left(8n\exp\left( -\frac{K^2}{8t^*} \right) + 8n\exp\left(-\frac{K^2}{16t^*}\right)\right) \le \frac{48n^2\log n}{\eta}\exp\left(-\frac{K^2}{16t^*}\right). \label{eq:var-bound4}
\end{align}
where (\ref{eq:var-bound2}) holds because $V_{t^*} > B$ implies that there exists a $0 \le t \le t^*$ such that $\Var\left[ B_{t+1}-B_t | X_t \right] > B/t^*$, (\ref{eq:var-bound3}) follows by an application of the union bound, and (\ref{eq:var-bound4}) follows from Lemma~\ref{lem:martingale-increment-variance}, Lemma~\ref{lem:goodset-largeprob} and Claim~\ref{clm:nice-sets}.

Now,
\begin{align}
&\Pr\left[\abss{f(X_{t^*}) - \E\left[f(X_{t^*})|X_0\right]}> \rad  \right]  = \Pr\left[\abss{B_{t^*}-B_0}> \rad  \right]\notag\\
&~~\le \Pr\left[\abss{B_{t^*}-B_0}> \rad \text{ and } V_{t^*} \le B \right] + \Pr\left[V_{t^*} > B\right] \label{eq:fr7}\\
&~~\le \Pr\left[\abss{B_{t^*}-B_0}> \rad \text{ and } V_{t^*} \le B \text{ and } t^* < T_K \right] + \Pr[t^* \ge T_K] + \Pr\left[V_{t^*} > B\right] \label{eq:fr3}\\
&~~\le \Pr\left[\left(\exists t \le t^* \text{ s.t. } \abss{B_t - B_0}> \rad \text{ and } V_{t} \le B\right) \text{ and } t^* < T_K \right] + \Pr[t^* \ge T_K] + \Pr\left[V_{t^*} > B\right] \label{eq:fr4}\\
&~~\le \Pr\left[\exists t < T_K \text{ s.t. } \abss{B_t - B_0}> \rad \text{ and } V_{t} \le B \right] + \Pr[t^* \ge T_K] + \Pr\left[V_{t^*} > B\right] \notag\\
&~~\le 2\exp\left(-\frac{\eta\rad}{1734n\log n} \right) + \frac{24n^2\log n}{\eta}\exp\left(-\frac{\eta \rad}{24n\log n} \right) + \Pr\left[V_{t^*} > B\right]  \label{eq:fr5}\\
&~~\le 3\exp\left(-\frac{\eta\rad}{1734n\log n} \right) + \Pr\left[V_{t^*} > B\right] \label{eq:fr6} \\
&~~\le 3\exp\left(-\frac{\eta\rad}{1734n\log n} \right) + \frac{48n^2\log n}{\eta}\exp\left(-\frac{\eta \rad}{48n \log n}\right) \label{eq:fr8}\\
&~~\le 5\exp\left(-\frac{\eta\rad}{1734n\log n} \right). \label{eq:fr9}
\end{align}
where (\ref{eq:fr7}) and (\ref{eq:fr3}) follow from the fact that $\Pr[A] \le \Pr[A \cap B] + \Pr[\neg B]$, (\ref{eq:fr4}) follows from the fact that $\Pr[A] \le \Pr[A \cup B]$, (\ref{eq:fr5}) follows from (\ref{eq:fr2}) and from Lemma \ref{lem:bilinear-stopping-time-large}, (\ref{eq:fr6}) holds because $\rad \ge 300n\log^2 n/\eta$, (\ref{eq:fr8}) follows from (\ref{eq:var-bound4}) and (\ref{eq:fr9}) again holds because $\rad \ge 300n\log^2 n/\eta$.
Note that we have implicitly assumed that $\eta > 1/n$, since otherwise the concentration bounds obtained are trivial.

\end{proof}

We note that this statement conditions on an initial state $X_0$.
In order to remove this conditioning, we must argue that after $t^*$ steps, the Glauber dynamics have mixed, and $X_{t^*}$ is very close in total variation distance to the true Ising model, for any starting point $x_0$.

We use Lemma \ref{lem:close-to-stationary} to remove the conditioning in our previous tail bound, which implies Theorem~\ref{thm:bilinear}.
\begin{lemma}
\label{lem:bilinear}
For all $\rad \ge 300n\log^2 n/\eta + 2$ and $n$ sufficiently large,
$$\Pr\left[ \abss{f_a(X_{t^*}) - \E\left[f_a(X_{t^*})\right]} \ge \rad \right] \le 5\exp\left(-\frac{\eta\rad}{1735n\log n} \right).$$
\end{lemma}
\begin{proof}
Since $X_0 \sim p$, we have that, $X_{t^*} \sim p$ as well.
From Lemma \ref{lem:close-to-stationary}, we have that for $t^* = 3\tmix$,
\begin{align*}
&\dtv(X_{t^*}|X_0, X_{t^*}) \le \exp\left(-2n\log n \right) \\
\implies &\abss{\E\left[f_a(X_{t^*})|X_0\right] - \E\left[f_a(X_{t^*})\right]} \le 2n^2\exp(- 2n \log n) \le 2\exp(-n).
\end{align*}
Now,
\begin{align}
&\Pr\left[ \abss{f_a(X_{t^*}) - \E\left[f_a(X_{t^*})\right]}> ge \rad \right] \le \Pr\left[ \abss{f_a(X_{t^*}) - \E\left[f_a(X_{t^*}) | X_0 \right]} \ge \rad-2\exp(-n) \right] \notag\\
&\le 5\exp\left(-\frac{\eta(\rad-2)}{1734n\log^2 n} \right) \notag\\
&\le 5\exp\left(-\frac{\eta\rad}{1735n\log^2 n} \right), \label{eq:bilinear1}
\end{align}
where (\ref{eq:bilinear1}) holds for sufficiently large $n$.
\end{proof}

\subsection{Concentration under an External Field}
\label{sec:bilinear-concentration-external-field}
Under an external field, not all bilinear functions concentrate nicely even in the high temperature regime. This can be seen easily, for instance, in the case of $f_a(X) = \sum_{u \neq v} X_uX_v$. On an empty graph with a uniform external field $h$ on each node, $\Var(f_a(X)) = c(h)n^3$ (where $c(\cdot)$ is a function depending only on $h$). Hence the best scale of concentration one could hope for is at a distance $n^{1.5}$ from $\E[f_a(X)]$. However, tighter concentration akin to the one we achieve when there is no external field can be shown for classes of appropriately centered bilinear functions. We briefly describe the reason a non-centered function such as the one above doesn't concentrate and then argue at a high level how a correctly `centered' function has sharper tails.
To see where our framework fails when trying to show concentration of measure for arbitrary bilinear functions, let us look at $f_a(X) = \sum_{u \neq v} X_uX_v$. Under an external field, the linear functions associated with taking a step along the censored Glauber dynamics starting at $X$, are no longer zero mean. Although linear functions still concentrate around their expectation with a radius of $\tilde{O}(\sqrt{n})$, the expectations can be of the order $\Omega(n)$. Hence we \emph{can't} use concentration of linear functions to argue that $\abss{f_a^v(X)} \approx O(\sqrt{n})$. And the example described above shows that indeed the variance is higher for this function and the best concentration of measure one could hope to show has tails bounds which kick in at deviations of $O(n^{1.5})$ from the mean. To get stronger tails, the fix is to center our bilinear functions so that the linear functions arising from $f_a(X)-f_a(X')$ are zero mean thereby enabling application of concentration at radius $\tilde{O}(\sqrt{n})$ on the quantity $\abss{f_a^v(X)-\E[f_a^v(X)]}$ rather than having to bound $\abss{f_a^v(X)}$ itself. There are multiple ways to achieve this. We present two simple and natural ways of doing so in Theorem \ref{thm:bilinear-concentration-external}.
\begin{theorem}[Concentration of Measure for Bilinear Functions Under an External Field]
	\label{thm:bilinear-concentration-external}
	\begin{enumerate}
		\item Bilinear functions on the Ising model of the form $f_a(X) = \sum_{u,v} a_{uv}(X_u-\E[X_u])(X_v -\E[X_v])$ satisfy the following inequality at high temperature. There exist absolute constants $c$ and $c'$ such that, for $\rad \ge cn\log^2 n/\eta$,
		$$\Pr\left[\abss{f_a(X)-\E[f_a(X)]} \ge \rad \right] \le 4\exp\left(-\frac{\rad}{c'n\log n} \right).$$
		
		\item Bilinear functions on the Ising model of the form $f_a(X^{(1)},X^{(2)}) = \sum_{u,v} a_{uv}(X_u^{(1)}-X_u^{(2)})(X_v^{(1)} -X_v^{(2)})$, where $X^{(1)},X^{(2)}$ are two i.i.d samples from the Ising model, satisfy the following inequality at high temperature. There exist absolute constants $c$ and $c'$ such that, for $\rad \ge cn\log^2 n/\eta$,
		$$\Pr\left[\abss{f_a(X^{(1)},X^{(2)})-\E[f_a(X^{(1)},X^{(2)})]} \ge \rad \right] \le 4\exp\left(-\frac{\rad}{c'n\log n} \right).$$
	\end{enumerate}
\end{theorem}
\begin{proof}
	Since most of the proof follows along similar lines as that of the case with no external field, we briefly sketch the outline and highlight the major differences here. The calculations are straightforward to verify.
	The first step would be to prove a version of Lemma~\ref{lem:linear} for linear functions in the case of external field with the main difference being that wherever we had $f(x) = \sum_{v} a_v x_v$ before, we replace it with $f(x) = \sum_{v} a_v (x_v - \E[X_v])$. With this replacement it can be seen that the lemma follows in the presence of an external field. 
	Next, we proceed to define a martingale sequence in the same way as was done in the case without external field. The linear functions in the stopping time definition are now replaced with their centered versions (i.e. are made zero mean). When studying the martingale differences we end up having to bound the difference in the expected value of our function at some future time $t^*$, conditioning on starting at two different starting states $X$ and $X'$, where $X'$ is obtained by doing one step of the Glauber dynamics from $X$.  
	%
	%In turn, to bound this expected difference, 
	%
	%look at $\abss{f_a(X)-f_a(X')}$ when $X'$ is obtained by taking a step of the Glauber dynamics starting from $X$. 
	For the first style of centered functions listed in the theorem statement, if we unravel our bounding procedure we end up needing to bound functions of the form $\abss{2\sum_{u \neq v} a_{uv}(X_u-\E[X_u] )}$ for different $v$'s. The linear function inside the absolute value is zero mean and hence we can bound it in absolute value with high probability using concentration of measure for linear functions. 
	Similarly, for the second style of functions in the theorem statement, we end up needing to bound functions of the form $\abss{2\sum_{u \neq v} a_{uv}(X_u^{(1)}-X_u^{(2)})}$ which again are zero mean, and we can still use concentration of measure of linear functions to bound them. The rest of the proof follows in the same way as in the case without external field. \end{proof}

\subsection{An Exponential Tail is Inherent for Bilinear Statistics}
\label{sec:bilinear-tightness}
In this section, we show that our tail bound of Theorem \ref{thm:bilinear} is asymptotically tight upto a $\log n$ factor in the radius of concentration. Informally this means that exponential tails are the best one could hope to get for bilinear functions and sharper tails, (e.g. a Gaussian tail: $\exp(-\rad^2/n^2)$), can't be obtained. The tightness will follow from the following theorem which shows that the tail given by the Chernoff bound is asymptotically tight for sums of bounded i.i.d. random variables.
\begin{theorem}
\label{thm:chernoff-tight}
(Folklore)
Let $X_1,\ldots,X_n$ be i.i.d. samples from $Ber(1/2)$ and let $g(X) = \sum_{i=1}^n X_i$. Then for any $\rad > 0$,
\begin{align*}
\Pr\left[g(X)-\E[g(X)] > \rad \right] \ge \exp\left(- \frac{9\rad^2}{n} \right),\\
\Pr\left[g(X)-\E[g(X)] < -\rad \right] \ge \exp\left(- \frac{9\rad^2}{n} \right).
\end{align*}
\end{theorem}
One possible proof of the above theorem follows by the application of Stirling's inequalities.
Now, consider the bilinear function $f(X) = \sum_{u \neq v} X_uX_v$ on an Ising model on an empty graph. Hence for each $u$, $(X_u+1)/2 \sim Ber(1/2)$ independently. Note that $\E[f(X)] = 0$.
\begin{align*}
2f(X) &= \left(\sum_u X_u \right)^2 - n \\
\implies \Pr\left[\abss{f(X)} > \rad \right] &\ge \Pr\left[f(X) > \rad \right] \\
&= \Pr\left[ \abss{\sum_u X_u}  > \sqrt{\rad/2 + n}\right] = \Pr\left[ \abss{\sum_u \frac{X_u+1}{2} - \frac{n}{2}}  > \frac{\sqrt{\rad/2 + n}}{2}\right]\\
&\ge \exp\left( -\frac{9(\rad/2+n)}{4n} \right) \ge \exp(-9/4)\exp\left( -\frac{9\rad}{8n} \right),
\end{align*}
where the last inequality follows from Theorem \ref{thm:chernoff-tight}.
This shows that the tail bound obtained from Theorem \ref{thm:bilinear} is asymptotically nearly-tight (up to a $O(\log n)$ factor in the radius of concentration and $O(1/\log n)$ factor in the exponent of the tail bound).

	\section{Concentration of Measure for $d$-linear Functions}
\label{sec:multilinear}

In this section we show concentration of measure for $d$-linear functions on an Ising model in high temperature, when $d \ge 3$. Again, we will focus on the setting with \emph{no external field}. Although we will follow a recipe similar to that used for bilinear functions, the proof is more involved and requires some new definitions and tools. The proof will proceed by induction on the degree $d$. Due to the proof being more involved, for ease of exposition, we present the proof of Theorem~\ref{thm:multilinear} without explicit values for constants.

Our main theorem statement is the following:
\begin{theorem}
\label{thm:multilinear}
Consider any degree-$d$ multilinear function
$$f_a(x) = \sum_{U \subseteq V : |U| = d} a_{U} \prod_{u \in U} x_u$$ 
on an Ising model $p$ (defined on a graph $G=(V,E)$ such that $\abss{V} = n$) in $\eta$-high-temperature regime with no external field.
Let $\|a\|_\infty = \max_{U \subseteq V : |U| = d} |a_U|$.
There exist constants $C_1 = C_1(d) > 0$ and $C_2 = C_2(d) > 0 $ depending only on $d$, such that
if $X \sim p$, then for any $r \geq C_1\|a\|_\infty (n \log^2 n/\eta)^{d/2}$, we have
$$\Pr\left[ \abss{f_a(X) - \E\left[f_a(X)\right]}> \rad \right] \le 2\exp\left(-\frac{\eta\rad^{2/d}}{C_2\|a\|_\infty^{2/d} n\log n} \right).$$
\end{theorem}
Note that, like our bilinear theorem statement, this statement is phrased for multilinear functions of degree exactly equal to $d$.
This is for convenience of notation in our proof.
A general purpose theorem for all degree-$d$ multilinear functions can be obtained by simply partitioning the terms based on their degree, and applying this theorem for each degree from $1$ to $d$.
This will not incur significant costs in the concentration bound, as the terms of order lower than $d$ have much tighter radii of concentration.

\begin{remark}
\label{rem:multilinear-tight}
	The bound presented in Theorem~\ref{thm:multilinear} is asymptotically tight up to an $O_d(\log^d n)$ factor in the radius of concentration and a $O(1/\log n)$ factor in the exponent of the tail bound. This can be shown via an argument similar to that employed in Section~\ref{sec:bilinear-tightness}. In particular, for an Ising model on an empty graph (where each node is completely independent of the others), we have for the $d$-linear function $f(x) = \sum_{U \subseteq V : |U| = d} \prod_{u \in U} x_u$, the following inequality for a big enough constant $C(d)$
	\begin{align*}
	\Pr\left[\abss{f(X) - \E\left[f(X)\right]}> \rad\right] \ge 2\exp\left(-\frac{\rad^{2/d}}{C(d)n}\right).
	\end{align*}
\end{remark}

\subsection{Overview of the Technique}
Our approach uses induction and is similar to the one used for bilinear functions. To show concentration for $d$-linear functions we will use the concentration of $(d-1)$-linear functions together with Freedman's martingale inequality.

Consider the following process: Sample $X_0 \sim p$ from the Ising model of interest. Starting at $X_0$, run the Glauber dynamics associated with $p$ for $t^* = (d+1)\tmix$ steps. We will study the target quantity, $\Pr\left[\abss{f_a(X_{t^*}) - \E[f_a(X_{t^*}) | X_0]} > K \right]$, by defining a martingale sequence similar to Definition~\ref{def:bilinear-doob-martingale}. However, to bound the increments of the martingale for $d$-linear functions we will require an induction hypothesis which is more involved. The reason is that with higher degree multilinear functions ($d > 2$), the argument for bounding increments of the martingale sequence runs into multilinear terms which are a function of not just a single instance of the dynamics $X_t$, but also of the configuration obtained from the coupled run, $X_t'$. We call such multilinear terms \emph{hybrid} terms and multilinear functions involving \emph{hybrid} terms as \emph{hybrid} multilinear functions henceforth. Since the two runs (of the Glauber dynamics) are coupled greedily to maximize the probability of agreement and they start with a small Hamming distance from each other ($\le 1$), these hybrid terms behave very similar to the non-hybrid multilinear terms. Showing that their behavior is similar, however, requires some supplementary statements about them which are presented in Theorem \ref{thm:hybrid-supplement-induction}. Theorem \ref{thm:hybrid-supplement-induction} will also be proven using induction. In addition to the martingale technique of Section \ref{sec:bilinear}, an ingredient that is crucial to the proving concentration for $d \ge 3$ is a bound on the magnitude of the $(d-1)$-order marginals of the Ising model (i.e. terms of the form $\abss{\E[X_{u_1} X_{u_2} \ldots X_{u_{d-1}}]}$). This is because when studying degree $d \ge 3$ functions we find ourselves having to bound expected values of degree $d-1$ multilinear functions on the Ising model. A naive bound of $O_d(n^{d-1})$ can be argued for these functions but by exploiting the fact that we are in high temperature, we can show a bound of $O_d(n^{(d-1)/2})$. When $d=2$, $(d-1)$-linear functions are just linear functions which are zero mean. However, for $d \ge 3$, this is not the case. Hence, we first need to prove this desired bound on the marginals of an Ising model in high temperature. We will do it by a set of tools which are quite different from those used in Section \ref{sec:bilinear}.\\

\noindent In Section \ref{sec:multilinear-setup} we state some lemmata and definitions.
We then show a bound on the expected value of $d$-linear functions on the Ising model in high temperature (Section \ref{sec:marginals-bound}). 
We proceed by showing our main result (Theorem~\ref{thm:multilinear-induction}) in Section \ref{sec:multilinear-induction}. 
This result requires us to relate the expected values and tail probabilities of \emph{hybrid} terms to those of non-hybrid terms, which we do as Theorem~\ref{thm:hybrid-supplement-induction} in Section~\ref{sec:hybrid}.

\subsection{Setup}
\label{sec:multilinear-setup}
We will now proceed with the setup of our argument for concentration of $d$-linear functions.

Recall from Claim \ref{clm:glauber-step-characterization} the linear functions that arose when looking at the difference in the value of a bilinear function due to a step of the Glauber dynamics. In a similar vein, we define a family $F_a^d$ of multilinear functions on the Ising model of degree $\le d-1$ associated with any $d$-linear function $f_a(x)$. 
\begin{definition}
	\label{def:function-family}
	\begin{align}
	&F_a^d = \bigcup_{l=0}^{d-1} F_a(l)\ \mathrm{ where }\\
	&F_a^d(l) = \{ f_a^{v_1,v_2,\ldots,v_{d-l}} \: |  \: \forall \: \text{ distinct }v_1,v_2,\ldots,v_{d-l} \in V  \}\ \mathrm{ and}\\
	&f_a^{v_1,v_2,\ldots,v_k}(x) = \sum_{ u_1,u_2,\ldots,u_{d-k} \in V \setminus \{v_1,v_2,\ldots,v_k \} } a_{u_1 u_2 \ldots u_{d-k} v_1 v_2 \ldots v_k} X_{u_1}X_{u_2}\ldots X_{u_{d-k}}.
	\end{align}
\end{definition}
In the set of functions defined above, the degree $d-1$ functions arise (up to scaling) from looking at the difference in values of $f_a(X)$ when a single step of the Glauber dynamics is taken. 
More generally, the degree $l-1$ functions in the definition arise when looking at the difference in values of a degree $l$ function from $F_a^d(l)$ when a single step of the Glauber dynamics is taken.

We will also need to generalize the greedy coupling (Definition \ref{def:greedy-coupling}) used in Section \ref{sec:bilinear} to couple two runs of the Glauber dynamics. The generalization will provide a way of coupling an arbitrary number of runs of the Glauber dynamics on a common Ising model $p$.
\begin{definition}[The $k$-Greedy Coupling]
	\label{def:generalized-greedy-coupling}
	
	Given an Ising model $p$ in high temperature, for any $k > 0$, consider the following process: Let $x_0^{(1)}, x_0^{(2)}, \ldots , x_0^{(k)} \in \Omega$ be $k$ starting configurations. Run $k$ instances of the Glauber dynamics associated with $p$ with the $i^{th}$ instance starting at state $X_0^{(i)} = x_0^{(i)}$. Let the sequence of states observed in the $i^{th}$ run of the dynamics be $X_0^{(i)}, X_1^{(i)}, X_2^{(i)},\ldots$. Couple the $k$ runs in the following way: At each time step $t$ choose a vertex $v \in V$ uniformly at random to update in all of the $k$ runs. Let $p^i$ denote the probability that the $i^{th}$ Glauber dynamics instance sets $X_{t,v}^{(i)} = 1$. Let $p_1 \le p_2 \le \ldots p_k$ be a rearrangement of the $p^i$ values in increasing order. Also let $p_0 = 0$ and $p_{k+1} = 1$. Draw a number $x$ uniformly at random from $[0,1]$ and couple the updates according to the following rule:\\
	
	\noindent If $x \in [p_l,p_{l+1}]$ for some $0 \le l \le k$, set $X_{t,v}^{(i)} = -1$ for all $1 \le i \le l$ and $X_{t,v}^{(i)} = 1$ for all $l < i \le k$.\\
	
	\noindent We call this coupling the generalized greedy coupling of the $k$ runs or the $k$-greedy coupling.
\end{definition}

Now we list some properties the generalized greedy coupling (Definition \ref{def:generalized-greedy-coupling}) satisfies.
\begin{lemma}[Properties of the $k$-Greedy Coupling]
	\label{lem:generalized-greedy-properties}
	The $k$-Greedy coupling (Definition \ref{def:generalized-greedy-coupling}) is a valid coupling of $k$ runs of Glauber dynamics with the following properties.
	\begin{enumerate}
		\item If $X_0^{(i)} \sim p$, then $X_{t}^{(i)} \sim p$ for all $t \ge 0$ and for all $1 \le i \le k$.
		\item If $p$ is an Ising model in $\eta$-high temperature, for any pair of runs $i \ne j$, 
		$$\E\left[ \dh(X_t^{(i)}, X_t^{(j)}) \middle| (X_0^{(i)}, X_0^{(j)}) \right] \le \left(1 - \frac{\eta}{n} \right)^t \dh(X_0^{(i)}, X_0^{(j)}).$$
		That is, the joint distribution of any two of the runs is a greedy coupling as described in Definition \ref{def:greedy-coupling}.
		\item For any pair of runs $i \ne j$, the distribution of $X_t^{(i)}$, for any $t \ge 0$, conditioned on $X_0^{(i)}$ is independent of $X_0^{(j)}$.
		
	\end{enumerate}
\end{lemma}
\begin{proof}
	First we will argue that the $k$-greedy coupling is a valid coupling. Consider the marginal distribution of any one of the $k$ runs: $X_0^{(j)},X_1^{(j)},\ldots$. The process of generating $X_{t+1}^{(j)}$ from $X_t^{(j)}$ corresponds precisely to a step of the Glauber dynamics. Firstly, the sampling of a node among all choices is common to all runs and hence also to run $j$. Secondly, the update probabilities for the selected node are exactly what Glauber dynamics would have prescribed. Hence, it is a valid coupling of the $k$ runs. Since $p$ is the stationary distribution corresponding to all the $k$ runs, $X_0^{(i)} \sim p \implies X_t^{(i)} \sim p$ for all $t \ge 0$. We will now argue that any pair of runs $i \ne j$ are coupled according to the greedy coupling of Definition (\ref{def:greedy-coupling}). Since the node to be updated in any step is chosen to be the same for all runs it is also the same for runs $i$ and $j$. Moreover, the updates of the selected node in runs $i$ and $j$ are coupled in precisely the same way as they were under the greedy coupling. Hence, the $k$-greedy coupling is a greedy coupling for any pair of runs $i$ and $j$. Hence, from Lemma \ref{lem:greedy-properties}, we have
	$$\E\left[ \dh(X_t^{(i)}, X_t^{(j)}) \middle| (X_0^{(i)}, X_0^{(j)}) \right] \le \left(1 - \frac{\eta}{n} \right)^t \dh(X_0^{(i)}, X_0^{(j)}).$$
	Also from Lemma \ref{lem:greedy-properties}, we have that the distribution of $X_t^{(i)}$, for any $t \ge 0$, conditioned on $X_0^{(i)}$ is independent of $X_0^{(j)}$. 
\end{proof}

The $k$-greedy coupling we have defined above will be useful in showing the following property about the concentration of the Hamming distance between two greedily coupled runs which is stated as Lemma \ref{lem:hamming-concentration}.
\begin{lemma}
	\label{lem:hamming-concentration}
	Let $x_0,y_0 \in \Omega$ be two configurations for a high temperature Ising model $p$ on $n$ nodes. Let $\{X_t\}_{t\ge 0}, \{Y_t\}_{t\ge 0}$ be two runs of Glauber dynamics associated with $p$, coupled greedily with $X_0 = x_0, Y_0=y_0$. Then, for any integer $t >0$ and any real $K > 0$,
	$$\Pr\left[ \abss{\dh(X_t,Y_t) - \E\left[ \dh(X_t,Y_t) | X_0,Y_0 \right]} > K \middle| X_0=x_0,Y_0=y_0 \right] \le 2\exp\left(-\frac{K^2}{16t}\right).$$
\end{lemma}
\begin{proof}
	We will use Azuma's inequality (Lemma \ref{lem:azuma}).
	Consider the Doob martingale associated with $\dh(X_t,Y_t)$ (similar to that of Definition~\ref{def:bilinear-doob-martingale})but now defined using the greedily coupled dynamics), parameterized by $x_0,y_0$. The $i^{th}$ term in the martingale sequence is $H_i = \E\left[\dh(X_t,Y_t) | X_i,Y_i \right]$ (where we have used the Markovian property of the Glauber dynamics). We look at $\abss{H_{i+1}-H_i}$ for any $0 < i+1 \le t$.
	\begin{align}
	&\abss{H_{i+1}-H_i} = \abss{\E\left[\dh(X_t,Y_t) | X_{i+1},Y_{i+1} \right] - \E\left[\dh(X_t,Y_t) | X_i,Y_i \right]} \\
	&\le \abss{\E\left[\dh(X_t,Y_t) | X_{i+1},Y_{i+1} \right] - \E\left[\dh(X_{t-1}',Y_{t-1}') | X_i',Y_i' \right]} \label{eq:hamm-conc1} \\
	&~~+ \abss{\E\left[\dh(X_t,Y_t) | X_i,Y_i \right]- \E\left[\dh(X_{t-1},Y_{t-1}) | X_i,Y_i \right]} \label{eq:hamm-conc2}\\
	&\le \abss{\E\left[\dh(X_t,Y_t)-\dh(X_{t-1}',Y_{t-1}') | X_{i+1},Y_{i+1}, X_i',Y_i' \right]} + 2 \label{eq:hamm-conc3}\\
	&\le \E\left[\abss{\dh(X_t,Y_t)-\dh(X_{t-1}',Y_{t-1}')} | X_{i+1},Y_{i+1}, X_i',Y_i' \right] + 2 \notag\\
	&\le \E\left[\abss{\dh(X_t,X_{t-1}')+\dh(Y_{t},Y_{t-1}')} | X_{i+1},Y_{i+1}', X_i',Y_i' \right] + 2 \label{eq:hamm-conc5} \\
	&\le 4 \label{eq:hamm-conc6}
	\end{align}
	where (\ref{eq:hamm-conc1}) and (\ref{eq:hamm-conc2}) follows by adding and subtracting the term $\E\left[\dh(X_{t-1},Y_{t-1}) | X_i,Y_i \right]$ to the difference inside the absolute value. We have also renamed $\E\left[\dh(X_{t-1},Y_{t-1})|X_i,Y_i\right]$ as $\E\left[\dh(X_{t-1}',Y_{t-1}')|X_i',Y_i'\right]$ in (\ref{eq:hamm-conc1}) to avoid notational confusion in the later steps of our bounding, maintaining the understanding that the sequence $\{X_t',Y_t'\}_t$ has the same distribution as $\{X_t,Y_t\}_t$. The first term in (\ref{eq:hamm-conc3}) bounds the term of (\ref{eq:hamm-conc1}) for any valid coupling of the two greedily coupled probability spaces, namely $\{X_t,Y_t\}_{t \ge i}$ and $\{X_t',Y_t'\}_{t \ge i}$. Here we couple them using the $4$-greedy coupling (Definition \ref{def:generalized-greedy-coupling}). Also, by triangle inequality, $\E[|\dh(X_t,Y_t) - \dh(X_{t-1},Y_{t-1})| | X_i,Y_i] \le \E[\dh(X_t,X_{t-1}) + \dh(Y_t,Y_{t-1}) | X_i, Y_i] \le 2$. Hence (\ref{eq:hamm-conc3}) follows. Similarly, (\ref{eq:hamm-conc5}) follows because $\abss{\dh(X_t,Y_{t}) - \dh(X_{t-1}',Y_{t-1}')} \le \dh(X_t,X_{t-1}')+\dh(Y_{t},Y_{t-1}')$ and (\ref{eq:hamm-conc6}) follows because $\E[\dh(X_t,X_{t-1}') | X_{i+1},X_i'] \le \dh(X_{i+1},X_i') \le 1$ (Lemma \ref{lem:greedy-properties}) and similarly $\E[\dh(Y_t,Y_{t-1}') | Y_{i+1},Y_i'] \le \dh(Y_{i+1},Y_i') \le 1$.\\
	
	\noindent Hence by Azuma's inequality applied on the martingale sequence from $0$ to $t$, we get
	$$\Pr\left[ \abss{\dh(X_t,X_t') - \E\left[ \dh(X_t,X_t') | X_0,X_0' \right]} > K \middle| X_0=x_0,X_0'=x_0' \right] \le 2\exp\left(-\frac{K^2}{16t}\right).$$
\end{proof}

\subsection{Bounding Marginals of an Ising Model in High Temperature}
\label{sec:marginals-bound}
The goal of this section will be to obtain a bound on the expected values of the $d$-linear functions under consideration when computed over a sample from a high temperature Ising model.
We start by bounding the marginals of ferromagnetic Ising models ($\th_{uv} \ge 0$ for all $u,v$). We will show later, using a generalization of the Fortuin-Kastelyn (FK) model, that this suffices to yield the result for non-ferromagnetic Ising models as well. The FK model connects bond percolation with the Ising model and offers powerful tools to show stochastic domination inequalities which will enable us to bound the marginals of the Ising model. We assume familiarity with the FK model as described in Chapter 10 of~\cite{RassoulAghaS15}.
\begin{lemma}
	\label{lem:ferro-marginals-bound}
	Consider a ferromagnetic Ising model $p$ ($\th_{uv} \ge 0$ for all $(u,v) \in E$) at high temperature. Let $d$ be a positive integer. We have
	$$\sum_{u_1,\ldots,u_d} \E\left[\prod_{i=1}^d X_{u_i}\right] \le 2\left(\frac{4nd\log n}{\eta}\right)^{d/2} .$$
\end{lemma}
\begin{proof}
%	First, we will argue that each of the marginals within the absolute values are non-negative using the coupling between the ferromagnetic Ising and the FK models.
%	When the degree of the multilinear term $\prod_{i=1}^d X_{u_i}$ is odd, we have $\E\left[\prod_{i=1}^d X_{u_i} \right] = 0$. When the degree $d$ is even, then according to the FK model, $\E\left[\prod_{i=1}^d X_{u_i}\right] = \Pr\left[X_{u_1},X_{u_2},\ldots,X_{u_d} \text{ belong to same cluster }\right] \ge 0$.
%	Therefore, we have that
%	\begin{align}
%	\sum_{u_1,\ldots,u_d} \abss{\E\left[\prod_{i=1}^d X_{u_i}\right]} = \sum_{u_1,\ldots,u_d} \E\left[\prod_{i=1}^d X_{u_i}\right].
%	\end{align}
%	Now,
	We have,
	\begin{align}
	&\left(\sum_{v \in V} X_v\right)^d = \sum_{u_1,\ldots,u_d \in V} \prod_{i=1}^d X_{u_i} \\
	\implies &\E\left[ \sum_{u_1,\ldots,u_d} \prod_{i=1}^d X_{u_i} \right] = \E\left[ \left( \sum_v X_v \right)^d\right].
	\end{align}
	Since we are in high temperature, we have from Lemma \ref{lem:lipschitz-lemma} $\Pr\left[\abss{\sum_v X_v} > K  \right] \le 2\exp\left(-\frac{\eta K^2}{8n}\right)$. Since the maximum value of $\sum_v X_v = n$, we have for all $K>0$,
	\begin{align}
	\E\left[ \left( \sum_v X_v \right)^d\right] \le K^d + 2n^d\exp\left(-\frac{\eta K^2}{8n}\right).
	\end{align}
	Setting $K= 2\sqrt{nd\log n/\eta}$ we get,
	\begin{align}
	\E\left[ \sum_{u_1,\ldots,u_d} \prod_{i=1}^d X_{u_i} \right] = \E\left[ \left( \sum_v X_v \right)^d\right] \le 2\left(\frac{4nd\log n}{\eta}\right)^{d/2}.
	\end{align}
	
\end{proof}

For any Ising model $p$ on graph $G=(V,E)$ with parameter vector represented by $\th$, we associate a ferromagnetic Ising model denoted by $p^+$ defined on the same graph $G$ where all the edges retain their magnitude but are now forced to be ferromagnetic interactions. That is, $\abss{\th_{uv}^p} = \th_{uv}^{p^+}$ for all $u,v$.
We have the following relation between the marginals of $p$ and those of $p^+$.
\begin{lemma}
	\label{lem:ferro-dominates-nonferro}
	Consider any Ising model $p$ defined on $G=(V,E)$. Consider any subset of $k$ nodes $\{u_1,\ldots,u_k\} \subseteq V$. Then, 
	$$\E_{p}\left[ X_{u_1}X_{u_2}\ldots X_{u_k} \right] \le \E_{p^+}\left[ X_{u_1}X_{u_2}\ldots X_{u_k} \right].$$
\end{lemma}
\begin{proof}
	If $k$ is odd, then the quantities on the LHS and RHS are both 0 and hence the Lemma holds.
	To handle the case when $k$ is even, we consider a generalization of the FK model to possible non-ferromagnetic Ising model. 
	This generalization is discussed in detail, for instance, in Newman's paper~\cite{Newman90}. The generalization retains many nice properties of the FK model. In particular, when $k$ is even, $\E_{p}\left[\prod_{i=1}^k X_{u_i}\right] = \Pr\left[X_{u_1},X_{u_2},\ldots,X_{u_k} \text{ belong to same cluster }\right]$ still holds.
	Equation (20) in Section 5 of~\cite{Newman90} notes that the percolation measure associated with any Ising model is stochastically dominated by the measure associated with its corresponding ferromagnetic Ising model. 
	%	A useful property of the FK-model generalization (a specific case of which is noted in Equation (8) of~\cite{Newman90}) is that 
	%	$$\E_{p}\left[ X_{u_1}X_{u_2}\ldots X_{u_k} \right] = \Pr_p\left[ X_{u_1},X_{u_2} \ldots X_{u_k} \text{ belong to the same cluster } \right].$$
	%	The stochastic dominance noted above
	A consequence of this stochastic dominance, noted in homework problem 3 in the chapter on Phase Transitions by Griffiths in~\cite{DeWittS71}, is the desired inequality
	$$\E_{p}\left[ X_{u_1}X_{u_2}\ldots X_{u_k} \right] \le \E_{p^+}\left[ X_{u_1}X_{u_2}\ldots X_{u_k} \right].$$
	% Another way of viewing why this holds is as follows: stochastic dominance of random cluster measure corresponding to the ferromagnetic model over the non-ferromanetic one is defined over the lattice. This directly implies that the set of all configurations in which the nodes X_{u_1}X_{u_2}\ldots X_{u_k} lie in the same cluster has a higher probability under the ferromagnetic model. Which yields the desired result.
\end{proof}

\begin{lemma}
	\label{lem:ferro-dominates}
	Consider any Ising model $p$ defined on $G=(V,E)$. Consider any subset of $k$ nodes $\{u_1,\ldots,u_k\} \subseteq V$. Then,
	$$\abss{\E_{p}\left[ X_{u_1}X_{u_2}\ldots X_{u_k} \right]} \le \E_{p^+}\left[ X_{u_1}X_{u_2}\ldots X_{u_k} \right].$$
\end{lemma}
\begin{proof}
	If we show that $-\E_{p}\left[ X_{u_1}X_{u_2}\ldots X_{u_k} \right] \le \E_{p^+}\left[ X_{u_1}X_{u_2}\ldots X_{u_k} \right]$, then together with Lemma \ref{lem:ferro-dominates-nonferro} we get the desired result. 
	To show the above inequality we build an Ising model $\tilde{p}$ and apply Lemma \ref{lem:ferro-dominates-nonferro} to it. $\tilde{p}$ is defined as follows. The set of vertices $V$ on which $p$ is defined is augmented with $k$ dummy vertices, $\tilde{u}_1,\tilde{u}_2,\ldots,\tilde{u}_k$ and the set of edges $E$ is augmented by the addition of the set of edges $\tilde{E}$:
	$$\tilde{E} = \left\{ (u_i,\tilde{u}_i) \text{ for } i=1,2,\ldots,k  \right\}.$$
	The parameters for the new edges are all set to $+ \infty$ except for the edge $(u_1,\tilde{u}_1)$ whose parameter is set to $-\infty$. 
	Under this construction, we have that
	\begin{align}
	\E_{\tilde{p}}\left[ X_{\tilde{u}_1}X_{\tilde{u}_2}\ldots X_{\tilde{u}_k} \right]  = -\E_{p}\left[ X_{u_1}X_{u_2}\ldots X_{u_k} \right]. \label{eq:fk1}
	\end{align}
	From Lemma \ref{lem:ferro-dominates-nonferro}, we have
	\begin{align}
	\E_{\tilde{p}}\left[ X_{\tilde{u}_1}X_{\tilde{u}_2}\ldots X_{\tilde{u}_k} \right] \le \E_{\tilde{p}^+}\left[ X_{\tilde{u}_1}X_{\tilde{u}_2}\ldots X_{\tilde{u}_k} \right] \label{eq:fk2}
	\end{align}
	And since all edges of the form $(u_i,\tilde{u_i})$ for $i=1,2,\ldots,k$ have a parameter of $+ \infty$ under $\tilde{p}^+$, we have
	\begin{align}
	&\E_{\tilde{p}^+}\left[ X_{\tilde{u}_1}X_{\tilde{u}_2}\ldots X_{\tilde{u}_k} \right] = \E_{\tilde{p}^+}\left[ X_{u_1}X_{u_2}\ldots X_{u_k} \right] \label{eq:fk3}
	\end{align}
	Finally, we observe that this construction doesn't change the values of any of the original marginals. In particular,
	\begin{align}
	\E_{\tilde{p}^+}\left[ X_{u_1}X_{u_2}\ldots X_{u_k} \right]  = \E_{p^+}\left[ X_{u_1}X_{u_2}\ldots X_{u_k} \right]. \label{eq:fk4}
	\end{align}
	(\ref{eq:fk1}),(\ref{eq:fk2}),(\ref{eq:fk3}) and (\ref{eq:fk4}) combined give us the desired result.
\end{proof}

Lemma \ref{lem:ferro-marginals-bound} together with Lemma \ref{lem:ferro-dominates} gives Corollary \ref{cor:marginal-bound}.
\begin{corollary}
	\label{cor:marginal-bound}
	Consider any Ising model $p$ at high temperature. Let $d$ be a positive integer. We have
	$$\abss{\sum_{u_1,\ldots,u_d} \E_{p}[X_{u_1}X_{u_2}\ldots,X_{u_d}]} \le 2\left(\frac{4nd\log n}{\eta}\right)^{d/2} .$$
\end{corollary}
\begin{proof}
	We have,
	\begin{align*}
	&\abss{\sum_{u_1,\ldots,u_d} \E_{p}[X_{u_1}X_{u_2}\ldots,X_{u_d}]} \le \sum_{u_1,\ldots,u_d} \abss{\E_{p}[X_{u_1}X_{u_2}\ldots,X_{u_d}]} \\
	&\le \sum_{u_1,\ldots,u_d} \E_{p^+}[X_{u_1}X_{u_2}\ldots,X_{u_d}] \le 2\left(\frac{4nd\log n}{\eta}\right)^{d/2}.
	\end{align*}
\end{proof}

%--------------------------------------------------------
%--------------------------------------------------------
%--------------------------------------------------------

\subsection{Main Theorem Statement for $d$-Linear Functions}
\label{sec:multilinear-induction}
We are now ready to show concentration of measure for $d$-linear functions.
% Good set definition. **IMP: Degree in superscript of G corresponds to degree of $f_a$. Hence all functions considered within G are of degree <= d-1
First, we define a notion of a `good' set of configurations corresponding to a $d$-linear function $f_a(x)$ similar to how it was defined in Section~\ref{sec:bilinear}. Doing so will help us define a stopping time for the martingale sequence we consider later on in the argument.
For any multilinear function $f_a(x)$ of degree $d$, $K > 0$, and $t_1 \ge t$, define the set $G_K^{a,d}(t_1,t)$ to be the following set of configurations:
\small
\begin{align}
&G_K^{a,d}(t_1,t) = \left\{ x_t \in \Omega \middle| \: \forall 1\le l \le d-1,\: \forall f \in F_a^d(l) \: \max\{\abss{\E[f(X_{t_1})|X_t=x_t]}, \abss{\E[f(X_{t_1-1})|X_t=x_t]}\} \le K^{l/(d-1)}\right\} \notag\\
&~ \bigcap \left\{ x_t \middle| \: \forall 1\le l \le d-1,\: \forall f \in F_a^d(l) \: \Pr\left[ \abss{f(X_{t_1}) - \E\left[ f(X_{t_1}) | X_t \right]} > K^{l/(d-1)} \middle| X_t=x_t  \right] \le 2\exp\left(-\frac{K^{2/(d-1)}}{c_1(l)t_1} \right) \right\} \label{eq:multilinear-Gt}\\
&~ \bigcap \left\{ x_t \middle| \: \forall 1\le l \le d-1,\: \forall f \in F_a^d(l) \: \Pr\left[ \abss{f(X_{t_1-1}) - \E\left[ f(X_{t_1-1}) | X_t \right]} > K^{l/(d-1)} \middle| X_t=x_t  \right] \le 2\exp\left(-\frac{K^{2/(d-1)}}{c_1(l)t_1} \right) \right\} \notag
\end{align}
\normalsize
where $\E[f(X_{t_1}) | X_{t}]$, is defined as 0 for $t > t_1$ for any function $f$ and $c_1(l)>0$ is a function of $l$ which is sufficiently large. The definition may seem complicated at this moment but its usefulness will become more apparent once we delve into the proof of Theorem \ref{thm:multilinear-induction}.

$G_K^{a,d}(t_1,t)$ was deliberately constructed so as to satisfy the following layering property which will be very useful in the inductive argument of Theorem \ref{thm:multilinear-induction}.
The following corollary is immediate from the definition of $G_K^{a,d}(t_1,t)$.
\begin{corollary}
	\label{cor:goodset-layering}
	Suppose $f_a(x)$ is a $d$-linear function and $x \in G_K^{a,d}(t_1,t)$. For any $v \in V$, let $a^v$ represent the coefficient vector of the $(d-1)$-linear function $f_a^v(x)$. Then $x \in G_{\hat{K}}^{a^v,d-1}(t_1,t)$ where $\hat{K} = K^{(d-2)/(d-1)}$.
\end{corollary}

We will obtain the desired concentration bound by showing that the following set of statements hold for any $d$-linear function $f_a(X_{t^*})$ (where $d$ is a constant) with bounded coefficients ($\inftynorm{a} \le 1 $). To show Statement (\ref{stmt:freedman-conc}) of Theorem~\ref{thm:multilinear-induction}, we will use Theorem~\ref{thm:hybrid-supplement-induction}. %However, to show Theorem~\ref{thm:hybrid-supplement-induction} we will appeal to statements of Theorem~\ref{thm:multilinear-induction} again. Here we describe in more detail the structure of these dependencies and argue that it is not a cyclic dependency, but rather a layered one. In particular, to show Statement (\ref{stmt:freedman-conc}) of Theorem~\ref{thm:multilinear-induction} for a $d+1$-linear function $f_a(x)$, we employ Theorem~\ref{thm:hybrid-supplement-induction} for the $d+1$-linear function $f_a(x)$. In turn, to show Theorem~\ref{thm:hybrid-supplement-induction} for the $d+1$-linear function $f_a(x)$, we use Statements (\ref{stmt:f)
\begin{theorem}
	\label{thm:multilinear-induction}
	Consider an Ising model $p$ in the $\eta$-high temperature regime. Let $\tmix = n\log n / \eta$ denote the mixing time of the Glauber dynamics associated with $p$. Let $f_a : \Omega \rightarrow \mathbb{R}$ be any $d$-linear function for some $d \ge 1$, such that $f_a(x) = \sum_{u_1, u_2, \ldots, u_d} a_{u_1u_2 \ldots u_d} x_{u_1}x_{u_2}\ldots x_{u_d}$ where $a \in [-1,1]^{V \choose d}$. Let $2\tmix \le t^* \le (n+1)\tmix$.
	\begin{enumerate}
		%-------------------------
		\item Let $X_0 \sim p$. Consider a run of the Glauber dynamics associated with $p$ running for $t^*$ steps: $X_0,X_1,\ldots,X_{t^*}$. For any $0 \le t_0 \le t^*$, there exist $c(d),c_2(d)>0$ which are increasing functions of $d$ only, such that, for any $\rad > c(d)(n\log^2 n/\eta)^{d/2}$, we have,
		$$\Pr\left[ \abss{f_a(X_{t^*}) - \E[f_a(X_{t^*}) | X_{t_0}] } \ge \rad \right] \le 2\exp\left( -\frac{\rad^{2/d}}{c_2(d)t^*} \right).$$ \label{stmt:freedman-conc}

		%-------------------------
		\item If $X \sim p$ is a sample from the Ising model, there exist $c(d),c_3(d)>0$ which are increasing functions of $d$ only, such that, for any $\rad > c(d)(n\log^2 n/\eta)^{d/2}$, 
		$$\Pr\left[\abss{f_a(X) - \E[f_a(X)]} \ge \rad  \right] \le 2\exp\left( -\frac{r^{2/d}\eta}{c_3(d)n\log n} \right).$$ \label{stmt:main-conc}
		
		%-------------------------
		\item For any $0 \le t_0 \le t^*$, there exist $c(d),c_4(d)>0$ which are increasing functions of $d$ alone, such that, for any $\rad > c(d)(n\log^2 n/\eta)^{d/2}$, 
		$$\Pr\left[ \abss{\E\left[ f_a(X_{t^*}) | X_{t_0} \right]} \ge \rad \right] \le 2\exp\left( - \frac{\rad^{2/d}}{c_4(d)t^*}\right).$$ \label{stmt:conc-of-cond-expectation}
	\end{enumerate}
\end{theorem}
\begin{proof}
	The proof will proceed by induction on $d$.
	
	%-------------------------
	% Base case: d=1
	%-------------------------
	\noindent \textbf{Base Case $d=1$:} Statement~\ref{stmt:freedman-conc} follows from Statement 3 of Lemma \ref{lem:linear}.
	Statement~\ref{stmt:main-conc} of the Theorem follows immediately from Lemma \ref{lem:lipschitz-lemma} applied to linear functions.
	Statement~\ref{stmt:conc-of-cond-expectation} follows from Statement 2 of Lemma \ref{lem:linear}. 
	Hence, we have shown that Theorem~\ref{thm:multilinear-induction} holds when $d=1$.\\

	%----------------------------
	% Induction from $d$ to $d+1$
	%----------------------------
	\noindent \textbf{Inductive Hypothesis:} Suppose the statements of the theorem hold for some $d > 1$. We will now show that they hold for $d+1$.
	
	%-------------------------
	% Stmt freedman-conc uses Freedman's inequality
	%-------------------------
	\noindent \textbf{Statement~\ref{stmt:freedman-conc}:} We aim to show this statement for $d+1$-linear functions. We will use Freedman's inequality in a similar manner as was done in Section \ref{sec:bilinear}. We begin by defining a martingale sequence associated with $f_a(x)$.
	\begin{definition}[The $d$-Linear Martingale Sequence]
		\label{def:multilinear-doob-martingale}
		Let $X_0 \in \Omega = \{\pm 1\}^n$ be a starting state. Consider a walk of the Glauber dynamics starting at $X_0$ and running for $t^*$ steps: $X_0,X_1,\ldots,X_{t^*}$. $X_{t^*}$ can be viewed as a function of all the random choices made by the dynamics up to that point. That is, $X_{t^*} = h(X_0,R_1,\ldots,R_{t^*})$ where $R_i$ is a random variable representing the random choices made by the dynamics in step $i$. Hence $f_a(X_{t^*}) = \tilde{f}_{a}(X_0,R_1,\ldots,R_{t^*})$ where $\tilde{f}_a = f_a \circ h$.
		Consider the Doob martingale associated with $\tilde{f}_{a}$ defined on the probability space $(O,2^{O},P)$ where $O$ is the set of all possible values of the variables $X_0,X_1,X_2,\ldots,X_{t^*}$ under the Glauber dynamics and $P$ is the function which assigns probability to events in $2^O$ according to the underlying Glauber dynamics. Also consider the increasing sequence of sub-$\sigma$-fields $2^{O_0} \subset 2^{O_1} \subset 2^{O_2} \subset \ldots 2^{O_{t^*}} = 2^O$ where $O_i$ is the set of all possible values to the variables $X_0,X_1,X_2,\ldots,X_i$ under the Glauber dynamics. The terms in the martingale sequence are as follows:
		\begin{align}
		B_0 &= \E\left[\tilde{f}_{a}(X_0,R_1,\ldots,R_{t^*}) \middle| X_0\right]\notag \\
		&\cdots \notag \\
		B_i &= \E[\tilde{f}_{a}(X_0,R_1,\ldots,R_{t^*}) | X_0,R_1,\ldots,R_i] \label{eq:multilinear-doob}\\
		&\cdots \notag \\
		B_{t^*} &= \tilde{f}_{a}(X_0,R_1,\ldots,R_{t^*}) \notag
		\end{align}
		
		Since the dynamics are Markovian, we can also write $B_i$ as follows:
		\begin{align*}
		B_i = \E[f_a(X_{t^*}) | X_i] \quad \forall \: 0 \le i \le t^*.
		\end{align*}
	\end{definition}
	
	Next, we define a stopping time $T_K$ on the above martingale sequence. The definition generalizes the stopping time defined in Section \ref{sec:bilinear} by requiring that many conditional expectations are small together.
	\begin{definition}[Stopping Time for $d$-linear functions]
		\label{def:multilinear-stopping-time}
		Consider the martingale sequence defined in Definition~\ref{def:multilinear-doob-martingale}.
		Let $T_K : O \rightarrow \{0\} \bigcup \mathbb{N}$ be a stopping time defined as follows:
		\begin{align*}
		T_K = \min\{\min_{t \ge 0} \left\{ t \:\: \middle| t \notin G_K^{a,d+1}(t^*,t) \right\}, t^*+1\}.
		\end{align*}
		Note that the event $\{T_K = t\}$ lies in the $\sigma$-field $2^{O_t}$ and hence the above definition is a valid stopping time.
	\end{definition}

	Using the induction hypothesis, we will show that the stopping time defined above is large with a good probability for the parameter range which is of interest to us.
	\begin{lemma}
		\label{lem:multilinear-goodset-largeprob}
		For any $t \ge 0$, $t^* \le (n+1)\tmix$, there exists $c(d)>0$ such that, for any $K > c(d)(n\log^2 n/\eta)^{d/2}$,
		$$\Pr\left[ X_t \notin G_K^{a,d+1}(t^*,t) \right] \le 8dn^{d+1}\exp\left(-\frac{K^{2/d}}{2c_1(d)t^*}\right),$$
		where $c_1(d)$ is as defined in (\ref{eq:multilinear-Gt}).
	\end{lemma}
	\begin{proof}
		For any $1 \le k \le d+1$, and $v_1,v_2,\ldots, v_k \in V$, let $E_K(v_1,v_2,\ldots,v_k)$ be the following event:
		$$E_K(v_1,v_2,\ldots,v_k) = \max\left\{\abss{\E[f_a^{v_1,v_2,\ldots,v_k}(X_{t^*})|X_t=x_t]}, \abss{\E[f_a^{v_1,v_2,\ldots,v_k}(X_{t^*-1})|X_t=x_t]}\right\} > K^{(d+1-k)/d}.$$
		Since $X_0$ is a sample from the stationary distribution $p$ of the dynamics, it follows from the property of stationary distributions that $X_t$ is also a sample from $p$. Hence we have, from the induction hypothesis, Statement~\ref{stmt:conc-of-cond-expectation} for multilinear functions of degree $\le d$, and a union bound, that for any $1 \le k \le d+1$, $v_1,v_2,\ldots, v_k \in V$, and for $K > c(d)(n\log^2 n/\eta)^{d/2}$,
		\begin{align}
		\Pr\left[ \: E_K(v_1,v_2,\ldots,v_k) \right] \le 4\exp\left( -\frac{K^{2/d}}{c_4(d+1-k)t^*}  \right) \label{eq:multilinear-gt1}
		\end{align}
		
		From (\ref{eq:multilinear-gt1}) and a union bound, we get
		\begin{align}
		&\Pr\left[ \: \exists 1 \le k \le d+1 \text{ and } v_1,v_2,\ldots,v_k \in V \: \text{s.t. } E_K(v_1,v_2,\ldots,v_k) \right] \notag\\
		&~~~\le \sum_{k=1}^{d+1} 4n^k\exp\left(-\frac{K^{2/d}}{c_4(d+1-k)t^*}\right) \le 4dn^{d+1}\exp\left(-\frac{K^{2/d}}{c_4(d+1)t^*}\right) \label{eq:multilinear-gt2}.
		\end{align}
		For any $1 \le k \le d+1$, $v_1,v_2,\ldots, v_k \in V$, let $D_K(v_1,v_2,\ldots,v_k)$ denote the following event:
		\begin{align*}
		&D_K(v_1,v_2,\ldots,v_k) = \max\left\{\abss{f_a^{v_1,v_2,\ldots,v_k}(X_{t^*}) - \E\left[ f_a^{v_1,v_2,\ldots,v_k}(X_{t^*}) | X_t \right]}, \right. \\ 
		&~~~~~\left. \abss{f_a^{v_1,v_2,\ldots,v_k}(X_{t^*-1}) - \E\left[ f_a^{v_1,v_2,\ldots,v_k}(X_{t^*-1}) | X_t \right]} \right\} > K^{(d+1-k)/d}.
		\end{align*}
		Let $D_K^{a}(t,k)$ be the event defined as 
		$$D_K^a(t,k) =\exists \: v_1,v_2,\ldots,v_k \in V\ \textrm{such that}\  D_K(v_1,v_2,\ldots,v_k).$$
		From the inductive hypothesis, Statement~\ref{stmt:freedman-conc} for multilinear functions of degree $(d+1-k) (\le d)$, and a union bound, we have,
		\begin{align}
		&\Pr\left[ D_K^a(t,k) \right] = \E\left[ \Pr\left[ D_K^a(t,k) | X_t \right] \right] \le 4n^k\exp\left(-\frac{K^{2/d}}{c_2(d+1-k)t^*}\right) \notag \\
		\implies  &\Pr\left[ \Pr\left[ D_K^a(t,k) | X_t \right] > \exp\left(-\frac{K^{2/d}}{c_1(d+1-k)t^*}\right) \right] \le 4n^k\exp\left(-\frac{K^{2/d}}{2c_2(d+1-k)t^*}\right) \label{eq:multilinear-gt3}\\
		\implies &\Pr \left[ \bigcup_{k=1}^{d+1} \Pr\left[ D_K^a(t,k) | X_t \right] > \exp\left(-\frac{K^{2/d}}{c_1(d+1-k)t^*}\right) \right] \le \sum_{k=1}^{d+1}4n^l\exp\left(-\frac{K^{2/d}}{2c_2(d+1-k)t^*}\right) \notag \\
		&~~~~\le 4dn^{d+1}\exp\left(-\frac{K^{2/d}}{2c_2(d)t^*}\right) \label{eq:multilinear-gt4}
		\end{align}
		where (\ref{eq:multilinear-gt3}) follows from Markov's inequality (and holds for sufficiently large $c_2(d)$) and (\ref{eq:multilinear-gt4}) follows from a union bound.
		Hence,
		\begin{align}
		&\Pr\left[ \bigcup_{k=1}^{d+1} \exists \: v_1,v_2,\ldots,v_k \in V \: | \: \Pr[D_K(v_1,v_2,\ldots,v_k)] > \exp\left(-\frac{K^{2/d}}{c_1(d+1-k)t^*}\right) \right] \notag\\
		& \le \Pr\left[ \bigcup_{k=1}^{d+1} \Pr\left[ D_K^a(t,k) | X_t \right] > \exp\left(-\frac{K^{2/d}}{c_1(d+1-k)t^*}\right) \right] \le 4dn^{d+1}\exp\left(-\frac{K^{2/d}}{2c_2(d)t^*}\right) \label{eq:multilinear-gt5}.
		\end{align}
		From (\ref{eq:multilinear-gt2}) and (\ref{eq:multilinear-gt5}), we have,
		\begin{align*}
		& \Pr\left[ \neg G_K^{a,d+1}(t^*,t) \right] \le 4dn^{d+1}\exp\left(-\frac{K^{2/d}}{c_4(d)t^*}\right) + 4dn^{d+1}\exp\left(-\frac{K^{2/d}}{2c_2(d)t^*}\right)  \le 8dn^{d+1}\exp\left(-\frac{K^{2/d}}{2c_2(d)t^*}\right) \\
		\end{align*}
	\end{proof}

	\begin{lemma}
		\label{lem:multilinear-stopping-time-large}
		For $t^* \le (n+1)\tmix$, there exists $c(d)>0$ such that, for any $K > c(d)(n\log^2 n/\eta)^{d/2}$,
		$$\Pr\left[ t^* \ge T_K \right] \le 8dt^* n^{d+1}\exp\left(-\frac{K^{2/d}}{2c_2(d)t^*}\right),$$
		where $c_2(d)$ is as defined in Theorem~\ref{thm:multilinear-induction}.
	\end{lemma}
	\begin{proof}
		From Lemma~\ref{lem:multilinear-goodset-largeprob}, we have,
		\begin{align*}
		& \Pr\left[ X_t \notin G_K^{a,d+1}(t^*,t) \right] \le 8dn^{d+1}\exp\left(-\frac{K^{2/d}}{2c_2(d)t^*}\right) \\
		\implies &\Pr[t^* \ge T_K] = \Pr\left[ \bigcup_{t=0}^{t^*} X_t \notin G_K^{a,d+1}(t^*,t) \right] \\
		& \le \sum_{t=0}^{t^*}\Pr\left[X_t \notin G_K^{a,d+1}(t^*,t) \right] \le 8dt^* n^{d+1}\exp\left(-\frac{K^{2/d}}{2c_2(d)t^*}\right).
		\end{align*}
	\end{proof}
	
	Now we will argue that the increments of the martingale are bounded up until stopping time. 
	%For ease of exposition, we will refer to $G_K^{a,d+1}(t^*,t)$ by just $G_t$ in the following lemma. 
	\begin{lemma}
		\label{lem:multilinear-martingale-increment-bound}
		Consider the Doob martingale defined in Definition~\ref{def:bilinear-doob-martingale}. Suppose $X_i \in G_K^{a,d+1}(t^*,i)$ and $X_{i+1} \in G_K^{a,d+1}(t^*,i+1)$. For $K > c(d)(n\log^2 n/\eta)^{d/2}$, and a large enough constant $c_5$,
		\begin{align*}
		\abss{B_{i+1} - B_i} \le 2d^2 c_5 K.
		\end{align*}
	\end{lemma}
	\begin{proof}
	For ease of exposition, we will refer to $G_K^{a,d+1}(t^*,i)$ as simply $G_i$ in the following proof.
	\small
	\begin{align}
	&\abss{B_{i+1}-B_i} = \abss{\E\left[f_a(X_{t^*}) \middle| X_{i+1}\right] - \E\left[f_a(X_{t^*}) \middle| X_{i}\right]} \\ 
	&~~\le \max_{\substack{x,y: \dh(x,y)=1,\\ x\in G_{i+1}, y \in G_{i}}} \abss{\E\left[f_a(X_{t^*}) \middle| X_{i+1}=x\right] - \E\left[f_a(X_{t^*}') \middle| X_{i}'=y\right]} \label{eq:mmd1}\\
	&~~\le \max_{\substack{x,y: \dh(x,y)=1,\\ x\in G_{i+1}, y \in G_{i}}} \abss{\E\left[f_a(X_{t^*}) \middle| X_{i+1}=x\right] - \E\left[f_a(X_{t^*-1}') \middle| X_{i}'=y\right]} + \abss{\E\left[f_a(X_{t^*-1}') \middle| X_{i}'=y\right] - \E\left[f_a(X_{t^*}') \middle| X_{i}'=y\right]} \label{eq:mmd2}\\
	&~~\le \max_{\substack{x,y: \dh(x,y)=1,\\ x\in G_{i+1}, y \in G_{i}}} \abss{\E\left[f_a(X_{t^*})-f_a(X_{t^*-1}') \middle| X_{i+1}=x,X_{i}'=y\right]} \label{eq:mmd3}\\
	&~~~+ \max_{\substack{x,y: \dh(x,y)=1,\\ x\in G_{i+1}, y \in G_{i}}} \abss{\E\left[f_a(X_{t^*-1}') - f_a(X_{t^*}') \middle| X_{i}'=y\right]} \label{eq:mmd4}
	\end{align}
	\normalsize
	where in (\ref{eq:mmd1}) we relabeled the variables in the second expectation to avoid notational confusion in the later steps of our bounding, maintaining the understanding that the sequence $\{X_i',Y_i'\}_i$ has the same distribution as $\{X_i,Y_i\}_i$, in (\ref{eq:mmd2}) we added and subtracted the term $\E[f_a(X_{t^*-1}') | X_i'=y]$, (\ref{eq:mmd3}) holds for any valid coupling of the two chains, one starting at $X_{i+1}$ and the other starting at $X_i'$, and both running for $t^*-1-i$ steps. In particular, we use the greedy coupling between these two runs (Definition~\ref{def:greedy-coupling}). Consider (\ref{eq:mmd3}).
	\small
	\begin{align}
	&~~(\ref{eq:mmd3}) = \max_{\substack{x,y: \dh(x,y)=1,\\ x\in G_{i+1}, y \in G_{i}}} \left\lvert \E\left[\sum_{u_1,\ldots,u_{d}} \prod_{e=1}^{d}X_{t^*,u_e}\left(\sum_{u_{d+1}} a_{u_1 u_2 \ldots u_{d+1}} (X_{t^*,u_{d+1}}-X_{t^*-1,u_{d+1}}')\right) +\right. \right. \notag\\
	&~~~+ \left. \left. \sum_{u_1,\ldots,u_{d-1},u_{d+1}} \prod_{e=1}^{d-1} X_{t^*,u_e}X_{t^*-1,u_{d+1}}'\left(\sum_{u_{d}} a_{u_1 u_2 \ldots u_{d+1}}(X_{t^*,u_{d}}-X_{t^*-1,u_{d}}')\right) + \ldots \right. \right. \notag\\
	&~~~+ \left. \left. \sum_{u_2,\ldots,u_{d+1}} \prod_{e=2}^{d+1}X_{t^*-1,u_e}'\left(\sum_{u_1} a_{u_1u_2\ldots u_{d+1}}\left(X_{t^*,u_1} - X_{t^*-1,u_1}'\right) \right) \middle| X_{i+1} = x, X_{i}'=y\right] \right\rvert \notag\\
	&~~\le \max_{\substack{x,y: \dh(x,y)=1,\\ x\in G_{i+1}, y \in G_{i}}} \E\left[\sum_{u_{d+1}} \abss{X_{t^*,u_{d+1}}-X_{t^*-1,u_{d+1}}'}\abss{\sum_{u_1,\ldots,u_{d}} a_{u_1u_2\ldots u_{d+1}} \prod_{e=1}^{d}X_{t^*,u_e}} \right. \notag \\
	&~~~\left. + \sum_{u_{d}} \abss{X_{t^*,u_{d}}-X_{t^*-1,u_{d}}'}\abss{\sum_{u_1,\ldots,u_{d-1},u_{d+1}} a_{u_1 u_2 \ldots u_{d+1}} \prod_{e=1}^{d-1} X_{t^*,u_e}X_{t^*-1,u_{d+1}}'} + \ldots \right. \notag \\
	&~~~\left. + \sum_{u_1} \abss{X_{t^*,u_1} - X_{t^*-1,u_1}'}\abss{\sum_{u_2,u_3,\ldots,u_{d+1}} a_{u_1u_2\ldots u_{d+1}} \prod_{e=2}^{d+1}X_{t^*-1,u_e}' } \:\: \middle| X_{i+1} = x, X_{i}'=y\right] \label{eq:mmd6}
	\end{align}
	We see the \emph{hybrid} terms arise in (\ref{eq:mmd6}). A generic term in (\ref{eq:mmd6}) looks as follows:
	\small
	\begin{align}
	\E\left[\sum_{u_l} \abss{X_{t^*,u_l} - X_{t^*-1,u_l}'}\abss{\sum_{u_1,u_2,\ldots,u_{l-1},u_{l+1},\ldots, u_{d+1}} a_{u_1u_2\ldots u_{d+1}} \prod_{e=1}^{l-1} X_{t^*,u_e}\prod_{e=l+1}^{d+1} X_{t^*-1,u_{e}}'} \middle| X_{i+1}=x,X_{i}'=y  \right] \label{eq:mmd7}	
	\end{align}
	\normalsize
	
	Since $x \in G_{i+1}$ and $y \in G_i$, from Statements~\ref{stmt:hybrid-cond-expectation} and \ref{stmt:hybrid-cond-concentration} of Theorem \ref{thm:hybrid-supplement-induction}
	\small
	\begin{align}
	&\abss{\E\left[\sum_{u_1,u_2,\ldots,u_{l-1},u_{l+1},\ldots, u_{d+1}} a_{u_1u_2\ldots u_{d+1}} \prod_{e=1}^{l-1} X_{t^*,u_e}\prod_{e=l+1}^{d+1} X_{t^*-1,u_{e}}' \middle| X_{i+1}=x,X_{i}'=y  \right]} \le dK \text{ and}\label{eq:mmd8}\\
	&\Pr\left[ \left\lvert \sum_{u_1,u_2,\ldots,u_{l-1},u_{l+1},\ldots, u_{d+1}} a_{u_1u_2\ldots u_{d+1}} \prod_{e=1}^{l-1} X_{t^*,u_e}\prod_{e=l+1}^{d+1} X_{t^*-1,u_{e}}' - \right. \right. \notag \\
	&~\left. \left. \E\left[\sum_{u_1,u_2,\ldots,u_{l-1},u_{l+1},\ldots, u_{d+1}} a_{u_1u_2\ldots u_{d+1}} \prod_{e=1}^{l-1} X_{t^*,u_e}\prod_{e=l+1}^{d+1} X_{t^*-1,u_{e}}' \middle| X_{i+1},X_{i}'  \right]\right\rvert > K \middle| X_{i+1}=x,X_i' = y\right] \le 2\exp\left(-\frac{K^{2/d}}{c_{6}(d)t^*}\right), \label{eq:mmd9}
	\end{align}
	\normalsize
	where $c_6(d)$ is the constant function in Statement~\ref{stmt:hybrid-cond-concentration} of Theorem \ref{thm:hybrid-supplement-induction}.
	(\ref{eq:mmd8}) and (\ref{eq:mmd9}) together with the Hamming contraction property of the greedy coupling (Lemma~\ref{lem:greedy-properties}) imply, there exists a constant $c_5 > 0$, such that, for $K > c(d)(n\log^2 n/\eta)^{d/2}$,
	\begin{align}
	&(\ref{eq:mmd7}) \le c_5dK \implies (\ref{eq:mmd6}) \le c_5d^2K.
	\end{align}
	
	Now, we consider (\ref{eq:mmd4}) and bound it using the same approach as was used to bound (\ref{eq:mmd3}).
	\small
	\begin{align}
	&(\ref{eq:mmd4}) =  \max_{y \in G_{i}} \left\lvert \E\left[\sum_{u_1,\ldots,u_{d}} \prod_{e=1}^{d}X_{t^*,u_e}'\left(\sum_{u_{d+1}} a_{u_1 u_2 \ldots u_{d+1}} (X_{t^*,u_{d+1}}'-X_{t^*-1,u_{d+1}}')\right) +\right. \right. \label{eq:mmd10}\\
	&~~~+ \left. \left. \sum_{u_1,\ldots,u_{d-1},u_{d+1}} \prod_{e=1}^{d-1} X_{t^*,u_e}'X_{t^*-1,u_{d+1}}'\left(\sum_{u_{d}} a_{u_1 u_2 \ldots u_{d+1}}(X_{t^*,u_{d}}'-X_{t^*-1,u_{d}}')\right) + \ldots \right. \right. \notag\\
	&~~~+ \left. \left. \sum_{u_2,\ldots,u_{d+1}} \prod_{e=2}^{d+1}X_{t^*-1,u_e}'\left(\sum_{u_1} a_{u_1u_2\ldots u_{d+1}}\left(X_{t^*,u_1}' - X_{t^*-1,u_1}'\right) \right) \middle| X_{i}'=y\right] \right\rvert \notag\\
	&~~\le \max_{y \in G_{i}} \E\left[\sum_{u_{d+1}} \abss{X_{t^*,u_{d+1}}'-X_{t^*-1,u_{d+1}}'}\abss{\sum_{u_1,\ldots,u_{d}} a_{u_1u_2\ldots u_{d+1}} \prod_{e=1}^{d}X_{t^*,u_e}'} \right. \notag \\
	&~~~\left. + \sum_{u_{d}} \abss{X_{t^*,u_{d}}'-X_{t^*-1,u_{d}}'}\abss{\sum_{u_1,\ldots,u_{d-1},u_{d+1}} a_{u_1 u_2 \ldots u_{d+1}} \prod_{e=1}^{d-1} X_{t^*,u_e}'X_{t^*-1,u_{d+1}}'} \ldots \right. \notag \\
	&~~~\left. + \sum_{u_1} \abss{X_{t^*,u_1}' - X_{t^*-1,u_1}'}\abss{\sum_{u_2,u_3,\ldots,u_{d+1}} a_{u_1u_2\ldots u_{d+1}} \prod_{e=2}^{d+1}X_{t^*-1,u_e}' } \:\: \middle| X_{i}'=y\right] \label{eq:mmd11}
	\end{align}
	\normalsize
	where in (\ref{eq:mmd10}) we have used Statement 3 of Lemma~\ref{lem:greedy-properties}.
	A generic term in (\ref{eq:mmd11}) looks as follows:
	\small
	\begin{align}
	\E\left[\sum_{u_l} \abss{X_{t^*,u_l}' - X_{t^*-1,u_l}'}\abss{\sum_{u_1,u_2,\ldots,u_{l-1},u_{l+1},\ldots, u_{d+1}} a_{u_1u_2\ldots u_{d+1}} \prod_{e=1}^{l-1} X_{t^*,u_e}'\prod_{e=l+1}^{d+1} X_{t^*-1,u_{e}}'} \middle| X_{i}'=y  \right] \label{eq:mmd12}	
	\end{align}
	\normalsize
	Since $y \in G_i$, from the definition of $G_i$ and the fact that $\dh(X_{t^*}',X_{t^*-1}') \le 1$ we have that,
	\small
	\begin{align}
	&\abss{\E\left[\sum_{u_1,u_2,\ldots,u_{l-1},u_{l+1},\ldots, u_{d+1}} a_{u_1u_2\ldots u_{d+1}} \prod_{e=1}^{l-1} X_{t^*,u_e}'\prod_{e=l+1}^{d+1} X_{t^*-1,u_{e}}' \middle| X_{i}'=y  \right]} \le 2K \text{ and}\label{eq:mmd13}\\
	&\Pr\left[ \left\lvert \sum_{u_1,u_2,\ldots,u_{l-1},u_{l+1},\ldots, u_{d+1}} a_{u_1u_2\ldots u_{d+1}} \prod_{e=1}^{l-1} X_{t^*,u_e}'\prod_{e=l+1}^{d+1} X_{t^*-1,u_{e}}' - \right. \right. \notag \\
	&~\left. \left. \E\left[\sum_{u_1,u_2,\ldots,u_{l-1},u_{l+1},\ldots, u_{d+1}} a_{u_1u_2\ldots u_{d+1}} \prod_{e=1}^{l-1} X_{t^*,u_e}'\prod_{e=l+1}^{d+1} X_{t^*-1,u_{e}}' \middle| X_{i}'  \right]\right\rvert > 2K \middle| X_{i}' = y\right] \le 2\exp\left(-\frac{(2K)^{2/d}}{c_{1}(d)t^*}\right). \label{eq:mmd14}
	\end{align}
	\normalsize
	(\ref{eq:mmd13}) and (\ref{eq:mmd14}) together with the property that $\dh(X_{t^*}',X_{t^*-1}') \le 1$ imply that for $c_5$ sufficiently large and $K > c(d)(n\log^2 n/\eta)^{d/2}$,
	\begin{align}
	&(\ref{eq:mmd12}) \le c_5dK \implies (\ref{eq:mmd11}) \le c_5d^2K.
	\end{align}
	Hence we get that, when $X_i \in G_i$ and $X_{i+1} \in G_{i+1}$,
	\begin{align}
	\abss{B_{i+1}-B_i} \le 2c_5d^2K.
	\end{align}
	
	\end{proof}

	As a consequence of Lemma \ref{lem:multilinear-martingale-increment-bound}, we get the following two useful corollaries.
	\begin{corollary}
		\label{cor:multilinear-martingale-increment-prob-bound}
		Consider the martingale sequence defined in Definition~\ref{def:multilinear-doob-martingale}. 
		$$\Pr\left[\forall \: 0 < i+1 < T_K, \: \abss{B_{i+1}-B_i} \le 2c_5d^2K \right] = 1.$$
	\end{corollary}
	\begin{proof}
		Let $\kappa = 2c_5d^2K$.
		\begin{align}
		&~\Pr\left[\forall \: 0 < i+1 < T_K, \: \abss{B_{i+1}-B_i} \le \kappa \right] \notag\\
		&=1 - \Pr\left[\exists \: 0 < i+1 < T_K, \: \abss{B_{i+1}-B_i} > \kappa \right] \notag\\
		&=1 - \Pr\left[\exists \: 0 < i+1 < T_K, \: \left(X_i \in G_K^{a,d+1}(t^*,i), X_{i+1}\in G_K^{a,d+1}(t^*,i+1) \text{ and } \abss{B_{i+1}-B_i} > \kappa\right) \right. \notag\\
		&~~~\left. \text{ or }  \left(\left(X_i \notin G_K^{a,d+1}(t^*,i) \text{ or } X_{i+1}\notin G_K^{a,d+1}(t^*,i+1)\right) \text{ and } \abss{B_{i+1}-B_i} > \kappa\right)\right] \notag\\
		&= 1- \Pr\left[\exists \: 0 < i+1 < T_K, \: \left(X_i \in G_K^a(i), X_{i+1}\in G_K^a(i+1) \text{ and } \abss{B_{i+1}-B_i} > \kappa\right) \right] \label{eq:mpr-one3}\\
		& = 1-0 \label{eq:mpr-one4}
		\end{align}
		where (\ref{eq:mpr-one3}) follows because by the definition of $T_K$, $\Pr\left[\exists \: 0 < i+1 < T_K, \: \left(X_i \notin G_K^a(i) \text{ or } X_{i+1}\notin G_K^a(i+1)\right)\right] = 0$, and (\ref{eq:mpr-one4}) follows because $X_i \in G_K^a(i), X_{i+1}\in G_K^a(i+1) \implies \abss{B_{i+1}-B_i} \le \kappa$ (Lemma \ref{lem:multilinear-martingale-increment-bound}).
	\end{proof}
	Corollary~\ref{cor:multilinear-martingale-increment-prob-bound} will give us one of the required conditions to apply Freedman's inequality.

	As a corollary of Lemma \ref{lem:multilinear-martingale-increment-bound}, we get a bound on the variance of the martingale differences which holds with high probability and to show it we first show Claim~\ref{clm:multilinear-nice-sets} which states that, informally, for any time step $i$, with a large probability we hit an $X_i$ such that the probability of transitioning from $X_i$ to an $X_{i+1} \in G_K^a(i+1)$ is large.
	\begin{claim}
		\label{clm:multilinear-nice-sets}
		Denote by $N_K^{a,d+1}(t^*,i)$ the following set of configurations:
		\begin{align}
		N_K^{a,d+1}(t^*,i) = \left\{ x_i \in \Omega \middle| \Pr\left[X_{i+1} \notin G_K^{a,d+1}(t^*,i+1) \middle| X_i = x_i \right] \le \exp\left( -\frac{K^{2/d}}{4c_2(d)t^*} \right) \right\}.
		\end{align}
		Then,
		$$\Pr\left[X_i \notin N_K^{a,d+1}(t^*,i)\right] \le 8dn^{d+1}\exp\left(-\frac{K^{2/d}}{4c_2(d)t^*}\right).$$
	\end{claim}
	\begin{proof}
		We have from Lemma \ref{lem:multilinear-goodset-largeprob}, that 
		\begin{align}
		&\Pr\left[X_{i} \notin G_K^{a,d+1}(t^*,i) \right] \le 8dn^{d+1}\exp\left( -\frac{K^{2/d}}{2c_2(d)t^*} \right) \text{ and} \label{eq:mnice1}\\
		&\Pr\left[X_{i+1} \notin G_K^{a,d+1}(t^*,i+1) \right] \le 8dn^{d+1}\exp\left( -\frac{K^{2/d}}{2c_2(d)t^*} \right). \label{eq:mnice2}
		\end{align}
		From the definition of the set $N_K^{a,d+1}(t^*,i)$ we have,
		\begin{align}
		\Pr\left[X_{i+1} \in G_K^{a,d+1}(t^*,i) | X_i \notin N_K^{a,d+1}(t^*,i)\right] \le 1-\exp\left( -\frac{K^{2/d}}{4c_2(d)t^*} \right).
		\end{align}
		Then we have,
		\begin{align}
		&1-8dn^{d+1}\exp\left( -\frac{K^{2/d}}{2c_2(d)t^*} \right) \le \Pr\left[X_{i+1} \in G_K^{a,d+1}(t^*,i)\right]\\
		&= \Pr\left[X_{i+1} \in G_K^{a,d+1}(t^*,i) | X_i \in N_K^{a,d+1}(t^*,i)\right] \Pr\left[X_i \in N_K^{a,d+1}(t^*,i)\right] \notag \\
		&~~+ \Pr\left[X_{i+1} \in G_K^{a,d+1}(t^*,i) | X_i \notin N_K^{a,d+1}(t^*,i)\right]\Pr\left[X_i \notin N_K^{a,d+1}(t^*,i)\right] \\
		&\le \Pr\left[X_i \in N_K^{a,d+1}(t^*,i)\right] + \left(1-\exp\left( -\frac{K^{2/d}}{4c_2(d)t^*} \right)\right)\Pr\left[X_i \notin N_K^{a,d+1}(t^*,i)\right] \\
		&=   \left(1-\exp\left( -\frac{K^{2/d}}{4c_2(d)t^*} \right)\right) + \exp\left( -\frac{K^{2/d}}{4c_2(d)t^*} \right)\Pr\left[X_i \in N_K^{a,d+1}(t^*,i)\right]. \label{eq:mnice3}
		\end{align}
		(\ref{eq:mnice3}) implies,
		\begin{align}
		&\Pr\left[X_i \in N_K^{a,d+1}(t^*,i)\right] \ge  \frac{\exp\left( -\frac{K^{2/d}}{4c_2(d)t^*} \right) - 8dn^{d+1}\exp\left( -\frac{K^{2/d}}{2c_2(d)t^*} \right)}{\exp\left( -\frac{K^{2/d}}{4c_2(d)t^*} \right)}\\
		&=1-8dn^{d+1}\exp\left(-\frac{K^{2/d}}{4c_2(d)t^*}\right).
		\end{align}
	\end{proof}

	As a corollary of Lemma \ref{lem:multilinear-martingale-increment-bound}, we get a bound on the variance of the martingale differences which holds with high probability.
	\begin{lemma}
		\label{lem:multilinear-martingale-increment-variance}
		Consider the martingale sequence defined in Definition~\ref{def:multilinear-doob-martingale}. Denote by $N_K^{a,d+1}(t^*,i)$ the following set of configurations:
		\begin{align}
		N_K^{a,d+1}(t^*,i) = \left\{ x_i \in \Omega \middle| \Pr\left[X_{i+1} \notin G_K^{a,d+1}(t^*,i+1) \middle| X_i = x_i \right] \le \exp\left( -\frac{K^{2/d}}{4c_2(d)t^*} \right) \right\}.
		\end{align}
		Let $b = (c_5d^2K/2)^2 + n^{2d +2}\exp\left( -\frac{K^{2/d}}{4c_2(d)t^*} \right)$ where $c_5$ is the constant from Lemma~\ref{lem:multilinear-martingale-increment-bound}. Then,
		\begin{align*}
		\Pr\left[\Var[B_{i+1} - B_i | \mathcal{F}_i] > b \middle| X_i \in G_K^{a,d+1}(t^*,i) \cap N_K^{a,d+1}(t^*,i) \right] = 0.
		\end{align*}
		where $\mathcal{F}_i = 2^{O_i}$.
	\end{lemma}
	\begin{proof}
%		We have, by definition, $E[X | \mathcal{F}]$ is any $\mathcal{F}$-measurable function which satisfies
%		\begin{align}
%		\int_F \left[E[X | \mathcal{F}]\right] dP = \int_F X dP
%		\end{align}
%		for any $F \in \mathcal{F}$. Since any event $F_i \in \mathcal{F}_i = 2^{O_i}$ is a set of assignments of values to the variables $X_0,X_1,\ldots,X_i$, we have,
		Since, the random variables $X_0,\ldots,X_i$ together characterize every event in $\mathcal{F}_i$, we have,
		\begin{align}
		\Var[B_{i+1}-B_i | \mathcal{F}_i] = \Var[B_{i+1}-B_i | X_0,X_1,\ldots,X_i] = \Var[B_{i+1}-B_i | X_i] \label{eq:mult-desigmafy} 
		\end{align}
		where the last equality follows from the Markov property of the Glauber dynamics.
		By the definition of $N_K^{a,d+1}(t^*,i)$, we have that
		\begin{align}
		\Pr\left[X_{i+1} \notin G_K^a(i+1) \middle| X_i \in N_K^a(i)\right] \le \exp\left( -\frac{K^{2/d}}{4c_2(d)t^*} \right).
		\end{align}
		This implies that,
		\begin{align}
		&\Pr\left[X_{i+1} \in G_K^a(i+1) \text{ and } X_i \in G_K^a(i) \middle| X_i \in G_K^a(i) \text{ and } X_i \in N_K^a(i)\right] \ge 1 - \exp\left( -\frac{K^{2/d}}{4c_2(d)t^*} \right)\\
		\implies &\Pr\left[\abss{B_{i+1}-B_i} < c_5d^2K \middle| X_i \in G_K^{a,d+1}(t^*,i) \cap N_K^{a,d+1}(t^*,i)\right] \ge 1 - \exp\left( -\frac{K^{2/d}}{4c_2(d)t^*} \right) \label{eq:mvar6}\\
		\implies &\Var\left[B_{i+1}-B_i \middle| X_i \in G_K^{a,d+1}(t^*,i) \cap N_K^{a,d+1}(t^*,i)\right] \le (c_5d^2K/2)^2 + n^{2d+2}\exp\left( -\frac{K^{2/d}}{4c_2(d)t^*} \right) \label{eq:mvar7}
		\end{align}
		where (\ref{eq:mvar6}) follows from Lemma~\ref{lem:multilinear-martingale-increment-bound}, and (\ref{eq:mvar7}) follows from the law of total variance and from the fact that $\Var(X) \le (b-a)^2/4$ when $X \in [a,b]$ with probability 1. The last inequality implies the statement of the lemma.
		
	\end{proof}

	With Lemma~\ref{lem:multilinear-martingale-increment-bound} and Lemma~\ref{lem:multilinear-martingale-increment-variance} to bound the martingale increments, and Lemma~\ref{lem:multilinear-stopping-time-large} to show that the stopping time is large, we are ready to apply Freedman's inequality on the martingale defined in Definition~\ref{def:multilinear-doob-martingale} to yield Lemma~\ref{lem:multilinear-apply-freedman}.
	\begin{lemma}
		\label{lem:multilinear-apply-freedman}
		For any $0 \le t_0 \le t^*$, there exists $c(d)>0$ which is a function of $d$ alone, such that for any $\rad > c(d+1)(n\log^2 n/\eta)^{(d+1)/2}$, 
		$$\Pr\left[ \abss{f_a(X_{t^*}) - \E\left[f_a(X_{t^*})|X_{t_0}\right]} \ge \rad  \right] \le 4\exp\left(-\frac{\rad^{2/(d+1)}}{c_2(d+1)t^*} \right).$$
	\end{lemma}
	\begin{proof}
		From Freedman's inequality (Lemma \ref{lem:freedman}) applied on the martingale sequence (Definition~\ref{def:multilinear-doob-martingale}) starting from $t_0$, we get
		\begin{align}
		\Pr\left[\exists t < T_K \text{ s.t. } \abss{B_t-B_{t_0}}\ge \rad \text{ and } V_{t} \le B \right] \le 2\exp\left(-\frac{\rad^2}{2(\rad K_1 + B)} \right) \label{eq:mfr1}
		\end{align}
		where $K_1 \le c_5d^2K$ (Lemma~\ref{lem:multilinear-martingale-increment-bound}) and $V_t$ is defined as follows:
		\begin{align}
		V_t = \sum_{i=0}^{t-1} \Var\left[ B_{i+1} - B_{i} | \mathcal{F}_i \right].\label{eq:multilinear-vt-def}
		\end{align}
		Set $B = t^*(c_5d^2K/2)^2 + t^*n^{2d+2}\exp\left( -\frac{K^{2/d}}{4c_2(d)t^*} \right)$ and $K = \rad^{d/(d+1)}$. Then we have, $\rad K_1 = c_5 d^2 K^{(2d+1)/d} \le c_5 d^2 K^2 n$ where the last inequality holds because $\rad \le n^{d+1}$ which in turn implies $K \le n^d$. 
		Similarly, since $\rad \ge c(d+1)(n\log^2 n/\eta)^{(d+1)/2}$, we have $K \ge c(d+1)^{d/(d+1)}(n\log^2 n/\eta)^{d/2}$.
		Combined with the fact that $t^* \le (d+1)n\log n/\eta$, this implies that $t^*n^{2d+2}\exp\left( -\frac{K^{2/d}}{4c_2(d)t^*} \right) \leq t^*$ for a sufficiently large value of $c(d+1)$.
		This in turn implies that $B \le t^* c_5^2d^4K^2 /2$.
		
		%On the other hand, since $t^* \ge 2n\log n/\eta$, we have $B \ge c_5^2 d^4 K^2 n\log n/2\eta$. For a sufficiently large value of the constant $c_5$, this implies $\rad K_1 \le B$. Moreover, since $\rad \ge c(d+1)(n\log^2 n/\eta)^{(d+1)/2}$, we have $K \ge c(d+1)^{d/(d+1)}(n\log^2 n/\eta)^{d/2}$, which for a sufficiently large value of $c(d+1)$ implies that  $n^4\exp\left( -\frac{K^{2/d}}{4c_2(d)t^*} \right) \le c_5^2d^4K^2/4$ (we have also used the fact that $t^* \le (d+1)n\log n/\eta$). This in turn implies that $B \le t^* c_5^2d^4K^2 /2$.
		Hence (\ref{eq:mfr1}) becomes,
		\begin{align}
		&\Pr\left[\exists t < T_K \text{ s.t. } \abss{B_t - B_0}\ge \rad \text{ and } V_{t} \le B \right] \le 2\exp\left(-\frac{\rad^2}{2( c_5 d^2 K^2 n + t^* c_5^2d^4K^2 /2)} \right) \notag\\
		&\le 2\exp\left(-\frac{\rad^2}{3c_5^2d^4K^2t^*} \right). \label{eq:mfr2}
		\end{align}

		Next we will bound, $\Pr\left[V_{t^*} > B\right]$ which will be useful for obtaining the desired concentration bound from (\ref{eq:fr2}). 
		\begin{align}
		&\Pr\left[V_{t^*} > B\right] \le \Pr\left[V_{t^*} > B \middle| \forall \: 0 \le t \le t^* \: X_t \in G_K^{a,d+1}(t^*,t) \cap N_K^{a,d+1}(t^*,t) \right] \notag \\
		&~~+ \Pr\left[\exists \: 0 \le t \le t^* \: X_t \notin G_K^{a,d+1}(t^*,t) \cup N_K^{a,d+1}(t^*,t) \right] \notag\\
		&\le \Pr\left[\exists \: 0 \le t < t^* \: \text{ s.t. } \Var\left[ B_{t+1}-B_t | X_t \right] > B/t^* \middle| \forall \: 0 \le t \le t^* \: X_t \in G_K^{a,d+1}(t^*,t) \cap N_K^{a,d+1}(t^*,t) \right] \label{eq:mvar-bound2}\\
		&~~+ \sum_{t=0}^{t^*}\left(\Pr\left[X_t \notin G_K^{a,d+1}(t^*,t) \right] + \Pr\left[X_t \notin N_K^{a,d+1}(t^*,t) \right]\right) \label{eq:mvar-bound3}\\
		&\le 0 + (t^*+1)\left(8dn^{d+1}\exp\left(-\frac{K^{2/d}}{2c_1(d)t^*}\right) + 8dn^{d+1}\exp\left(-\frac{K^{2/d}}{4c_2(d)t^*}\right)\right) \le \frac{16(d+2)^2n^{d+2}\log n}{\eta}\exp\left(-\frac{K^{2/d}}{4c_2(d)t^*}\right). \label{eq:mvar-bound4}
		\end{align}
		where (\ref{eq:mvar-bound2}) holds because $V_{t^*} > B$ implies that there exists a $0 \le t \le t^*$ such that $\Var\left[ B_{t+1}-B_t | X_t \right] > B/t^*$, (\ref{eq:mvar-bound3}) follows by an application of the union bound, and (\ref{eq:mvar-bound4}) follows from Lemma~\ref{lem:multilinear-martingale-increment-variance}, Lemma~\ref{lem:multilinear-goodset-largeprob} and Claim~\ref{clm:multilinear-nice-sets}.

		Now,
		\begin{align}
		&\Pr\left[\abss{f(X_{t^*}) - \E\left[f(X_{t^*})|X_{t_0}\right]}\ge \rad  \right] = \Pr\left[\abss{B_{t^*}-B_0} \ge \rad \right] \notag\\
		&~~\le \Pr\left[\abss{B_{t^*}-B_0}\ge \rad \text{ and } V_{t^*} \le B \right] + \Pr\left[V_{t^*} > B\right] \label{eq:mfr7}\\
		&~~\le \Pr\left[\abss{B_{t^*}-B_0}\ge \rad \text{ and } V_{t^*} \le B \text{ and } t^* < T_K \right] + \Pr[t^* \ge T_K] + \Pr\left[V_{t^*} > B\right] \label{eq:mfr3}\\
		&~~\le \Pr\left[\left(\exists t \le t^* \text{ s.t. } \abss{B_t - B_0}\ge \rad \text{ and } V_{t} \le B\right) \text{ and } t^* < T_K \right] + \Pr[t^* \ge T_K] + \Pr\left[V_{t^*} > B\right] \label{eq:mfr4}\\
		&~~\le \Pr\left[\exists t < T_K \text{ s.t. } \abss{B_t - B_0}\ge \rad \text{ and } V_{t} \le B \right] + \Pr[t^* \ge T_K] + \Pr\left[V_{t^*} > B\right] \notag\\
		&~~\le 2\exp\left(-\frac{\rad^2}{3c_5^2d^4K^2t^*} \right) + 8dt^* n^{d+1}\exp\left(-\frac{\rad^{2/(d+1)}}{2c_2(d)t^*}\right) + \Pr\left[V_{t^*} > B\right]  \label{eq:mfr5}\\
		&~~\le 3\exp\left(-\frac{\rad^2}{3c_5^2d^4K^2t^*} \right) + \Pr\left[V_{t^*} > B\right] \label{eq:mfr6}\\
		&~~\le 3\exp\left(-\frac{\rad^2}{3c_5^2d^4K^2t^*} \right) + \frac{16(d+2)^2n^{d+2}\log n}{\eta}\exp\left(-\frac{K^{2/d}}{4c_2(d)t^*}\right) \label{eq:mfr8}\\
		&~~\le 2\exp\left(-\frac{\rad^{2/(d+1)}}{c_2(d+1)t^*} \right) \label{eq:mfr9}
		\end{align}
		where (\ref{eq:mfr7}) and (\ref{eq:mfr3}) follow from the fact that $\Pr[A] \le \Pr[A \cap B] + \Pr[\neg B]$, (\ref{eq:mfr4}) follows from the fact that $\Pr[A] \le \Pr[A \cup B]$, (\ref{eq:mfr5}) follows from (\ref{eq:mfr2}) and from Lemma \ref{lem:multilinear-stopping-time-large}, (\ref{eq:mfr6}) holds for a sufficiently large $c(d+1)$ because $\rad > c(d+1)(n\log^2 n/\eta)^{(d+1)/2}$, (\ref{eq:mfr8}) follows from (\ref{eq:mvar-bound4}) and (\ref{eq:mfr9}) again holds for a sufficiently large $c(d+1)$, $c_2(d+1)$, because $\rad > c(d+1)(n\log^2 n/\eta)^{(d+1)/2}$ and $K = \rad^{d/(d+1)}$.
		Note that we have implicitly assumed that $\eta > 1/n$, since otherwise the concentration bounds obtained are trivial.
	\end{proof}

	%-------------------------
	% Stmt main-conc follows from stmt freedman-conc and mixing time properties
	%-------------------------
	\noindent \textbf{Statement~\ref{stmt:main-conc}:} This statement will follow from Statement~\ref{stmt:freedman-conc} applied to the case $t_0=0$ together with an application of the mixing time properties of the Glauber dynamics. Set $t^* = (d+2)\tmix$. From Lemma \ref{lem:close-to-stationary}, we have that 
	\begin{align}
	&\dtv(X_{t^*}|X_0, p) \le \exp\left(-(d+1)n\log n \right) \notag\\
	\implies &\abss{\E\left[f_a(X_{t^*})|X_0\right] - \E\left[f_a(X_{t^*})\right]} \le 2n^{d+1}\exp(- (d+1)n \log n) \le 2\exp(-n). \label{eq:mtb1}
	\end{align}
	Hence,
	\begin{align}
	&\Pr\left[ \abss{f_a(X_{t^*}) - \E\left[f_a(X_{t^*})\right]} \ge \rad \right] \le \Pr\left[ \abss{f_a(X_{t^*}) - 	\E\left[f_a(X_{t^*}) | X_0 \right]} > \rad-2\exp(-n) \right] \label{eq:mtb3}\\
	&\le 4\exp\left(-\frac{(\rad-2)^{2/(d+1)}}{c_2(d+1)t^*} \right) \label{eq:mtb4}\\
	&\le 2\exp\left(-\frac{\eta\rad^{2/(d+1)}}{c_3(d+1)n\log n}\right) \label{eq:mtb5}\notag
	\end{align}
	(\ref{eq:mtb3}) follows from (\ref{eq:mtb1}), (\ref{eq:mtb4}) follows from Lemma~\ref{lem:multilinear-apply-freedman} and (\ref{eq:mtb5}) holds for a sufficiently large constant $c_3(d)$. \\

	%-------------------------
	% Stmt conc-of-cond-expectation follows from stmt freedman-conc and stmt main-conc
	%-------------------------
	\noindent \textbf{Statement~\ref{stmt:conc-of-cond-expectation}:} This follows from Corollary~\ref{cor:marginal-bound}, and Statements~\ref{stmt:main-conc}, \ref{stmt:freedman-conc} of the theorem. Indeed, since $X_{t^*} \sim p$, we have for any $\rad > c(d+1)(n\log^2 n/\eta)^{(d+1)/2}$,
	\begin{align}
	&\Pr\left[\abss{\E[f_a(X_{t^*}) | X_{t_0}] - \E[f_a(X_{t^*})]} \ge \rad  \right] \\
	&\le \Pr\left[\abss{f_a(X_{t^*}) - \E[f_a(X_{t^*})]} \ge \rad/2  \right] + \left[ \abss{f_a(X_{t^*}) - \E[f_a(X_{t^*}) | X_{t_0}] } \ge \rad/2 \right] \label{eq:mcce1}\\
	&\le 2\exp\left(-\frac{\eta\rad^{2/(d+1)}}{4^{1/(d+1)}c_3(d)n\log n}\right) + 2\exp\left(-\frac{\rad^{2/(d+1)}}{4^{1/(d+1)}c_2(d+1)t^*} \right) \label{eq:mcce2}\\
	&\le 2\exp\left( - \frac{\rad^{2/(d+1)}}{c_7(d+1)t^*}\right). \label{eq:mcce3}
	\end{align}
	where (\ref{eq:mcce1}) follows because $\abss{\E[f_a(X_{t^*}) | X_{t_0}] - \E[f_a(X_{t^*})]} \ge \rad \implies \abss{f_a(X_{t^*}) - \E[f_a(X_{t^*}) | X_{t_0}] } \ge \rad/2$ or $\abss{f_a(X_{t^*}) - \E[f_a(X_{t^*})]} \ge \rad/2$, (\ref{eq:mcce2}) follows from Statements~\ref{stmt:main-conc} and \ref{stmt:freedman-conc} respectively and (\ref{eq:mcce3}) holds for a sufficiently large constant $c_7(d+1)$.
	Since $X_{t^*} \sim p$, from Corollary~\ref{cor:marginal-bound}, and from the fact that $\rad > c(d+1)(n\log^2 n/\eta)^{(d+1)/2}$ we get that $\abss{\E[f_a(X_{t^*})]} \le 2(4n(d+1)\log n / \eta)^{(d+1)/2} \le \rad/2$ for $c(d+1)$ sufficiently large. This implies in turn that,
	\begin{align}
	&\Pr\left[\abss{\E[f_a(X_{t^*}) | X_{t_0}]} \ge \rad  \right] \le \Pr\left[\abss{\E[f_a(X_{t^*}) | X_{t_0}]- \E[f_a(X_{t^*})]} \ge \rad/2  \right]\\
	&~~~\le 2\exp\left( - \frac{\rad^{2/(d+1)}}{4^{1/(d+1)}c_7(d+1)t^*}\right) \le 2\exp\left( - \frac{\rad^{2/(d+1)}}{c_4(d+1)t^*}\right),
	\end{align}
	for a sufficiently large constant $c_4(d+1)$.
	This shows the theorem holds by induction.
\end{proof}
Note that a straightforward corollary of Theorem~\ref{thm:multilinear-induction} is the desired statement for concentration of $d$-linear functions.

\subsection{Supplementary Theorem Statement for \emph{Hybrid} Functions}
\label{sec:hybrid}
\begin{theorem}
	\label{thm:hybrid-supplement-induction}
	Let $p$ be an Ising model in the $\eta$-high temperature regime.
	Let $\tmix = n\log n / \eta$ denote the mixing time of the Glauber dynamics associated with $p$.
	Let $f_a(x) = \sum_{u_1,u_2,\ldots ,u_d} a_{u_1 u_2 \ldots u_d} x_{u_1}x_{u_2}\ldots x_{u_d}$ be a $d$-linear function. Let $G_K^{a,d}(t_1,t)$ be the `good' set associated with $f_a(.)$ as defined in (\ref{eq:multilinear-Gt}). Additionally, define $G_K^{a,0}(t_1,t) = \{\pm 1\}^n$. Also let $2\tmix \le t^* \le (n+1)\tmix$, $0 \le t_0 \le t^*$ and let $x_{t_0}^{(1)}$ be a starting state such that $x_{t_0}^{(1)} \in G_K^{a,d}(t^*,t_0)$. Let $x_{t_0}^{(2)}$ be a state obtained by taking a step of the Glauber dynamics starting from $x_{t_0}^{(1)}$. Suppose we also have that $x_{t_0}^{(2)} \in G_{K}^{a,d}(t^*,t_0)$.
	Consider the 2 runs of the Glauber dynamics associated with $p$ with the $j^{th}$ run starting at $X_{t_0}^{(j)} = x_{t_0}^{(j)}$ respectively, and coupled together using the greedy coupling (Definition \ref{def:greedy-coupling}). Denote the state of run $j$ at time $t \ge t_0$ by $X_t^{(j)}$. Consider any $l$-linear function from $F_a^d(l)$: $f_a^{v_1,v_2,\ldots,v_{d-l}}$. Denote its coefficient vector by $\a$. That is,
	$$f_a^{v_1,v_2,\ldots,v_{d-l}}(x) = \sum_{u_1,u_2,\ldots ,u_l} \a_{u_1 u_2 \ldots u_l} x_{u_1}x_{u_2}\ldots x_{u_l}.$$
	Note that $\a_{u_1 u_2 \ldots u_l} = a_{u_1,u_2,\ldots,u_l,v_1,v_2,\ldots,v_{d-l}}$. 
	For each $f_a^{v_1,v_2,\ldots,v_{d-l}}$ we define an associated class of hybrid functions defined over the concatenated states from the two runs of Glauber dynamics described above as follows:
	\begin{align}
	&f_a^{v_1,v_2,\ldots,v_{d-l}}\left( x^{(1:2)} \right) = \sum_{u_1,u_2,\ldots,u_l} \a_{u_1 u_2 \ldots u_l} x_{u_1}^{(1)} x_{u_2}^{(1)} \ldots x_{u_{l_1}}^{(1)}x_{u_{l_1+1}}^{(2)} x_{u_l}^{(2)} \\
	&=\sum_{u_1,u_2,\ldots,u_l} \a_{u_1 u_2 \ldots u_l} \prod_{b=1}^{2} \prod_{e=1}^{l_b}x_{u_{(b-1)l_1+e}}^{(b)}
	\end{align}
	where $l_2 = l-l_1$.
	Then, the following two statements hold for all $0 \le l \le d-1$, and for any $f \in F_a^d(l)$, there exist constants $c(d),c_6(l)$ such that:
	\begin{enumerate}
	\item \small For any $K > c(d)(n\log^2 n/\eta)^{(d-1)/2}$
	$$\abss{\E\left[ f\left(X_{t^*}^{(1:2)}\right) \middle| X_{t_0}^{(j)}=x_{t_0}^{(j)} \: \text{for } j=1,2 \right]} \le (l+1)K^{l/(d-1)}.$$ \normalsize \label{stmt:hybrid-cond-expectation}
		
	\item For any $K > c(d)(n\log^2 n/\eta)^{(d-1)/2}$,
	\small
	$$\Pr\left[\abss{f\left(X_{t^*}^{(1:2)}\right) - \E\left[ f\left(X_{t^*}^{(1:2)}\right) \middle| X_{t_0}^{(j)} \: \text{for } j=1,2 \right]} > K^{l/(d-1)} \middle| X_{t_0}^{(j)}=x_{t_0}^{(j)} \: \text{for } j=1,2 \right] \le 2\exp\left(-\frac{K^{2/(d-1)}}{c_6(l)t^*}\right).$$
	\normalsize
	 \label{stmt:hybrid-cond-concentration}
	
	\end{enumerate}
\end{theorem}
\begin{proof}
	The proof will proceed by induction on $l$.\\
	
	%---------------------------
	% Base Cases l=0, l=1
	%---------------------------
	\noindent\textbf{Base Cases: $l=0,1$:} When $l=0$, the functions under consideration are all just constant functions and hence both the statements hold immediately. Consider the next case $l=1$ as well. In this case the functions are linear and hence no hybrid terms can arise. The statements of the theorem follow immediately from the definition of $G_K^{a,d}(t^*,t_0)$.

	We will assume the statements of the theorem hold for some $1 < l < d-1$. And proceed to show them for $l+1$.
	
	%---------------------------
	% Stmt hybrid-cond-expectation
	%---------------------------
	\noindent \textbf{Induction Step for Statement~\ref{stmt:hybrid-cond-expectation}:}
	We will begin with Statement~\ref{stmt:hybrid-cond-expectation}. We wish to show it for $l+1$-linear hybrid functions. At a high level, we will try to express any hybrid function of degree $l+1$ as a non-hybrid function of degree $l+1$ plus functions which resemble hybrid functions of degree $l$ multiplied with the Hamming distance between the two runs at time $t^*$. The definition of the `good' set will allow us to bound the conditional expectation of the non-hybrid function of degree $l+1$. The inductive hypothesis together with Hamming contraction properties will help us bound the other functions. The total number of such functions we will encounter is $\poly(d)$ and hence we incur a constant factor ($\poly(d)$) loss in the final bound.
	Consider any function $f_a^{v_1,v_2,\ldots,v_{d-l-1}}\left( x^{(1:2)}\right)$ from the family $F_a^d(l+1)$ with coefficient vector $\a$. We will show the statement by inducting on $l_2 = l+1-l_1$.
	For a given degree $l+1$ and a certain value of $l_2$ the inductive claim is as follows. For any $K > c(d)(n\log^2 n/\eta)^{(d-1)/2}$,
	\begin{align}
	\abss{\E\left[ \sum_{u_1,u_2,\ldots,u_{l+1}} \a_{u_1 u_2 \ldots u_{l+1}} \prod_{b=1}^{2} \prod_{e=1}^{l_b} X_{t^*,u_{s_b+e}}^{(b)} \middle| X_{t_0}^{(j)} = x_{t_0}^{(j)} \: \text{for } j=1,2 \right]} \le (l_2+1) K^{(l+1)/(d-1)}. \label{eq:he1}
	\end{align}
	As a base case consider the scenario when $l_2 = 0$. Then the function under consideration is a vanilla non-hybrid $l+1$-linear function from $F_a^d(l+1)$ and the statement holds by the definition of $G_K^{a,d}(t_0)$. Suppose the statement holds for some $l_2$. We will show that it holds for any $l+1$-linear hybrid function in $F_a^d(l+1)$ with $l_2+1$ terms coming from the 2nd run. Then the LHS of Statement~\ref{stmt:hybrid-cond-expectation} is of the form,
	\begin{align}
	&\abss{\E\left[ \sum_{u_1,u_2,\ldots,u_{l+1}} \a_{u_1 u_2 \ldots u_{l+1}} \prod_{e=1}^{l_1-1} X_{t^*,u_{e}}^{(1)} \prod_{e=l_1}^{l+1} X_{t^*,u_{e}}^{(2)} \middle| X_{t_0}^{(j)}=x_{t_0}^{(j)} \: \text{for } j=1,2 \right]} \notag\\
	&~ \le  \abss{\E\left[ \sum_{u_1,u_2,\ldots,u_{l+1}} \a_{u_1 u_2 \ldots u_{l+1}} \prod_{e=1}^{l_1-1} X_{t^*,u_{e}}^{(1)} \prod_{e=l_1}^{l} X_{t^*,u_{e}}^{(2)} X_{t^*,u_{l+1}}^{(1)}\middle| X_{t_0}^{(j)}=x_{t_0}^{(j)} \: \text{for } \: 1 \le j \le 2 \right]} \label{eq:he3}\\
	&~~+ \abss{\E\left[ \sum_{u_{l+1}} \left(X_{t^*,u_{l+1}}^{(2)} - X_{t^*,u_{l+1}}^{(1)} \right)\sum_{u_1,u_2,\ldots,u_{l}} \a_{u_1 u_2 \ldots u_l u_{l+1}} \prod_{e=1}^{l_1-1} X_{t^*,u_{e}}^{(1)} \prod_{e=l_1}^{l} X_{t^*,u_{e}}^{(2)} \middle| X_{t_0}^{(j)}=x_{t_0}^{(j)} \: \text{for } j=1,2 \right]}. \label{eq:he4}
	\end{align}
	
	We have, by the inductive hypothesis, Statement~\ref{stmt:hybrid-cond-expectation} for functions in $F_a^d(l+1)$ with $l_2$ terms from the second run that,
	\begin{align}
	(\ref{eq:he3}) \le l_2 K^{(l+1)/(d-1)} . \label{eq:he5}
	\end{align}
	Similarly, from the inductive hypothesis for functions in $F_a^d(l)$, Statements~\ref{stmt:hybrid-cond-expectation}, \ref{stmt:hybrid-cond-concentration} and Corollary~\ref{cor:goodset-layering}, we have
	\begin{align}
	&\abss{\E\left[\sum_{u_1,u_2,\ldots,u_{l}} \a_{u_1 u_2 \ldots u_l u_{l+1}} \prod_{e=1}^{l_1-1} X_{t^*,u_{e}}^{(1)} \prod_{e=l_1}^{l} X_{t^*,u_{e}}^{(2)} \middle| X_{t_0}^{(j)} = x_{t_0}^{(j)} \: \text{for } j=1,2 \right]} \le (l+1)K^{l/(d-1)}, \label{eq:he6}\\
	&\Pr\left[ \left\lvert \sum_{u_1,u_2,\ldots,u_{l}} \a_{u_1 u_2 \ldots u_l u_{l+1}} \prod_{e=1}^{l_1-1} X_{t^*,u_{e}}^{(1)} \prod_{e=l_1}^{l} X_{t^*,u_{e}}^{(2)} - \right. \right. \notag \\
	&~~~\left. \left. \E\left[\sum_{u_1,u_2,\ldots,u_{l}} \a_{u_1 u_2 \ldots u_l u_{l+1}} \prod_{e=1}^{l_1-1} X_{t^*,u_{e}}^{(1)} \prod_{e=l_1}^{l} X_{t^*,u_{e}}^{(2)} \middle| X_i^{(j)} \: \text{for } j=1,2 \right]\right\rvert > K^{l/(d-1)} \middle| X_{t_0}^{(j)} = x_{t_0}^{(j)} \: \text{for } j=1,2\right] \notag \\
	&~~~~~ \le 2\exp\left(-\frac{K^{2/(d-1)}}{c_6(d)t^*}\right). \label{eq:he7}
	\end{align}
	%where we could apply the induction hypothesis because $x_{t_0}^{(j)} \in G_K^{a^{u_{l+1}},l}(t_0)$ from Corollary \ref{cor:goodset-layering}.
	Putting together (\ref{eq:he6}) and (\ref{eq:he7}) with the Hamming contraction properties for any pair of runs in the coupled dynamics and the fact that $(\ref{eq:he4}) \le 2n^{l+1}$ always (similar to how it was shown in Section \ref{sec:bilinear}), we get
	\begin{align}
	(\ref{eq:he4}) \le (l+1)K^{l/(d-1)} + 4n^{l+1}\exp\left(-\frac{K^{2/(d-1)}}{c_6(d)t^*}\right) \le K^{(l+1)/(d-1)} \label{eq:he8}
	\end{align}
	where the last inequality holds because $K > c(d)(n\log^2 n/\eta)^{(d-1)/2}$.
	(\ref{eq:he5}) and (\ref{eq:he8}) together imply the desired bound for the case $l_2+1$. 
	Hence this proves Statement~\ref{stmt:hybrid-cond-expectation} for $l+1$-linear functions.\\

	%-------------------------
	% Stmt hybrid-cond-concentration
	%-------------------------
	\noindent \textbf{Induction Step for Statement~\ref{stmt:hybrid-cond-concentration}:}
	Next we look at Statement~\ref{stmt:hybrid-cond-concentration} for $l+1$-linear hybrid functions in $F_a^d(l+1)$. The high level approach is similar to that used above in the proof for Statement~\ref{stmt:hybrid-cond-expectation}. We will try to express any hybrid function of degree $l+1$ as a non-hybrid function of degree $l+1$ plus functions which resemble hybrid functions of degree $l$ multiplied with the Hamming distance between the two runs at time $t^*$. To bound the probability of deviation of the non-hybrid function we use the definition of the `good' set and to bound the probability of deviation of the other functions we appeal to the induction hypothesis and Hamming contraction properties. We incur an additional factor which is $\exp(d)$ in the final bound (recall that we are treating $d$ as fixed).
	Consider any function $f_a^{v_1,v_2,\ldots,v_{d-l-1}}\left( x\right)$ from the family $F_a^d(l+1)$. Denote its coefficient vector by $\a$. That is, $\a_{u_1 u_2 \ldots u_l u_{l+1}} = a_{u_1 u_2 \ldots u_{l+1}v_1 v_2 \ldots v_{d-l-1}}$.
	We again induct on $l_2$, the number of terms in the function corresponding to the second run of the Glauber dynamics.
	Given an $l+1$-linear $f \in F_a^d(l+1)$ with coefficient vector $\a$, the inductive claim for the class of hybrid functions associated with $f$ with a certain value of $l_2$ is as follows:
	\begin{align}
	&\Pr\left[\abss{f\left(X_{t^*}^{(1:2)}\right) - \E\left[ f\left(X_{t^*}^{(1:2)}\right) \middle| X_{t_0}^{(j)} \: \forall \: 1 \le j \le 2 \right]} > K^{(l+1)/(d-1)}\middle| X_{t_0}^{(j)}=x_{t_0}^{(j)} \: \forall \: 1 \le j \le 2 \right] \notag \\
	&~~~~~~\le 2\exp\left(-\frac{K^{2/(d-1)}}{c_8(l+1+l_2)t^*}\right), \label{eq:hcc1}
	\end{align}
	where $c_8(l)$ is an increasing function of $l$.
	As a base case consider the scenario when $l_2 = 0$. Then the function under consideration is a vanilla non-hybrid $l+1$-linear function from $F_a^d(l+1)$ and the statement holds by the definition of $G_K^{a,d}(t_0)$. Suppose the statement holds for some $l_2 >0$. We will show that it holds for any $l+1$-linear function from $F_a^d(l+1)$ with number of terms corresponding to the second run equal to $l_2+1$. Consider the LHS of the statement for any such function:
	\small
	\begin{align}
	&\Pr\left[\left\lvert \sum_{u_1,u_2,\ldots,u_{l+1}} \a_{u_1 u_2 \ldots u_{l+1}} \prod_{e=1}^{l_1-1} X_{t^*,u_{e}}^{(1)} \prod_{e=l_1}^{l+1} X_{t^*,u_{e}}^{(2)} \right. \right. \notag\\
	&~~~~~ \left. \left. - \E\left[ \sum_{u_1,u_2,\ldots,u_{l+1}} \a_{u_1 u_2 \ldots u_{l+1}} \prod_{e=1}^{l_1-1} X_{t^*,u_{e}}^{(1)} \prod_{e=l_1}^{l+1} X_{t^*,u_{e}}^{(2)} \middle| X_{t_0}^{(j)} \: \forall \: 1 \le j \le 2 \right] \right\rvert > K^{(l+1)/(d-1)} \middle| X_{t_0}^{(j)}=x_{t_0}^{(j)} \: \forall \: 1 \le j \le 2 \right] \label{eq:hcc2}\\
	&\le \Pr\left[\left\lvert \sum_{u_1,u_2,\ldots,u_{l+1}} \a_{u_1 u_2 \ldots u_{l+1}} \prod_{e=1}^{l_1-1} X_{t^*,u_{e}}^{(1)} \prod_{e=l_1}^{l} X_{t^*,u_{e}}^{(2)}X_{t^*,u_{l+1}}^{(1)} \right. \right. \notag\\
	&~~~~\left. \left. - \E\left[\sum_{u_1,u_2,\ldots,u_{l+1}} \a_{u_1 u_2 \ldots u_{l+1}} \prod_{e=1}^{l_1-1} X_{t^*,u_{e}}^{(1)} \prod_{e=l_1}^{l} X_{t^*,u_{e}}^{(2)}X_{t^*,u_{l+1}}^{(1)} \middle| X_{t_0}^{(j)} \: \forall \: 1 \le j \le 2\right]\right\rvert \ge K^{(l+1)/(d-1)}/2 \middle| X_{t_0}^{(j)}=x_{t_0}^{(j)} \: \forall \: 1 \le j \le 2\right] \label{eq:hcc3}\\
	&+ \Pr\left[\left\lvert \sum_{u_{l+1}} \left(X_{t^*,u_{l+1}}^{(1)} - X_{t^*,u_{l+1}}^{(2)} \right)\sum_{u_1,u_2,\ldots,u_{l}} \a_{u_1 u_2 \ldots u_l u_{l+1}} \prod_{e=1}^{l_1-1} X_{t^*,u_{e}}^{(1)} \prod_{e=l_1}^{l} X_{t^*,u_{e}}^{(2)} \right. \right. \notag \\
	&\left. \left. - \E\left[\sum_{u_{l+1}} \left(X_{t^*,u_{l+1}}^{(1)} - X_{t^*,u_{l+1}}^{(2)} \right)\sum_{u_1,u_2,\ldots,u_{l}} \a_{u_1 u_2 \ldots u_l u_{l+1}} \prod_{e=1}^{l_1-1} X_{t^*,u_{e}}^{(1)} \prod_{e=l_1}^{l} X_{t^*,u_{e}}^{(2)}\right]\right\rvert > K^{(l+1)/(d-1)}/2 \middle| X_{t_0}^{(j)}=x_{t_0}^{(j)} \: \forall \: 1 \le j \le 2\right] \label{eq:hcc4}
	\end{align}
	\normalsize
	We have, by the inductive hypothesis for $l+1$-linear functions in $F_a^d(l+1)$ with $l_2$ terms from the 1st run that,
	\begin{align}
	(\ref{eq:hcc3}) \le 2\exp\left(-\frac{K^{2/(d-1)}}{c_8(l+1+l_2)t^*}\right)  . \label{eq:hcc5}
	\end{align}
	From the property of the `good' set $G_K^{a,d}(t_0)$ (Corollary \ref{cor:goodset-layering}), and the inductive hypothesis, Statements~\ref{stmt:hybrid-cond-expectation} and \ref{stmt:hybrid-cond-concentration} for functions in $F_a^d(l)$ we get,
	\small
	\begin{align}
	&\abss{\E\left[\sum_{u_1,u_2,\ldots,u_{l}} \a_{u_1 u_2 \ldots u_l u_{l+1}} \prod_{e=1}^{l_1-1} X_{t^*,u_{e}}^{(1)} \prod_{e=l_1}^{l} X_{t^*,u_{e}}^{(2)} \middle| X_{t_0}^{(j)}=x_{t_0}^{(j)} \: \forall \: 1 \le j \le 2\right]} \le (l+1)K^{l/(d-1)} \label{eq:hcc6}\\
	& \Pr\left[ \left\lvert \sum_{u_1,u_2,\ldots,u_{l}} \a_{u_1 u_2 \ldots u_l u_{l+1}} \prod_{e=1}^{l_1-1} X_{t^*,u_{e}}^{(1)} \prod_{e=l_1}^{l} X_{t^*,u_{e}}^{(2)} \right. \right. \notag\\
	&\left. \left. - \E\left[\sum_{u_1,u_2,\ldots,u_{l}} \a_{u_1 u_2 \ldots u_l u_{l+1}} \prod_{e=1}^{l_1-1} X_{t^*,u_{e}}^{(1)} \prod_{e=l_1}^{l} X_{t^*,u_{e}}^{(2)} \middle| X_{t_0}^{(j)} \: \forall \: 1 \le j \le 2\right]\right\rvert > K^{l/(d-1)}/2 \middle| X_{t_0}^{(j)}=x_{t_0}^{(j)} \: \forall \: 1 \le j \le 2  \right] \notag \\
	&~~~~~~~\le 2\exp\left(-\frac{K^{2/(d-1)}}{ 4^{1/l}c_8(l+l_2)t^*}\right) . \label{eq:hcc7}
	\end{align}
	\normalsize
	When $K > c(d)(n\log^2 n/\eta)^{(d-1)/2}$, (\ref{eq:hcc6}) and (\ref{eq:hcc7}) together with the Hamming contraction property of the coupled dynamics (Lemma \ref{lem:greedy-properties}) imply that
	\begin{align}
	&\E\left[\sum_{u_{l+1}} \left(X_{t^*,u_{l+1}}^{(1)} - X_{t^*,u_{l+1}}^{(2)} \right)\sum_{u_1,u_2,\ldots,u_{l}} \a_{u_1 u_2 \ldots u_{l+1}} \prod_{e=1}^{l_1-1} X_{t^*,u_{e}}^{(1)} \prod_{e=l_1}^{l} X_{t^*,u_{e}}^{(2)} \middle| X_{t_0}^{(j)}=x_{t_0}^{(j)} \: \forall \: 1 \le j \le 2\right] \le (l+1)K^{l/d} \label{eq:hcc8}\\
	&\implies (\ref{eq:hcc4}) \le \notag\\
	&\Pr\left[\left\lvert \sum_{u_{l+1}} \left(X_{t^*,u_{l+1}}^{(1)} - X_{t^*,u_{l+1}}^{(2)} \right)\sum_{u_1,u_2,\ldots,u_{l}} \a_{u_1 u_2 \ldots u_{l+1}} \prod_{e=1}^{l_1-1} X_{t^*,u_{e}}^{(1)} \prod_{e=l_1}^{l} X_{t^*,u_{e}}^{(2)} \right\rvert > K^{(l+1)/(d-1)}/4 \middle| X_{t_0}^{(j)}=x_{t_0}^{(j)} \: \forall \: 1 \le j \le 2\right]. \label{eq:hcc9}
	\end{align}
	Note this follows since $(l+1)K^{l/d} \leq K^{(l+1)/(d-1)}/2$ when $K > c(d)(n\log^2 n/\eta)^{(d-1)/2}$ and $c(d)$ is sufficiently large.
	
	From Lemma \ref{lem:hamming-concentration} we have, for any $K_1 > 2$,
	\begin{align}
	&\Pr\left[ \abss{\sum_{u_{l+1}} \left(X_{t^*,u_{l+1}}^{(2)} - X_{t^*,u_{l+1}}^{(1)}  \right) } > K_1 \middle| X_{t_0}^{(j)}=x_{t_0}^{(j)} \: \forall \: 1 \le j \le 2\right] \\
	&\le \Pr\left[\dh(X_{t^*}^{(2)}, X_{t^*}^{(1)}) > K_1/2\middle| X_{t_0}^{(j)}=x_{t_0}^{(j)} \: \forall \: 1 \le j \le 2\right] \\
	&\le \Pr\left[\abss{\dh(X_{t^*}^{(2)}, X_{t^*}^{(1)}) - \E\left[\dh(X_{t^*}^{(2)}, X_{t^*}^{(1)}) \middle| X_{t_0}^{(j)}=x_{t_0}^{(j)} \: \forall \: 1 \le j \le 2 \right]} > K_1/2-1\middle| X_{t_0}^{(j)}=x_{t_0}^{(j)} \: \forall \: 1 \le j \le 2\right] \notag\\
	&~~\le 2\exp\left(-\frac{(K_1-2)^2}{64(t^*-t_0)} \right) \le 2\exp\left(-\frac{K_1^2}{70t^*} \right)\label{eq:hcc10}
	\end{align}
	where (\ref{eq:hcc10}) follows because $t^*-t_0 \le t^*$, and $\E[\dh(X_{t^*}^{(1)}, X_{t^*}^{(2)}) | X_{t_0}^{(j)}=x_{t_0}^{(j)} \: \forall \: 1 \le j \le 2] \le 1$. Now, set $K_1 = K^{1/(d-1)}/2$.
	Applying (\ref{eq:hcc10}) for these parameter values, we get,
	\begin{align}
	&(\ref{eq:hcc9}) \le \Pr\left[\abss{\sum_{u_{l+1}} \left(X_{t^*,u_{l+1}}^{(1)} - X_{t^*,u_{l+1}}^{(2)} \right)} > K^{1/(d-1)}/2 \middle| X_{t_0}^{(j)}=x_{t_0}^{(j)} \: \forall \: 1 \le j \le 2\right] \\
	&~~~~+  \Pr\left[\exists u_{l+1} \: \text{s.t. }\abss{\sum_{u_1,u_2,\ldots,u_{l}} \a_{u_1 u_2 \ldots u_{l+1}} \prod_{e=1}^{l_1-1} X_{t^*,u_{e}}^{(1)} \prod_{e=l_1}^{l} X_{t^*,u_{e}}^{(2)}} > K^{l/(d-1)}/2 \middle| X_{t_0}^{(j)}=x_{t_0}^{(j)} \: \forall \: 1 \le j \le 2\right]\\
	&\le 2\exp\left(-\frac{K^{2/(d-1)}}{280t^*} \right) + n\exp\left(-\frac{K^{2/(d-1)}}{c_8(l+l_2)t^*}\right) \le (n+1)\exp\left(-\frac{K^{2/(d-1)}}{c_8(l+l_2)t^*}\right). \label{eq:hcc12}
	\end{align}
	(\ref{eq:hcc12}) with (\ref{eq:hcc5}) implies
	\begin{align}
	(\ref{eq:hcc2}) \le (n+1)\exp\left(-\frac{K^{2/(d-1)}}{c_8(l+l_2)t^*}\right) + 2\exp\left(-\frac{K^{2/(d-1)}}{c_8(l+1+l_2)t^*}\right) \le 2\exp\left(-\frac{K^{2/(d-1)}}{c_8(l+2+l_2)t^*}\right),
	\end{align}
	for sufficiently large $c_8(.)$.
	This gives the desired bound for $l+1$-linear functions in $F_a^d(l+1)$ with $l_2+1$ terms from the second run. By induction, this proves Statement~\ref{stmt:hybrid-cond-concentration} for all $1 \le l \le d-1$ and all functions from $F_a^d(l)$.
	
\end{proof}

	\section{Experiments}
\label{sec:experiments}
In this section, we apply our family of bilinear statistics on the Ising model to a problem of statistical hypothesis testing.
Given a single sample from a multivariate distribution, we attempt to determine whether or not this sample was generated from an Ising model in the high-temperature regime.
More specifically, the null hypothesis is that the sample is drawn from an Ising model with a known graph structure with a common edge parameter and a uniform node parameter (which may potentially be known to be 0). 
In Section~\ref{sec:synthetic}, we apply our statistics to synthetic data.
In Section~\ref{sec:music}, we turn our attention to the Last.fm dataset from HetRec 2011 \cite{CantadorBK11}.

The running theme of our experimental investigation is testing the classical and common assumption which models choices in social networks as an Ising model~\cite{Ellison93, MontanariS10}.
To be more concrete, choices in a network could include whether to buy an iPhone or an Android phone, or whether to vote for a Republican or Democratic candidate.
Such choices are naturally influenced by one's neighbors in the network -- one may be more likely to buy an iPhone if he sees all his friends have one, corresponding to an Ising model with positive-weight edges\footnote{Note that one may also decide \emph{against} buying an iPhone in this scenario, if one places high value on individuality and uniqueness -- this corresponds to negative-weight edges.}
In our synthetic data study, we will leave these choices as abstract, referring to them only as ``values,'' but in our Last.fm data study, these choices will be whether or not one listens to a particular artist.

Our general algorithmic approach is as follows.
Given a single multivariate sample, we first run the maximum pseudo-likelihood estimator (MPLE) to obtain an estimate of the model's parameters.
The MPLE is a canonical estimator for the parameters of the Ising model, and it enjoys strong consistency guarantees in many settings of interest~\cite{Chatterjee07, BhattacharyaM16}.
If the MPLE gives a large estimate of the model's edge parameter, this is sufficient evidence to reject the null hypothesis.
Otherwise, we use Markov Chain Monte Carlo (MCMC) on a model with the MPLE parameters to determine a range of values for our statistic.
We note that, to be precise, we would need to quantify the error incurred by the MPLE -- in favor of simplicity in our exploratory investigation, we eschew this detail, and at this point attempt to reject the null hypothesis of the model learned by the MPLE.
Our statistic is bilinear in the Ising model, and thus enjoys the strong concentration properties explained earlier in this paper.
Note that since the Ising model will be in the high-temperature regime, the Glauber dynamics mix rapidly, and we can efficiently sample from the model using MCMC.
Finally, given the range of values for the statistic determined by MCMC, we reject the null hypothesis if $p \leq 0.05$.

\subsection{Synthetic Data}
\label{sec:synthetic}
We proceed with our investigation on synthetic data.
Our null hypothesis is that the sample is generated from an Ising model in the high temperature regime on the grid, with no external field (i.e. $\theta_u = 0$ for all $u$) and a common (unknown) edge parameter $\theta$ (i.e., $\theta_{uv} = \theta$ iff nodes $u$ and $v$ are adjacent in the grid, and $0$ otherwise).
For the Ising model on the grid, the critical edge parameter for high-temperature is $\theta_c = \frac{\ln(1 + \sqrt{2})}{2}$.
In other words, we are in high-temperature if and only if $\theta \leq \theta_c$, and we can reject the null hypothesis if the MPLE estimate $\hat \theta > \theta_c$.

To generate departures from the null hypothesis, we give a construction parameterized by $\tau \in [0, 1]$.
We provide a rough description of the departures, for a precise description, see the supplemental material.
Each node $x$ selects a random node $y$ at Manhattan distance at most $2$, and sets $y$'s value to $x$ with probability $\tau$.
The intuition behind this construction is that each individual selects a friend or a friend-of-a-friend, and tries to convince them to take his value -- he is successful with probability $\tau$.
Selecting either a friend or a friend-of-a-friend is in line with the concept of strong triadic closure~\cite{EasleyK10} from the social sciences, which suggests that two individuals with a mutual friend are likely to either already be friends (which the social network may not have knowledge of) or become friends in the future.

An example of a sample generated from this distribution with $\tau = 0.04$ is provided in Figure \ref{fig:models}, alongside a sample from the Ising model generated with the corresponding MPLE parameters.
We consider this distribution to pass the ``eye test'' -- one can not easily distinguish these two distributions by simply glancing at them.
However, as we will see, our multilinear statistic is able to correctly reject the null a large fraction of the time.

\begin{figure}
    \centering
    \begin{subfigure}[b]{0.3\textwidth}
        \includegraphics[width=\textwidth]{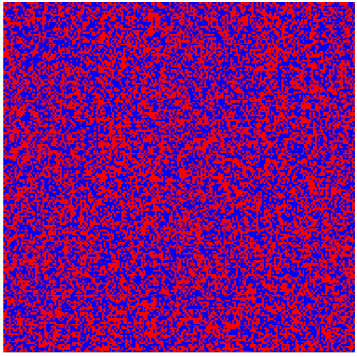}
        \caption{Our deviation from the null with $\tau = 0.04$}
    \end{subfigure}
~~~~~~~~~~~~~~~~~~~~~~~~~~~~~~
    \begin{subfigure}[b]{0.3\textwidth}
        \includegraphics[width=\textwidth]{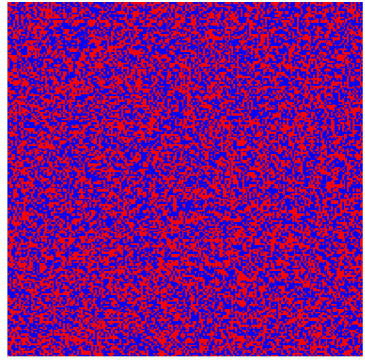}
        \caption{A sample from the Ising model with $\theta = 0.035$}
    \end{subfigure}
    \caption{A visual comparison between the null and a deviation from the null}\label{fig:models}
\end{figure}

Our experimental process was as follows.
We started with a $40 \times 40$ grid, corresponding to a distribution with $n = 1600$ dimensions.
We generated values for this grid according to the depatures from the null described above, with some parameter $\tau$.
We then ran the MPLE estimator to obtain an estimate for the edge parameter $\hat \theta$, immediately rejecting the null if $\hat \theta > \theta_c$.
Otherwise, we ran the Glauber dynamics for $O(n \log n)$ steps to generate a sample from the grid Ising model with parameter $\hat \theta$.
We repeated this process to generate $100$ samples, and for each sample, computed the value of the statistic
$Z_{local} = \sum_{u = (i,j)} \sum_{v = (k,l) : d(u,v) \leq 2} X_u X_v,$
where $d(\cdot, \cdot)$ is the Manhattan distance on the grid.
This statistic can be justified since we wish to account for the possibility of connections between friends-of-friends of which the social network may be lacking knowledge.
We then compare with the value of the statistic $Z_{local}$ on the provided sample, and reject the null hypothesis if this statistic corresponds to a $p$-value of $\leq 0.05$.
We repeat this for a wide range of values of $\tau \in [0, 1]$, and repeat $500$ times for each $\tau$.

Our results are displayed in Figure~\ref{fig:synthetic}
The x-axis marks the value of parameter $\tau$, and the y-axis indicates the fraction of repetitions in which we successfully rejected the null hypothesis.
The performance of the MPLE alone is indicated by the orange line, while the performance of our statistic is indicated by the blue line.
We find that our statistic is able to correctly reject the null at a much earlier point than the MPLE alone.
In particular, our statistic manages to reject the null for $\tau \geq 0.04$, while the MPLE requires a parameter which is an order of magnitude larger, at $0.4$.
As mentioned before, in the former regime (when $\tau \approx 0.04$), it appears impossible to distinguish the distribution from a sample from the Ising model with the naked eye. 

\begin{figure}
  \centering
    \includegraphics[width=0.5\textwidth]{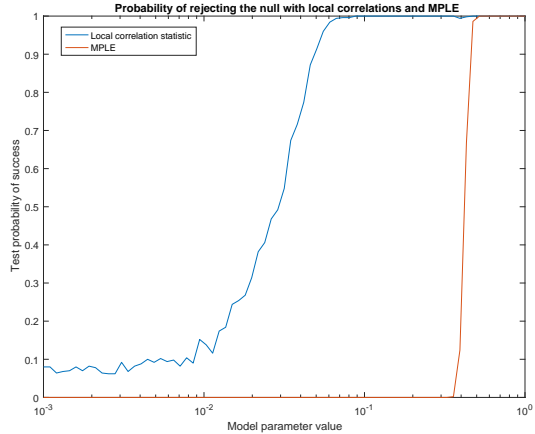}
  \caption{Power of our statistic on synthetic data.}
\label{fig:synthetic}
\end{figure}

\subsection{Last.fm Dataset}
\label{sec:music}
We now turn our focus to the Last.fm dataset from HetRec'11 \cite{CantadorBK11}.
This dataset consists of data from $n = 1892$ users on the Last.fm online music system.
On Last.fm, users can indicate (bi-directional) friend relationships, thus constructing a social network -- our dataset has $m = 12717$ such edges.
The dataset also contains users' listening habits -- for each user we have a list of their fifty favorite artists, whose tracks they have listened to the most times.
We wish to test whether users' preference for a particular artist is distributed according to a high-temperature Ising model.

Fixing some artist $a$ of interest, we consider the vector $X^{(a)}$, where $X_u^{(a)}$ is $+1$ if user $u$ has artist $a$ in his favorite artists, and $-1$ otherwise.
We wish to test the null hypothesis, whether $X^{(a)}$ is distributed according to an Ising model in the high temperature regime on the known social network graph, with common (unknown) external field $h$ (i.e. $\theta_u = h$ for all $u$) and edge parameter $\theta$ (i.e., $\theta_{uv} = \theta$ iff $u$ and $v$ are neighbors in the graph, and $0$ otherwise).

Our overall experimental process was very similar to the synthetic data case.
We gathered a list of the ten most-common favorite artists, and repeated the following process for each artist $a$.
We consider the vector $X^{(a)}$ (defined above) and run the MPLE estimator on it, obtaining estimates $\hat h$ and $\hat \theta$.
We then run MCMC to generate $100$ samples from the Ising model with these parameters, and for each sample, computed the value of the statistics
$Z_{k} = \sum_u \sum_{v : d(u,v) \leq k} (X_u - \tanh(\hat h)) (X_v - \tanh(\hat h)),$
where $d(\cdot, \cdot)$ is the distance on the graph, and $k = 1$ (the neighbor correlation statistic) or $2$ (the local correlation statistic).
Motivated by our theoretical results (Theorem~\ref{thm:bilinear-concentration-external}), we consider a statistic where the variables are recentered by their marginal expectations, as this statistic experiences sharper concentration.
We again consider $k = 2$ to account for the possibility of edges which are unknown to the social network.

Strikingly, we found that the plausibility of the Ising modelling assumption varies significantly depending on the artist.
We highlight some of our more interesting findings here, see the supplemental material for more details.
The most popular artist in the dataset was Lady Gaga, who was a favorite artist of $611$ users in the dataset.
We found that $X^{(\rm Lady\ Gaga)}$ had statistics $Z_1 = 9017.3$ and $Z_2 = 106540$.
The range of these statistics computed by MCMC can be seen in Figure~\ref{fig:gaga} -- clearly, the computed statistics fall far outside these ranges, and we can reject the null hypothesis with $p \ll 0.01$.
Similar results held for other popular pop musicians, including Britney Spears, Christina Aguilera, Rihanna, and Katy Perry.

\begin{figure}
    \centering
    \begin{subfigure}[b]{0.45\textwidth}
        \includegraphics[width=\textwidth]{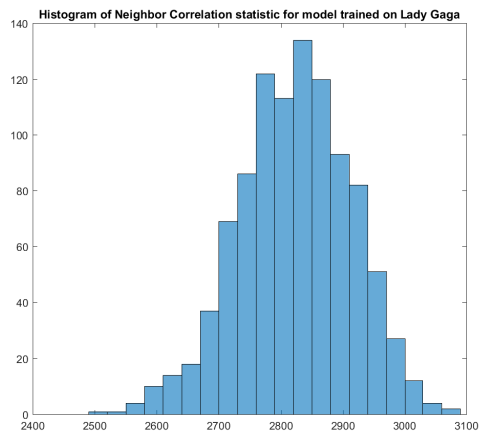}
        \caption{MCMC $Z_1$ for Lady Gaga}
    \end{subfigure}
~~~~~~~~~~
    \begin{subfigure}[b]{0.45\textwidth}
        \includegraphics[width=\textwidth]{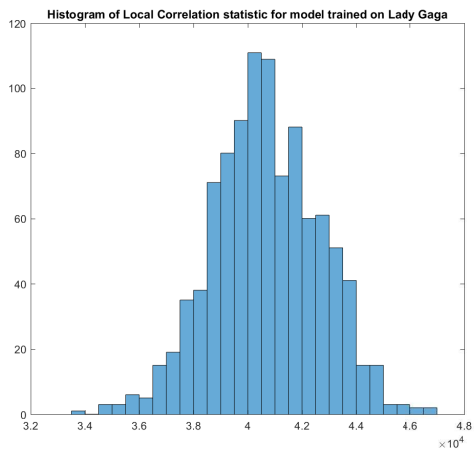}
        \caption{MCMC $Z_2$ for Lady Gaga}
    \end{subfigure}
    \caption{MCMC Statistics for Lady Gaga}\label{fig:gaga}
\end{figure}

However, we observed qualitatively different results for The Beatles, the fourth most popular artist, being a favorite of $480$ users.
We found that $X^{(\rm The\ Beatles)}$ had statistics $Z_1 = 2157.8$ and $Z_2 = 22196$.
The range of these statistics computed by MCMC can be seen in Figure~\ref{fig:beatles} of the supplementary material.
This time, the computed statistics fall near the center of this range, and we can not reject the null.
Similar results held for the rock band Muse.

\begin{figure}
    \centering
    \begin{subfigure}[b]{0.45\textwidth}
        \includegraphics[width=\textwidth]{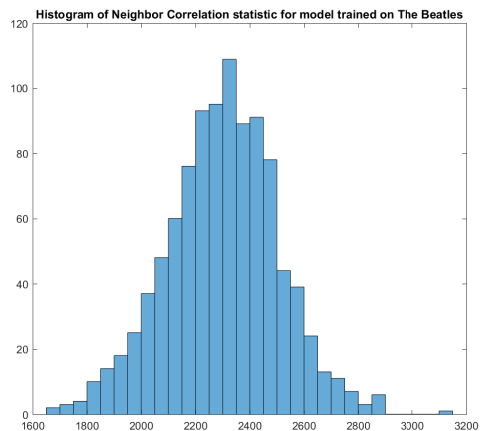}
        \caption{MCMC $Z_1$ for The Beatles}
    \end{subfigure}
~~~~~~~~~~
    \begin{subfigure}[b]{0.45\textwidth}
        \includegraphics[width=\textwidth]{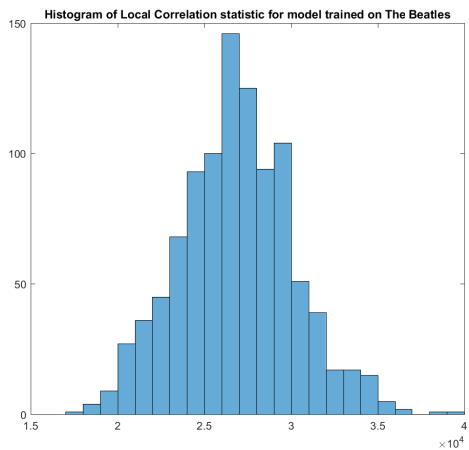}
        \caption{MCMC $Z_2$ for The Beatles}
    \end{subfigure}
    \caption{MCMC Statistics for The Beatles}\label{fig:beatles}
\end{figure}

Based on our investigation, our statistic seems to indicate that for the pop artists, the null fails to effectively model the distribution, while it performs much better for the rock artists.
We conjecture that this may be due to the highly divisive popularity of pop artists like Lady Gaga and Britney Spears -- while some users may love these artists (and may form dense cliques within the graph), others have little to no interest in their music.
The null would have to be expanded to accomodate heterogeneity to model such effects.
On the other hand, rock bands like The Beatles and Muse seem to be much more uniform in their appeal: users seem to be much more homogeneous when it comes to preference for these groups.

	\bibliographystyle{alpha}
	\bibliography{biblio}

\newcommand{\etalchar}[1]{$^{#1}$}
\newcommand{\noopsort}[1]{} \newcommand{\printfirst}[2]{#1}
  \newcommand{\singleletter}[1]{#1} \newcommand{\switchargs}[2]{#2#1}
\begin{thebibliography}{MdCCU16}

\bibitem[AKN06]{AbbeelKN06}
Pieter Abbeel, Daphne Koller, and Andrew~Y. Ng.
\newblock Learning factor graphs in polynomial time and sample complexity.
\newblock {\em Journal of Machine Learning Research}, 7(Aug):1743--1788, 2006.

\bibitem[BGS14]{BreslerGS14}
Guy Bresler, David Gamarnik, and Devavrat Shah.
\newblock Structure learning of antiferromagnetic {I}sing models.
\newblock In {\em Advances in Neural Information Processing Systems 27}, NIPS
  '14, pages 2852--2860. Curran Associates, Inc., 2014.

\bibitem[Bha16]{Bhattacharya16}
Bhaswar~B. Bhattacharya.
\newblock Power of graph-based two-sample tests.
\newblock {\em arXiv preprint arXiv:1508.07530}, 2016.

\bibitem[BK16]{BreslerK16}
Guy Bresler and Mina Karzand.
\newblock Learning a tree-structured {I}sing model in order to make
  predictions.
\newblock {\em arXiv preprint arXiv:1604.06749}, 2016.

\bibitem[BM16]{BhattacharyaM16}
Bhaswar~B. Bhattacharya and Sumit Mukherjee.
\newblock Inference in {I}sing models.
\newblock {\em Bernoulli}, 2016.

\bibitem[Bre15]{Bresler15}
Guy Bresler.
\newblock Efficiently learning ising models on arbitrary graphs.
\newblock In {\em Proceedings of the 47th Annual ACM Symposium on the Theory of
  Computing}, STOC '15, pages 771--782, New York, NY, USA, 2015. ACM.

\bibitem[CBK11]{CantadorBK11}
Iv\'{a}n Cantador, Peter Brusilovsky, and Tsvi Kuflik.
\newblock Second workshop on information heterogeneity and fusion in
  recommender systems (hetrec 2011).
\newblock In {\em Proceedings of the 5th ACM Conference on Recommender
  Systems}, RecSys '11, pages 387--388, New York, NY, USA, 2011. ACM.

\bibitem[Cha05]{Chatterjee05}
Sourav Chatterjee.
\newblock {\em Concentration Inequalities with Exchangeable Pairs}.
\newblock PhD thesis, Stanford University, June 2005.

\bibitem[Cha07]{Chatterjee07}
Sourav Chatterjee.
\newblock Estimation in spin glasses: A first step.
\newblock {\em The Annals of Statistics}, 35(5):1931--1946, October 2007.

\bibitem[CL68]{ChowL68}
C.K. Chow and C.N. Liu.
\newblock Approximating discrete probability distributions with dependence
  trees.
\newblock {\em IEEE Transactions on Information Theory}, 14(3):462--467, 1968.

\bibitem[CT06]{CsiszarT06}
Imre Csisz{\'a}r and Zsolt Talata.
\newblock Consistent estimation of the basic neighborhood of {M}arkov random
  fields.
\newblock {\em The Annals of Statistics}, 34(1):123--145, 2006.

\bibitem[DCG68]{DeserCG68}
Stanley Deser, Max Chr{\'e}tien, and Eugene Gross.
\newblock {\em Statistical Physics, Phase Transitions, and Superfluidity}.
\newblock Gordon and Breach, 1968.

\bibitem[DDK18]{DaskalakisDK18}
Constantinos Daskalakis, Nishanth Dikkala, and Gautam Kamath.
\newblock Testing {I}sing models.
\newblock In {\em Proceedings of the 29th Annual ACM-SIAM Symposium on Discrete
  Algorithms}, SODA '18, Philadelphia, PA, USA, 2018. SIAM.

\bibitem[DMR11]{DaskalakisMR11}
Constantinos Daskalakis, Elchanan Mossel, and S{\'e}bastien Roch.
\newblock Evolutionary trees and the {I}sing model on the {B}ethe lattice: A
  proof of {S}teel's conjecture.
\newblock {\em Probability Theory and Related Fields}, 149(1):149--189, 2011.

\bibitem[DS71]{DeWittS71}
C{\'e}cile DeWitt and Raymond Stora.
\newblock {\em Statistical mechanics and Quantum field theory}.
\newblock Gordon and Breach, 1971.

\bibitem[EK10]{EasleyK10}
David Easley and Jon Kleinberg.
\newblock {\em Networks, Crowds, and Markets: Reasoning about a Highly
  Connected World}.
\newblock Cambridge University Press, 2010.

\bibitem[Ell93]{Ellison93}
Glenn Ellison.
\newblock Learning, local interaction, and coordination.
\newblock {\em Econometrica}, 61(5):1047--1071, 1993.

\bibitem[Fel04]{Felsenstein04}
Joseph Felsenstein.
\newblock {\em Inferring Phylogenies}.
\newblock Sinauer Associates Sunderland, 2004.

\bibitem[Fre75]{Freedman75}
David~A. Freedman.
\newblock On tail probabilities for martingales.
\newblock {\em The Annals of Probability}, 3(1):100--118, 1975.

\bibitem[GG86]{GemanG86}
Stuart Geman and Christine Graffigne.
\newblock {M}arkov random field image models and their applications to computer
  vision.
\newblock In {\em Proceedings of the International Congress of Mathematicians},
  pages 1496--1517. American Mathematical Society, 1986.

\bibitem[GLP17]{GheissariLP17}
Reza Gheissari, Eyal Lubetzky, and Yuval Peres.
\newblock Concentration inequalities for polynomials of contracting {I}sing
  models.
\newblock {\em arXiv preprint arXiv:1706.00121}, 2017.

\bibitem[HKM17]{HamiltonKM17}
Linus Hamilton, Frederic Koehler, and Ankur Moitra.
\newblock Information theoretic properties of {M}arkov random fields, and their
  algorithmic applications.
\newblock In {\em Advances in Neural Information Processing Systems 30}, NIPS
  '17. Curran Associates, Inc., 2017.

\bibitem[Hop82]{Hopfield82}
John~J. Hopfield.
\newblock Neural networks and physical systems with emergent collective
  computational abilities.
\newblock {\em Proceedings of the National Academy of Sciences},
  79(8):2554--2558, 1982.

\bibitem[Isi25]{Ising25}
Ernst Ising.
\newblock Beitrag zur theorie des ferromagnetismus.
\newblock {\em Zeitschrift f{\"u}r Physik A Hadrons and Nuclei},
  31(1):253--258, 1925.

\bibitem[JJR11]{JalaliJR11}
Ali Jalali, Christopher~C. Johnson, and Pradeep~K. Ravikumar.
\newblock On learning discrete graphical models using greedy methods.
\newblock In {\em Advances in Neural Information Processing Systems 24}, NIPS
  '11, pages 1935--1943. Curran Associates, Inc., 2011.

\bibitem[KM17]{KlivansM17}
Adam Klivans and Raghu Meka.
\newblock Learning graphical models using multiplicative weights.
\newblock In {\em Proceedings of the 58th Annual IEEE Symposium on Foundations
  of Computer Science}, FOCS '17, Washington, DC, USA, 2017. IEEE Computer
  Society.

\bibitem[LPW09]{LevinPW09}
David~A. Levin, Yuval Peres, and Elizabeth~L. Wilmer.
\newblock {\em {M}arkov Chains and Mixing Times}.
\newblock American Mathematical Society, 2009.

\bibitem[MdCCU16]{MartindelCampoCU16}
Abraham Mart{\'\i}n~del Campo, Sarah Cepeda, and Caroline Uhler.
\newblock Exact goodness-of-fit testing for the {I}sing model.
\newblock {\em Scandinavian Journal of Statistics}, 2016.

\bibitem[MJC{\etalchar{+}}14]{MackeyJCFT14}
Lester Mackey, Michael~I. Jordan, Richard~Y. Chen, Brendan Farrell, and Joel~A.
  Tropp.
\newblock Matrix concentration inequalities via the method of exchangeable
  pairs.
\newblock {\em The Annals of Probability}, 42(3):906--945, 2014.

\bibitem[MS10]{MontanariS10}
Andrea Montanari and Amin Saberi.
\newblock The spread of innovations in social networks.
\newblock {\em Proceedings of the National Academy of Sciences},
  107(47):20196--20201, 2010.

\bibitem[New90]{Newman90}
Charles~M. Newman.
\newblock {I}sing models and dependent percolation.
\newblock {\em Lecture Notes--Monograph Series}, 16:395--401, 1990.

\bibitem[Ons44]{Onsager44}
Lars Onsager.
\newblock Crystal statistics. {I}. a two-dimensional model with an
  order-disorder transition.
\newblock {\em Physical Review}, 65(3--4):117, 1944.

\bibitem[RAS15]{RassoulAghaS15}
Firas Rassoul-Agha and Timo Sepp{\"a}l{\"a}inen.
\newblock {\em A Course on Large Deviations with an Introduction to Gibbs
  Measures}.
\newblock American Mathematical Society, 2015.

\bibitem[RWL10]{RavikumarWL10}
Pradeep Ravikumar, Martin~J. Wainwright, and John~D. Lafferty.
\newblock High-dimensional ising model selection using $\ell_1$-regularized
  logistic regression.
\newblock {\em The Annals of Statistics}, 38(3):1287--1319, 2010.

\bibitem[SK75]{SherringtonK75}
David Sherrington and Scott Kirkpatrick.
\newblock Solvable model of a spin-glass.
\newblock {\em Physical Review Letters}, 35(26):1792, 1975.

\bibitem[SW12]{SanthanamW12}
Narayana~P. Santhanam and Martin~J. Wainwright.
\newblock Information-theoretic limits of selecting binary graphical models in
  high dimensions.
\newblock {\em IEEE Transactions on Information Theory}, 58(7):4117--4134,
  2012.

\bibitem[SZ92]{StroockZ92}
Daniel~W. Stroock and Boguslaw Zegarlinski.
\newblock The logarithmic {S}obolev inequality for discrete spin systems on a
  lattice.
\newblock {\em Communications in Mathematical Physics}, 149(1):175--193, 1992.

\bibitem[VMLC16]{VuffrayMLC16}
Marc Vuffray, Sidhant Misra, Andrey Lokhov, and Michael Chertkov.
\newblock Interaction screening: Efficient and sample-optimal learning of
  {I}sing models.
\newblock In {\em Advances in Neural Information Processing Systems 29}, NIPS
  '16, pages 2595--2603. Curran Associates, Inc., 2016.

\end{thebibliography}
	\appendix
	\section{Additional details about Experiments}
\subsection{Details about synthetic experiments}
Our departures from the null hypothesis are generated in the following manner, parameterized by some parameter $\tau \in [0, 1]$.
The grid is initialized by setting each node independently to be $-1$ or $1$ with equal probability.
We then iterate over the nodes in column major order.
For the node $x$ at position $(i,j)$, we select a node $y$ at one of the following positions uniformly at random: $(i, j+1), (i, j+2),  (i+1, j+1), (i+1, j), (i+2, j), (i+1, j-1)$.
Then, with probability $\tau$, we set $y$ to have the same value as $x$.
We imagine this construction as a type of social network model, where each individual tries to convert one of his nearby connections in the network to match his signal, and is successful with probability $\tau$.

\subsection{Details about experiments on Last.fm dataset}
We report additional statistics extracted from the Last.fm dataset \cite{CantadorBK11}.

The network has $n = 1892$ and $17632$ artists.
There are $m = 12717$ edges, with an average degree of $13.443$.
There are $92834$ user-listened artist relations, where the artists listed for a user is truncated at 50.
On average, $5.265$ users listened to each artist, but we focus on artists who had significatly more listens ($\sim 400$ or more).

\begin{tabular}{|c |c |c| c| c| c| c| c| c|}
\hline
Artist & \# of favorites & MPLE $h$ & MPLE $\theta$ & $Z_1$ & $Z_2$ & Reject $Z_1$? & Reject $Z_2$? \\ \hline
Lady Gaga & 611 & $-0.481$ & $0.0700$ & $9017.3$ & $106540$ & Yes & Yes \\ \hline
Britney Spears & 522 & $-0.6140$ & $0.0960$ & $10585$ & $119560$ & Yes & Yes \\ \hline
Rihanna & 484 & $-0.715$ & $0.1090$ & $11831$ & $126750$ & Yes & Yes \\ \hline
The Beatles & 480 & $-0.3550$ & $0.0310$ & $2157.8$ & $22196$ & No & No \\ \hline
Katy Perry & 473 & $-0.6150$ & $0.0890$ & $8474$ & $90762$ & Yes & Yes \\ \hline
Madonna & 429 & $-0.5400$ & $0.0860$ & $4580.9$ & $40395$ & Yes & No \\ \hline
Avril Lavigne & 417 & $-0.5580$ & $0.1020$ & $5145.9$ & $48639$ & Yes & Yes \\ \hline
Christina Aguilera & 407 & $-0.7810$ & $0.1060$ & $9979.8$ & $101210$ & Yes & Yes \\ \hline
Muse & 400 & $-0.5430$ & $0.0160$ & $923.55$ & $6911$ & No & No \\ \hline
Paramore & 399 & $-0.4530$ & $0.0480$ & $2047.1$ & $18119$ & Yes & Yes \\ \hline
\end{tabular}

\end{document}